\documentclass{lms}

\usepackage{amsmath,amssymb,theoremref,float}
\usepackage[shortlabels]{enumitem}
\usepackage{xypic} %for diagrams
\usepackage{mathrsfs} %for script fonts

\newtheorem{thm}{Theorem}[section] % 1st argument is your name for it
\newtheorem{lem}[thm]{Lemma}     % 2nd argument is what is printed
\newtheorem{cor}[thm]{Corollary}
\newtheorem{prop}[thm]{Proposition}

\newnumbered{assertion}[thm]{Assertion}    % 1st argument is your name for it
\newnumbered{conjecture}[thm]{Conjecture}  % 2nd argument is what is printed
\newnumbered{defn}[thm]{Definition}
\newnumbered{hypothesis}[thm]{Hypothesis}
\newnumbered{rmk}[thm]{Remark}
\newnumbered{note}[thm]{Note}
\newnumbered{obs}[thm]{Observation}
\newnumbered{problem}[thm]{Problem}
\newnumbered{quest}[thm]{Question}
\newnumbered{algorithm}[thm]{Algorithm}
\newnumbered{example}[thm]{Example}
\newunnumbered{notation}[thm]{Notation} % This is usually unnumbered
\title[Elementary t.d.l.c. groups ]% end with percent
 {Elementary totally disconnected locally compact groups} % This is the full title of the paper
\author{Phillip Wesolek}

\classno{22D05 (primary), 03E15 (secondary)}

\newcommand{\Zb}{\mathbb{Z}}
\newcommand{\Nb}{\mathbb{N}}

\renewcommand{\Uc}{\mathcal{U}}

\newcommand{\Qp}{\mathbb{Q}_p}

\newcommand{\mc}[1]{\mathcal{#1}}

\newcommand{\ms}[1]{\mathscr{#1}} %requires mathrsfs

\newcommand{\acts}{\curvearrowright}
\newcommand{\sleq}{\leqslant}
\newcommand{\sgeq}{\geqslant}
\newcommand{\cc}{\trianglelefteq_{cc}}

\newcommand{\Es}{\mathscr{E}}
\newcommand{\wbaire}{\mathbb{N}^{<\mathbb{N}}}
\newcommand{\baire}{\mathbb{N}^{\mathbb{N}}}
\newcommand{\rest}{\upharpoonright}
\newcommand{\conc}{^{\smallfrown}}

\newcommand{\cgrp}[1]{\overline{\left\langle #1 \right\rangle}}
\newcommand{\ngrp}[1]{\left\langle \left\langle #1 \right\rangle \right\rangle}
\newcommand{\grp}[1]{\left\langle #1 \right\rangle}
\newcommand{\ol}[1]{\overline{#1}}

\newcommand{\qci}[2]{QC_{#1}(#2)}
\renewcommand{\qc}[2]{\overline{QC}_{#1}(#2)}

\newcommand{\Tsumw}[2]{\bigoplus_{n\in \omega}(#1,#2)}

\newcommand{\Tsumz}[2]{\bigoplus_{i\in \mathbb{Z}}(#1,#2)}

\newcommand{\Rad}[1]{\mathop{\rm Rad}_{#1}\nolimits}
\newcommand{\Res}[1]{\mathop{\rm Res}_{#1}\nolimits}
\newcommand{\SIN}{\mathop{\rm SIN}}
\newcommand{\rk}{\mathop{\rm rk}}

\begin{document}
\maketitle
\begin{abstract}
We identify the class of elementary groups: the smallest class of totally disconnected locally compact second countable (t.d.l.c.s.c.) groups that contains the profinite groups and the discrete groups, is closed under group extensions of profinite groups and discrete groups, and is closed under countable increasing unions. We show this class enjoys robust permanence properties. In particular, it is closed under group extension, taking closed subgroups, taking Hausdorff quotients, and inverse limits. A characterization of elementary groups in terms of well-founded descriptive-set-theoretic trees is then presented. We conclude with three applications. We first prove structure results for general t.d.l.c.s.c. groups. In particular, we show a compactly generated t.d.l.c.s.c. group decomposes into elementary groups and topologically characteristically simple groups via group extension. We then prove two local-to-global structure theorems: Locally solvable t.d.l.c.s.c. groups are elementary and $[A]$-regular t.d.l.c.s.c. groups are elementary.
\end{abstract}

\section{Introduction}
We study totally disconnected locally compact (t.d.l.c.) groups that are also second countable (s.c.). T.d.l.c.s.c. groups are members of the natural and robust class of Polish groups, i.e. topological groups that are separable and completely metrizable; cf. \cite[(5.3)]{K95}. Polish spaces and, therefore, Polish groups are considered to be the ``correct" topological spaces and topological groups to study. That is to say, they capture almost all naturally occurring examples, admit useful and general theorems such as the Baire category theorem, and exclude pathological examples. The second countability assumption is also quite mild. Every t.d.l.c. group is a directed union of open $K_{\sigma}$ subgroups that are second countable modulo a compact normal subgroup; cf. \cite[(8.7)]{HR79}. Our results may therefore be easily adapted to the t.d.l.c. setting.  \par

\indent In the study of t.d.l.c.s.c. groups, groups ``built" from profinite groups and discrete groups arise often. Certainly, there are a number of counterexamples built as such. Less obviously, the work of V.P. Platonov \cite{Plat66} demonstrates that a t.d.l.c.s.c. group such that every finite set generates a relatively compact subgroup may be written as a countable increasing union of profinite groups. More subtly still, a residually discrete t.d.l.c.s.c. group may be written as a countable increasing union of SIN groups by results of P-E. Caprace and N. Monod \cite{CM11}.\par

\indent T.d.l.c.s.c. groups built from profinite groups and discrete groups also seem to play an important role in the structure theory of t.d.l.c.s.c. groups. Indeed, \cite{GW01} shows the kernel of the adjoint representation of a $p$-adic Lie group is such a group. More generally, the work \cite{CM11} indicates groups built from profinite groups and discrete groups are the barrier to finding non-trivial minimal normal subgroups, which is an essential step in reducing a t.d.l.c.s.c. group into ``basic" groups.\par

\indent  The frequency of occurrence of t.d.l.c.s.c. groups built from profinite groups and discrete groups and the important role they appear to play in the structure of general t.d.l.c.s.c. groups encourage further study. We here initiate this study, make a number of contributions to the theory, and give a few applications.\par

\subsection{Statement of results}
\indent We consider a seemingly narrow class of t.d.l.c.s.c. groups built from profinite groups and discrete groups.

\begin{defn}
The class of \textbf{elementary groups} is the smallest class $\Es$ of t.d.l.c.s.c. groups such that
\begin{enumerate}[(i)]

\item $\Es$ contains all second countable profinite groups and countable discrete groups;

\item $\Es$ is closed under taking group extensions of second countable profinite groups and countable discrete groups. I.e. if $G$ is a t.d.l.c.s.c. group and $H\trianglelefteq G$ is a closed normal subgroup with $H\in \Es$ and $G/H$ either profinite or discrete, then $G\in \Es$; and

\item If $G$ is a t.d.l.c.s.c. group and $G=\bigcup_{i\in \omega}O_i$ where $(O_i)_{i\in \omega}$ is an $\subseteq$-increasing sequence of open subgroups of $G$ with $O_i\in\Es$ for each $i$, then $G\in\Es$. We say $\Es$ is \textbf{closed under countable increasing unions}.
\end{enumerate}
\end{defn}

\begin{rmk} 
The class $\Es$ could more accurately but less efficiently be called the class of \textit{topologically elementary groups}. We omit ``topologically" as it is redundant in light of the fact we are interested in the non-compact topological theory. We apologize for reducing the interesting and complicated classes of discrete groups and profinite groups to the status of ``elementary" groups. Of course, such definitional slights are common in mathematics; e.g. the rich and deep theory of finite groups concerns virtually trivial groups.
\end{rmk}

Our first results show $\Es$ satisfies strong permanence properties indicating this class indeed captures our intuitive idea of groups ``built" from profinite groups and discrete groups.

\begin{thm} $\Es$ enjoys the following permanence properties:
\begin{enumerate}[(1)]
\item $\Es$ is closed under group extension. I.e. if $G$ is a t.d.l.c.s.c. group and $H\trianglelefteq G$ is a closed normal subgroup such that $H,G/H\in \Es$, then $G\in \Es$.
\item If $G\in \Es$, $H$ is a t.d.l.c.s.c. group, and $\psi:H\rightarrow G$ is a continuous, injective homomorphism, then $H\in \Es$. In particular, $\Es$ is closed under taking closed subgroups.
\item $\Es$ is closed under taking quotients by closed normal subgroups.
\item If $G$ is a residually elementary t.d.l.c.s.c. group, then $G\in \Es$. In particular, $\Es$ is closed under inverse limits that result in a t.d.l.c.s.c. group.
\item $\Es$ is closed under quasi-products.
\item $\Es$ is closed under local direct products.
\item If $G$ is a t.d.l.c.s.c. group and $(C_i)_{i\in \omega}$ is an $\subseteq$-increasing sequence of elementary closed subgroups of $G$ such that $N_G(C_i)$ is open for each $i$ and $\ol{\bigcup_{i\in \omega}C_i}=G$, then $G\in \Es$. 
\end{enumerate}
\end{thm}

We next characterize elementary groups in terms of well-founded descriptive-set-theoretic trees. This characterization provides not only a tool to identify elementary groups in nature but also a new rank on elementary groups. Using this new rank, we obtain an additional permanence property.

\begin{thm}
Suppose $H\in \Es$, $G$ is a t.d.l.c.s.c. group, and $\psi:H\rightarrow G$ is a continuous, injective homomorphism. If $\psi(H)$ is normal and dense in $G$, then $G\in \Es$.
\end{thm}

Our investigations conclude with three applications. The first application is to the general structure theory of t.d.l.c.s.c. groups. We begin by isolating two canonical normal subgroups.

\begin{thm} Let $G$ be a t.d.l.c.s.c. group. Then
\begin{enumerate}[(1)]
\item There is a unique maximal closed normal subgroup $\Rad{\Es}(G)$ such that $\Rad{\Es}(G)$ is elementary. 
\item There is a unique minimal closed normal subgroup $\Res{\Es}(G)$ such that $G/\Res{\Es}(G)$ is elementary.
\end{enumerate}
\end{thm}
\noindent We call $\Rad{\Es}(G)$ and $\Res{\Es}(G)$ the \textbf{elementary radical} and \textbf{elementary residual}, respectively. \par

\indent Using the elementary radical, a result of Caprace and Monod, and a clever technique of V.I. Trofimov, we arrive at a compelling structure result for compactly generated t.d.l.c.s.c. groups.

\begin{thm}\label{thm:cg_dcomp_intro}
Let $G$ be a compactly generated t.d.l.c.s.c. group. Then there is a finite series of closed characteristic subgroups 
\[
\{1\}= H_0\sleq H_1\sleq \dots \sleq H_{n}\sleq G
\]
such that 
\begin{enumerate}[(1)]
\item $G/H_{n}\in \Es$, and 
\item for $0\sleq k \sleq n-1$, $(H_{k+1}/H_{k})/\Rad{\Es}(H_{k+1}/H_{k})$ is a quasi-product with $0<n_{k+1}<\infty$ many topologically characteristically simple non-elementary quasi-factors.
\end{enumerate}
\end{thm}

Theorem~\rm\ref{thm:cg_dcomp_intro} seems to be convincing evidence that elementary groups are fundamental building blocks of t.d.l.c.s.c. groups. Indeed, Theorem~\rm\ref{thm:cg_dcomp_intro} demonstrates that the study of compactly generated t.d.l.c.s.c. groups reduces to the study of elementary groups and topologically characteristically simple non-elementary groups. \par

\indent Since there is interest in compactly generated t.d.l.c. groups, we note a corollary that relaxes the second countability hypothesis.

\begin{cor}
Let $G$ be a compactly generated t.d.l.c. group. Then there is a finite series of closed normal subgroups 
\[
\{1\}\sleq H_0\sleq H_1\sleq \dots \sleq H_{n}\sleq G
\]
such that 
\begin{enumerate}[(1)]
\item $H_0$ is compact and $G/H_0$ is second countable,
\item $G/H_{n}\in \Es$, and
\item for $0\sleq k \sleq n-1$, $(H_{k+1}/H_{k})/\Rad{\Es}(H_{k+1}/H_{k})$ is a quasi-product with $0<n_{k+1}<\infty$ many topologically characteristically simple non-elementary quasi-factors.
\end{enumerate}
\end{cor}

\begin{rmk} 
We provide examples showing $n$ can be arbitrarily large and the topologically characteristically simple quasi-factors cannot be taken to be topologically simple.
\end{rmk}

In the non-compactly generated case, we obtain a much weaker, nevertheless quite useful structure result; cf. \cite{W_2_14}.

\begin{defn}
A t.d.l.c.s.c. group is said to be \textbf{elementary-free} if it has no non-trivial elementary closed normal subgroups and no non-trivial elementary Hausdorff quotients. 
\end{defn}

\begin{thm}\label{thm:dcomp}
Let $G$ be a t.d.l.c.s.c. group. Then there is a sequence of closed characteristic subgroups 
\[
\{1\}\sleq N\sleq Q\sleq G
\]
such that $N$ and $G/Q$ are elementary and $Q/N$ is elementary-free.
\end{thm}

\indent The next two applications are so called local-to-global structure results. These are structure results for t.d.l.c. groups with hypotheses on the compact open subgroups. Over the last two decades, there has been a compelling number of such theorems; cf. \cite{BEW11}, \cite{BM00}, \cite{CRW_1_13}, \cite{CRW_2_13}, \cite{Will07}. We first consider groups with an open solvable subgroup.

\begin{thm}
If $G$ is a t.d.l.c.s.c. group and $G$ has an open solvable subgroup, then $G\in \Es$.
\end{thm}

We then consider $[A]$-regular t.d.l.c.s.c. groups. These groups are identified in \cite{CRW_1_13} and are defined by a technical condition on the compact open subgroups. These groups are of interest as they are the barrier to applying the powerful new theory developed in \cite{CRW_1_13}.

\begin{thm} 
If $G$ is a t.d.l.c.s.c. group that is $[A]$-regular, then $G\in \Es$. 
\end{thm}

\begin{rmk} There are already a number of further applications of the theory of elementary groups. In \cite{W_3_14}, the author shows a large group extension stable superclass of $\Es$ is contained in the class $\ms{X}$. The class $\ms{X}$ is the class of locally compact groups for which relative amenable closed subgroups are amenable; $\ms{X}$ is identified by Caprace and Monod in \cite{CM_rel_13}. In a second application \cite{W_2_14}, the author shows l.c.s.c. $p$-adic Lie groups can be decomposed into elementary groups and compactly generated topologically simple groups via group extension. Contributing to this line of study, H. Gl\"{o}ckner \cite{G14} obtains detailed structure theorems for elementary $p$-adic Lie groups. 
\end{rmk}

\begin{acknowledgements} 
Many of the results herein form part of author's thesis work at the University of Illinois at Chicago. The author wishes to thank his thesis adviser Christian Rosendal and the University of Illinois at Chicago. The author also thanks Pierre-Emmanuel Caprace for his many helpful remarks. The author finally thanks Fran\c{c}ois Le Ma\^{i}tre for suggesting improvements to the proof of Theorem~\rm\ref{thm:cg_dcomp} and the anonymous referee for his or her detailed suggestions.
\end{acknowledgements}

\subsection{Structure of the paper}
Section~\rm\ref{sec:generalities} lists the background material necessary for our arguments. The reader familiar with the literature on t.d.l.c. groups may safely skip Section~\rm\ref{sec:generalities} with the exception of Subsection~\rm\ref{subsec:synthetic} where the notion of a synthetic subgroup is developed. Section~\rm\ref{sec:elementary} establishes the first permanence properties of the class of elementary groups. We then characterize elementary groups in Section~\rm\ref{sec:characterization} via descriptive-set-theoretic trees. This characterization gives a new rank, the decomposition rank, which allows for more subtle induction arguments. Section~\rm\ref{sec:furtherperm} uses the decomposition rank to prove an additional permanence property. We present a number of examples in Section~\rm\ref{sec:examples}. We encourage the reader to periodically look ahead to Section~\rm\ref{sec:examples}; in particular, after reading the definition of the decomposition rank, it is enlightening to see the computations of this rank performed in Section~\rm\ref{sec:examples}. We finally consider applications. The applications span Sections~\rm\ref{sec:structure},~\rm\ref{sec:locsolv}, and~\rm\ref{sec:a-reg}. The applications do not rely on one another, so the reader may skip any of these sections.

\section{Generalities on totally disconnected locally compact groups}\label{sec:generalities}
We begin with a brief overview of necessary background. The notations, definitions, and facts discussed here are used frequently and, typically, without reference.

\subsection{Notations} All groups are taken to be Hausdorff topological groups and are written multiplicatively. Topological group isomorphism is denoted $\simeq$. We use ``t.d.", ``l.c.", and ``s.c." for ``totally disconnected", ``locally compact", and ``second countable", respectively.\par

\indent For a topological group $G$, $S(G)$ and $\Uc(G)$ denote the collection of closed subgroups of $G$ and the collection of compact open subgroups of $G$, respectively. All subgroups are taken to be closed unless otherwise stated. We write $H\leqslant _oG$ and $H\sleq_{cc} G$ to indicate $H$ is an open subgroup and a cocompact subgroup of $G$, respectively. Recall $H\sleq G$  is cocompact if the quotient space $G/H$ is compact in the quotient topology. 

\indent For a subset $K\subseteq G$, $C_G(K)$ denotes the collection of elements of $G$ that centralize every element of $K$. We denote the collection of elements of $G$ that normalize $K$ by $N_G(K)$. The topological closure of $K$ in $G$ is denoted by $\ol{K}$. For $A,B\subseteq G$, we put 
\[
\begin{array}{c}
A^B:=\left\{bab^{-1}\mid a\in A\text{ and }b\in B\right\}, \\
\left[A,B\right]:=\grp{aba^{-1}b^{-1}\mid a\in A\text{ and }b\in B}, \text{ and }\\
A^n:=\left\{a_1\dots a_n\mid a_i\in A\right\}.
\end{array}
\]
For $a,b\in G$, $[a,b]:=aba^{-1}b^{-1}$.\par

\indent If $X$ is a set and $G$ acts on $X$, we write $G\acts X$. For $g\in G$ and $x\in X$, $g.x$ denotes the action of $g$ on $x$. For $x\in X$, the stabilizer of $x$ in $G$ is denoted $G_{(x)}$. \par

\indent If $D$ is a countable discrete group and $X$ a countable set, we put
\[
D^{<X}:=\left\{f\in D^X\mid f(x)=1\text{ for all but finitely many }x\right\}.
\]
\noindent The group $D^{<X}$ is again a countable discrete group.\par

\indent We consider $0$ to be a limit ordinal. The first countable transfinite ordinal is denoted by $\omega$; as we consider $\Nb$ to contain zero, $\omega=\Nb$ as linear orders. The first uncountable ordinal is denoted by $\omega_1$. This work will make frequent use of ordinal numbers; see \cite{Ku80} for a nice introduction to ordinal numbers and ordinal arithmetic.  

\subsection{Basic theory} The foundational theorem in the study of t.d.l.c. groups is an old result of D. van Dantzig.

\begin{thm}[(van Dantzig {\cite[(7.7)]{HR79}})]
 A t.d.l.c. group admits a basis at $1$ of compact open subgroups.
\end{thm}

It follows from van Dantzig's theorem that compact t.d.l.c. groups are profinite; in particular, the compact open subgroups of a t.d.l.c. group are profinite. Since profinite groups are inverse limits of finite groups, they admit a basis at $1$ of open \textit{normal} subgroups. We say $(U_i)_{i\in \omega}$ is a \textbf{normal basis at} $1$ for a profinite group $U$, if $U_0=U$, $(U_i)_{i\in \omega}$ is $\subseteq$-decreasing with trivial intersection, and for each $i$, $U_i\trianglelefteq_o U$. \par

\indent The obvious topological analogues of the familiar isomorphism theorems for groups hold for t.d.l.c.s.c. groups with the exception of the first isomorphism theorem. We thus recall its statement.

\begin{thm}[({\cite[(5.33)]{HR79}})] Let $G$ be a t.d.l.c.s.c. group, $A\sleq G$ a closed subgroup, and $H\trianglelefteq G$ a closed normal subgroup. If $AH$ is closed, then $AH/H\simeq A/(A\cap H)$ as topological groups.
\end{thm}

\indent In the category of t.d.l.c.s.c. groups, care must be taken with infinite unions. Suppose $(G_i)_{i\in \omega}$ is a countable increasing sequence of t.d.l.c.s.c. groups such that $G_i\sleq_oG_{i+1}$ for each $i$. The group $G:=\bigcup_{i\in \omega} G_i$ becomes a t.d.l.c.s.c. group under the \textbf{inductive limit topology}: $A\subseteq G$ is defined to be open if and only if $A\cap G_i$ is open in $G_i$ for each $i$. We remark that the condition $G_i\sleq _o G_{i+1}$ is required. \par

\indent Our notion of infinite union allows us to define an infinite direct product that stays in the category of t.d.l.c.s.c. groups.

\begin{defn}\label{def:ldp} Suppose $A$ is a countable set, $(G_a)_{a\in A}$ is a sequence of t.d.l.c.s.c. groups, and for each $a\in A$, there is a distinguished $U_a\in \Uc(G_a)$. Letting $\{a_i\}_{i\in \omega}$ enumerate $A$, set
\begin{enumerate}[$\bullet$]
\item $S_0:=\prod_{i\in \omega}U_{a_i}$ with the product topology and
\item $S_{n+1}:=G_{a_0}\times\dots\times G_{a_n}\times \prod_{i\sgeq n+1}U_{a_i}$ with the product topology.
\end{enumerate}
The \textbf{local direct product of $(G_a)_{a\in A}$ over $(U_a)_{a\in A}$} is defined to be 
\[
\bigoplus_{a\in A}\left(G_a,U_a\right):=\bigcup_{n\in \omega}S_n
\]
with the inductive limit topology.
\end{defn}

\indent The isomorphism type of a local direct product is independent of the enumeration of $A$ used in the definition. Since $S_n\sleq_o S_{n+1}$ for each $n\in \omega$, $\bigoplus_{a\in A}(G_a,U_a)$ is again a t.d.l.c.s.c. group with $\prod_{a\in A}U_a$ as a compact open subgroup. \par

\indent We mention an alternative definition of the local direct product: let $(G_a)_{a\in A}$ and $(U_a)_{a\in A}$ be as in Definition~\rm\ref{def:ldp}. The local direct product of $(G_a)_{a\in A}$ over $(U_a)_{a\in A}$ is the collection of functions $f:A\rightarrow \bigsqcup_{a\in A}G_a$ so that $f(a)\in G_a$ for all $a\in A$ and for all but finitely many $a\in A$, $f(a)\in U_a$. The topology is given by declaring $\prod_{a\in A}U_a$ to be open.\par

\indent This definition of the local direct product, while obfuscating the topological structure, is more useful when building examples. Indeed, suppose the $G_a$ are copies of the same t.d.l.c.s.c. group $G$ and the $U_a$ are copies of $U\in \Uc(G)$. Suppose further $H$ is a t.d.l.c.s.c. group with a continuous action by permutations on $A$. The group $H$ now has a continuous action by topological group automorphisms on $\bigoplus_{a\in A}(G_a,U_a)$ by shifting the coordinates: for $h\in H$ and $f\in \bigoplus_{a\in A}(G_a,U_a)$, the element $h$ acts on $f$ by $h.f(a):=f(h^{-1}.a)$. We call this action the \textbf{shift action} of $H$ on $\bigoplus_{a\in A}(G_a,U_a)$. The shift action allows us to form $\bigoplus_{a\in A}(G_a,U_a)\rtimes H$, which is again a t.d.l.c.s.c. group when given the product topology. \par

\indent We now recall a number of properties that a t.d.l.c.s.c. group $G$ may display. The group $G$ is said to be \textbf{compactly generated} if there is a compact set $K\subseteq G$ such that $G=\grp{K}$. Compactly generated groups have many nice features; we note two here. Any closed cocompact subgroup of a compactly generated group is again compactly generated \cite{MS59}. The next property we note is a folklore result; a proof is included for completeness.

\begin{prop}[(Folklore)]\label{prop:factor} 
Suppose $G$ is a compactly generated t.d.l.c. group, $X$ is a compact generating set, and $U\in \Uc(G)$. Then there is a finite symmetric set $A\subseteq G$ containing $1$ such that $X\subseteq AU$ and $UAU=AU$. Further, for any finite symmetric set $A$ containing $1$ with $X\subseteq AU$ and $UAU=AU$, it is the case that $G=\grp{A}U$.
\end{prop}
\begin{proof}
Since $\{xU:x\in X\}$ is an open cover of $X$, we may find $B$ a finite, symmetric set containing $1$ such that $X\subseteq BU$. The set $UB$ is again compact, so there is a finite, symmetric $A$ with $B\subseteq A\subseteq UBU$ and $UB\subseteq AU$. We now see that
\[
UAU=UBU\subseteq AUU=AU,
\]
hence, $UAU=AU$. Since $X\subseteq BU\subseteq UBU=UAU=AU$, we have proved the first claim. \par

\indent Suppose $X\subseteq AU$ and $UAU=AU$ with $A$ a finite symmetric set containing $1$. By induction, $(UAU)^n=A^nU$ for all $n\sgeq 1$. We conclude that
\[
G=\langle X \rangle=\langle UAU \rangle =\bigcup_{n\sgeq 1}(UAU)^n=\bigcup_{n\sgeq 1}A^nU=\langle A\rangle U
\]
verifying the second claim.
\end{proof}

\indent We say $G$ is a \textbf{small invariant neighbourhood} group, denoted SIN group, if $G$ admits a basis at $1$ of compact open \emph{normal} subgroups. This definition is equivalent to the usual definition of a t.d.l.c. SIN group; cf. \cite[Corollary 4.1]{CM11}. We will make frequent use of the following characterization of compactly generated SIN groups: A compactly generated t.d.l.c. group is SIN if and only if it is residually discrete \cite[Corollary 4.1]{CM11}. Recall a group is residually discrete if every non-trivial element is non-trivial in some discrete quotient. More generally, $G$ is \textbf{residually $Q$} for $Q$ some property of groups if for all $g\in G\setminus \{1\}$, there is $N\trianglelefteq G$ such that $G/N$ has property $Q$ and the image of $g$ in $G/N$ is non-trivial.  \par

\indent We call $G$ \textbf{locally elliptic} if every finite subset generates a relatively compact subgroup. There is again a useful characterization: A l.c. group is locally elliptic if and only if it is a directed union of compact open subgroups \cite{Plat66}.\par

\indent We say $G$ is a \textbf{quasi-product} with \textbf{quasi-factors} $N_1,\dots,N_k$ if $N_i\trianglelefteq G$ for each $1\sleq i\sleq k$ and the multiplication map $m:N_1\times\dots\times N_k\rightarrow G$ is injective with dense image.\par

\indent Two subgroups $M\sleq G$ and $N\sleq G$ are \textbf{commensurate}, denoted $M\sim_c N$, if $|M:M\cap N|$ and $|N:M\cap N|$ are finite. It is easy to check $\sim_c$ is an equivalence relation on $S(G)$ and is preserved under the action by conjugation of $G$ on $S(G)$. The commensurability relation allows us to define the commensurator subgroup: For $N\leqslant G$, the \textbf{commensurator subgroup} of $N$ in $G$ is 
\[
Comm_G(N):=\left\{g\in G\mid gNg^{-1}\sim_cN\right\}.
\]
If $Comm_G(N)=G$, we say $N$ is \textbf{commensurated}.

\subsection{Normal subgroups} A \textbf{filtering family} $\mc{F}$ of subgroups of a group $G$ is a collection of subgroups such that for all $N,M\in \mc{F}$, there is $K\sleq M\cap N$ with $K\in \mc{F}$.

\begin{prop}[(Caprace, Monod {\cite[Proposition 2.5]{CM11}})]\label{prop:CM_minimal}
Let $G$ be a compactly generated t.d.l.c. group. Then for any compact open subgroup $V$, the subgroup $Q_V:=\bigcap_{g\in G}gVg^{-1}$ is such that any filtering family of non-discrete closed normal subgroups of $G/Q_V$ has non-trivial intersection. 
\end{prop}

\begin{cor}\label{cor:CM_comp-by-dis} 
Suppose $G$ is a compactly generated t.d.l.c.s.c. group and $\mc{N}$ is a filtering family of closed normal subgroups of $G$. If $\bigcap \mc{N}$ is trivial, then some element of $\mc{N}$ is compact-by-discrete.
\end{cor}

\begin{proof} Suppose $\mc{N}$ is a filtering family of closed normal subgroups of $G$ such that $\bigcap \mc{N}=\{1\}$.  The collection $\{N^c\mid N\in \mc{N}\}$ is an open cover of $G\setminus \{1\}$, and since $G\setminus \{1\}$ is Lindel\"{o}f, there is a countable subcover. We may thus take $\mc{N}$ to be a countable filtering family of $G$ such that $\bigcap\mc{N}=\{1\}$. It follows from the countability of $\mc{N}$ that there is $(N_i)_{i\in \omega}$ an $\subseteq$-decreasing sequence of members of $\mc{N}$ such that $\bigcap_{i\in \omega}N_i=\{1\}$.\par

\indent Fix $V\in \Uc(G)$ and let $Q_V$ be as given by Proposition~\rm\ref{prop:CM_minimal}. Since $Q_V$ is compact, $N_iQ_V$ is a closed normal subgroup of $G$. Furthermore, $\bigcap_{i\in \omega}N_iQ_V=Q_V$. Indeed, take $x\in \bigcap_{i\in \omega}N_iQ_V$, so $x=n_iq_i\in N_iQ_V$ for each $i\in\omega$. Since $Q_V$ is compact, there is a subsequence $q_{i_j}$ that converges to some $q\in Q_V$, whereby $n_{i_j}$ also converges to some $n$. For each $k$, we see that $n_{i_j}\in N_k$ for all $i_j\sgeq k$ since $(N_i)_{i\in \omega}$ is decreasing. It follows $n\in N_k$ for all $k$ and, therefore, is trivial. Thus, $x\in Q_V$.\par

\indent We now have that $(N_iQ_V/Q_V)_{i\in \omega}$ is a filtering family in $G/Q_V$ with $\bigcap_{i\in \omega}N_iQ_V/Q_V=\{1\}$. Applying Proposition~\rm\ref{prop:CM_minimal}, some $N_iQ_V/Q_V$ is discrete, and as $N_iQ_V/Q_V\simeq N_i/(N_i\cap Q_V)$, some $N_i$ is compact-by-discrete.
\end{proof}

\indent There are a number of canonical normal subgroups that arise in a t.d.l.c.s.c. group $G$. Following M. Burger and S. Mozes \cite{BM00}, the \textbf{quasi-centre} of $G$ is 
\[
QZ(G):=\left\{g\in G\mid C_G(g)\text{ is open}\right\}.
\]
The subgroup $QZ(G)$ is a characteristic subgroup but is not necessarily closed; for example, consider the profinite group $F^{\Nb}$ for $F$ a non-abelian finite simple group \cite{BM00}. \par

\indent The \textbf{discrete residual} of $G$, denoted $\Res{}(G)$, is the intersection of all open normal subgroups. The quotient by $\Res{}(G)$ has a nice structure.

\begin{prop}\label{prop:res}
If $G$ is a t.d.l.c.s.c. group, then $G/\Res{}(G)$ is a countable increasing union of SIN groups. 
\end{prop}

\begin{proof} 
Let $(g_i)_{i\in \omega}$ list a countable dense subset of $G$, fix $U\in \Uc(G)$, and set $O_i:=\grp{U,g_0,\dots,g_i}$. Certainly, $G=\bigcup_{i\in \omega}O_i$, and since $G/\Res{}(G)$ is residually discrete, 
\[
O_i\Res{}(G)/\Res{}(G)\simeq O_i/O_i\cap \Res{}(G)
\]
is residually discrete as well as compactly generated. We conclude that $O_i\Res{}(G)/\Res{}(G)$ is a SIN group, and the proposition follows.
\end{proof}

\indent By work of Platonov \cite{Plat66}, a t.d.l.c. group $G$ admits a unique maximal closed normal subgroup that is locally elliptic. This subgroup is called the \textbf{locally elliptic radical} of $G$ and is denoted by $\Rad{\mc{LE}}(G)$.\par

\indent Following Caprace, C. Reid, and G. Willis \cite{CRW_1_13}, the \textbf{quasi-centralizer of $K$ modulo $L$} for $K,L\sleq G$ is
\[
QC_G\left(K/L\right):=\left\{g\in G\mid \exists\;U\in \Uc(G)\text{ such that } [g,U\cap  K]\subseteq L\right\}.
\]
Quasi-centralizers are not necessarily subgroups; however,

\begin{prop}[(Caprace, Reid, Willis {\cite[Lemma 6.5]{CRW_1_13}})]\label{prop:quasicentralizer} 
Let $G$ be a t.d.l.c. group, let $L$ and $K$ be non-trivial compact subgroups of $G$ such that $L\trianglelefteq K$, and suppose $L$ and $K$ are commensurated. Then $QC_G(K/L)$ is a normal subgroup of $G$ that contains $L$. 
\end{prop}

\indent The group $QC_G(K/L)$ often fails to be closed, so we define 
\[
\qc{G}{K/L}:=\ol{QC_G(K/L)}.
\]
When $K=U$ with $U\in \Uc(G)$, the compact open subgroup $V:=U\cap \qc{G}{U/L}$ of $\qc{G}{K/L}$ has a useful property: Since $U$ is open, $U\cap QC_G(U/L)$ is dense in $V$, and for each $g\in U\cap QC_G(U/L)$, there is $W\sleq_o U$ such that $[g,w]\in L$ for all $w\in W$. Therefore, $QZ(V/L)$ is dense in the group $V/L$. This observation will be used in an essential way in Section~\rm\ref{sec:a-reg}.

\subsection{Synthetic subgroups}\label{subsec:synthetic}
The inchoate idea of a synthetic subgroup appears in work of I.V. Protasov and V.S. \u{C}ar\={\i}n \cite{PC78} from the late seventies. Specifically, for a t.d.l.c. group $G$ and 
\[
P_1(G):=\{g\in G\mid \cgrp{g}\text{ is compact}\},
\]
Protasov and \u{C}ar\={\i}n define the periodic radical to be the collection of $g\in G$ such that $gP_1(G)\subseteq P_1(G)$. This collection is easily shown to be a normal subgroup of $G$. We abstract this idea to show a general relationship between sets of closed subgroups and normal subgroups. The key idea we borrow from Protasov and \u{C}ar\={\i}n is to consider the elements that can be \emph{added} to the members of a collection of subgroups. \par

\indent For a t.d.l.c. group $G$, a subset $\mathcal{A}\subseteq S(G)$ is \textbf{conjugation invariant} if $\mathcal{A}$ is fixed setwise under the action by conjugation of $G$ on $S(G)$. We say $\mc{A}$ is \textbf{hereditary} if for all $C\in \mathcal{A}$, $S(C)\subseteq \mathcal{A}$.

\begin{defn} 
Suppose $\mathcal{A}\subseteq S(G)$ is conjugation invariant and hereditary. The $\mathcal{A}$\textbf{-core}, denoted $N_{\mathcal{A}}$, is the collection of $g\in G$ such that for all $C\in \mc{A}$, $\cgrp{g,C}\in\mathcal{A}$.
\end{defn}

\begin{prop}\label{prop:syn_subgroup}
If $\mathcal{A}\subseteq S(G)$ is conjugation invariant and hereditary, then $N_{\mathcal{A}}$ is a normal subgroup of $G$.
\end{prop}

\begin{proof}
Certainly, $N_{\mathcal{A}}$ is closed under inverses. For multiplication, take $h,g\in N_{\mc{A}}$ and $C\in \mc{A}$. It is clear that $\cgrp{h,g,C}\in \mc{A}$, and since $\mc{A}$ is hereditary, $\cgrp{hg,C}\in \mc{A}$. Therefore, $hg\in N_{\mc{A}}$, and $N_{\mc{A}}$ is a subgroup.\par

\indent For normality, take $g\in N_{\mc{A}}$, $k\in G$, and $C\in \mc{A}$. Since $\mathcal{A}$ is conjugation invariant, we have that $k^{-1}Ck\in \mc{A}$, so $\cgrp{g,k^{-1}Ck}\in \mc{A}$. Conjugating with $k$, we see that $\cgrp{kgk^{-1},C}\in \mc{A}$. Therefore, $kgk^{-1}\in N_{\mc{A}}$ verifying that $N_{\mathcal{A}}$ is also normal.
\end{proof}

\begin{defn}
A subgroup of a t.d.l.c. group $G$ of the form $N_{\mathcal{A}}$ for some conjugation invariant and hereditary $\mathcal{A}\subseteq S(G)$ is called a \textbf{synthetic subgroup} of $G$.
\end{defn}

\begin{rmk}
The technique of synthetic subgroups gives a method to identify normal \emph{but not necessarily closed} subgroups. Additionally, the closure of a synthetic subgroup often has a regular internal structure that aids in the study of ``decompositions".
\end{rmk}

We give four examples of synthetic subgroups. The latter three appear regularly in this work. \par
\begin{flushenumerate}
\item  Let $N\trianglelefteq G$ be a closed normal subgroup of a t.d.l.c. group $G$, so $S(N)\subseteq S(G)$ is conjugation invariant and hereditary. One checks $N$ is the $S(N)$-core.\par

\item Let $G$ be a t.d.l.c. group and consider the quasi-centre of $G$, $QZ(G)$. Putting 
\[
\mc{QZ}:=\{C\in S(G)\mid C_G(C)\text{ is open}\},
\]
$QZ(G)$ is the $\mc{QZ}$-core. \par

\item Platonov's locally elliptic radical in the t.d.l.c. setting is an example in a non-trivial manner. Let $G$ be a t.d.l.c. group and $\mc{C}\subseteq S(G)$ be the collection of compact subgroups. The set $\mc{C}$ is conjugation invariant and hereditary, so we may form $N_{\mc{C}}$. As $\{1\}\in \mc{C}$, $\cgrp{X}$ is compact for all $X\subseteq N_{\mc{C}}$ finite. Via results of \cite{Plat66} or as an exercise, $\ol{N_{\mc{C}}}$ is locally elliptic and, therefore, is contained in $\Rad{\mc{LE}}(G)$. Conversely, take $r\in \Rad{\mc{LE}}(G)$ and $C\in \mc{C}$. Since $r^C\subseteq \Rad{\mc{LE}}(G)$ is compact, $r^C$ generates a compact subgroup. We infer that $\cgrp{r,C}\sleq \cgrp{r^C}C$ is a compact subgroup and $r\in N_{\mc{C}}$. Therefore, $N_{\mc{C}}=\Rad{\mc{LE}}(G)$, and $\Rad{\mc{LE}}(G)$ is the $\mc{C}$-core.\par

\item Our last example has not been previously isolated. Consider the following family of subgroups in a t.d.l.c.s.c. group $G$:
\[
\mc{SIN}(G):=\left\{C\in S(G)\mid \forall \hspace{1pt}V\in \mc{U}(G) \;\exists W\leqslant_oV\text{ such that }C\leqslant N_G(W)\right\}
\]
We call $\mc{SIN}(G)$ the \textbf{SIN-family}. The set $\mc{SIN}(G)$ is conjugation invariant and hereditary, so we may form a synthetic subgroup. We put $\SIN(G):=N_{\mc{SIN}(G)}$ where $N_{\mc{SIN}(G)}$ is the $\mc{SIN}(G)$-core and call $\SIN(G)$ the \textbf{SIN-core} of $G$.
\end{flushenumerate}

\indent The SIN-core plays an important role in this work, and therefore, we make a few observations. We first note that $QZ(G)\sleq \SIN(G)$ and that $\SIN(G)$ is a characteristic subgroup. Our next observation shows the quasi-centre is often a proper subgroup of $\SIN(G)$.

\begin{prop} 
If $G$ is a t.d.l.c.s.c. group, then $\SIN(G)$ is closed.
\end{prop}
\begin{proof}
Take $C\in \mc{SIN}(G)$ and let $U\in \Uc(G)$ be such that $C\sleq N_G(U)$. For $V\in \Uc(G)$, we may find $W\trianglelefteq_o U$ such that $W\sleq V$. Since $C\in \mc{SIN}(G)$, we may additionally find $L\sleq_oW$ with $C\sleq N_G(L)$. Thus,
\[
\bigcap_{cu\in CU}cuWu^{-1}c^{-1}=\bigcap_{c\in C}cWc^{-1}\sgeq L,
\]
so $\bigcap_{cu\in CU}cuWu^{-1}c^{-1}$ is an open subgroup of $V$ normalized by $CU$. We conclude that $CU\in \mc{SIN}(G)$. \par

\indent Now say $g_i\rightarrow g$ with $g_i\in \SIN(G)$ for all $i$. Take $C\in \mc{SIN}(G)$, find $U\in \Uc(G)$ so that $CU\in \mc{SIN}(G)$, and fix $i$ with $g_i^{-1}g\in U$. Since $g_i\in \SIN(G)$, we have that $\grp{g_i,CU}\in \mc{SIN}(G)$ and by choice of $i$, $\grp{g,CU}\in \mc{SIN}(G)$. The hereditary property now implies $\cgrp{g,C}\in \mc{SIN}(G)$, and the proposition follows.
\end{proof}

\indent The SIN-core also has a well-behaved internal structure.

\begin{prop}\label{prop:SIN}
If $G$ is a t.d.l.c.s.c. group, then $\SIN(G)$ is an increasing union of compactly generated relatively open SIN groups.
\end{prop}
\begin{proof}
Fix $U\in \Uc(G)$, let $(n_i)_{i\in \omega}$ list a countable dense subset of $N:=\SIN(G)$, and put $H_i:=\grp{N\cap U,n_0,\dots,n_i}$. Certainly, $N=\bigcup_{i\in \omega}H_i$ with $(H_i)_{i \in \omega}$ an $\subseteq$-increasing sequence of compactly generated relatively open subgroups of $N$. It remains to show each $H_i$ is a SIN group. To that end, take $O\subseteq H_i$ a neighbourhood of $1$ in $H_i$ and find $W\in \Uc(G)$ such that $W\cap H_i\subseteq O$. We have that $H_i\in \mc{SIN}(G)$ by our choice of $n_0,\dots,n_i$, so there is $L\sleq_oW$ with $H_i\sleq N_G(L)$. The subgroup $L\cap H_i$ is thus an open normal subgroup of $H_i$ contained in  $O$ verifying that $H_i$ is a SIN group.
\end{proof}

\section{Elementary groups}\label{sec:elementary}
As it is the central idea of this work, we reiterate the definition of the class of elementary groups.

\begin{defn}
The class of \textbf{elementary groups} is the smallest class $\Es$ of t.d.l.c.s.c. groups such that
\begin{enumerate}[(i)]

\item $\Es$ contains all second countable profinite groups and countable discrete groups;

\item $\Es$ is closed under taking group extensions of second countable profinite groups and countable discrete groups. I.e. if $G$ is a t.d.l.c.s.c. group and $H\trianglelefteq G$ is a closed normal subgroup with $H\in \Es$ and $G/H$ either profinite or discrete, then $G\in \Es$; and

\item If $G$ is a t.d.l.c.s.c. group and $G=\bigcup_{i\in \omega}O_i$ where $(O_i)_{i\in \omega}$ is an $\subseteq$-increasing sequence of open subgroups of $G$ with $O_i\in\Es$ for each $i$, then $G\in\Es$. We say $\Es$ is \textbf{closed under countable increasing unions}.
\end{enumerate}
\end{defn}

\subsection{The construction rank} For $G\in \mathscr{E}$, define:

\begin{enumerate}[$\bullet$]
\item $G\in \Es_0$ if and only if $G$ is either profinite or discrete.

\item Suppose $\Es_{\alpha}$ is defined. Put $G\in \Es_{\alpha}^e$ if and only if there exists $N\trianglelefteq G$ such that $N\in \Es_{\alpha}$ and $G/N\in \Es_{0}$. Put $G\in \Es_{\alpha}^l$ if and only if $G=\bigcup_{i\in \omega}H_i$ where $(H_i)_{i\in \omega}$ is an $\subseteq$-increasing sequence of open subgroups of $G$ with $H_i\in \Es_{\alpha}$ for each $i\in \omega$. Define $\Es_{\alpha+1}:=\Es_{\alpha}^e\cup \Es_{\alpha}^l$.

\item For $\lambda$ a limit ordinal, $\Es_{\lambda}:=\bigcup_{\beta<\lambda}\Es_{\beta}$.
\end{enumerate}
\noindent It is easy to verify $\Es=\bigcup_{\alpha <\omega_1}\Es_{\alpha}$. We may thus define 
\[
\rk(G):=\min\{\alpha\mid G\in \Es_{\alpha}\}.
\]
We call this rank the \textbf{construction rank} of $G$; the construction rank is inspired by work of D. Osin \cite{Os02}. Induction on the construction rank is one of the primary tools for analysis of elementary groups. The reader is directed to Section~\rm\ref{sec:examples} to see computations of this rank in examples.\par

\indent We now collect a number of basic facts about the construction rank. The proof of our first observation is straightforward and, therefore, left to the reader.

\begin{obs} Let $G$ be elementary.
\begin{enumerate}[(1)]
\item If $\rk(G)=\alpha>0$ and $G$ is compactly generated, then $G$ is a group extension of either a profinite group or a discrete group by an elementary group of strictly lower rank. 

\item The construction rank is either zero or a successor ordinal less than $\omega_1$.

\item If $G\simeq H$ as topological groups, then $H\in \Es$ and $\rk(G)=\rk(H)$.
\end{enumerate}
\end{obs}

\begin{prop}\label{prop:open_sgrp}
If $G\in \Es$ and $O\sleq_oG$, then $O\in \Es$ and $\rk(O)\sleq \rk(G)$.
\end{prop}
 
\begin{proof}
We induct on $\rk(G)$ for the proposition. As the base case is immediate, suppose the lemma holds for elementary groups with rank at most $\alpha$. Consider a group $G$ with rank $\alpha+1$ and $O\sleq_oG$. Suppose first the rank of $G$ is given by a group extension; say $N\trianglelefteq G$ is so that $\rk(N)= \alpha$ and $\rk(G/N)=0$. Certainly, $\rk(ON/N)=0$, and $O/O\cap N\simeq ON/N$ is either profinite or discrete. On the other hand, $O\cap N\sleq_oN$, so the induction hypothesis implies $O\cap N\in \Es$ with construction rank at most $\alpha$. We conclude that $O\in\Es$ with $\rk(O)\sleq \alpha+1$.\par
 
\indent Suppose the rank of $G$ is given by a union. Say $G=\bigcup_{i\in\omega}H_i$ where $(H_i)_{i\in \omega}$ is an $\subseteq$-increasing sequence of open subgroups with $\rk(H_i)\sleq \alpha$. Certainly, $O=\bigcup_{i\in \omega}O\cap H_i$, and applying the induction hypothesis, $O\cap H_i\in \Es$ with $\rk(O\cap H_i)\sleq \alpha$ for each $i$. We thus see that $O\in \Es$ with $\rk(O)\sleq \alpha+1$ completing the induction.
\end{proof}
 
\begin{prop}\label{prop:comp_quot}
If $G\in\Es$ and $K\trianglelefteq G$ is compact, then $G/K\in \Es$ and $\rk(G/K)\sleq \rk(G)$.
\end{prop}
 
\begin{proof} 
We induct on $\rk(G)$ for the proposition. As the base case is obvious, consider $G$ with $\rk(G)=\alpha+1$. Suppose the rank of $G$ is given by a group extension; say $N\trianglelefteq G$ is such that $\rk(N)= \alpha$ and $\rk(G/N)=0$. The group $NK$ is a closed normal subgroup of $G$, and $G/NK$ is either profinite or discrete. On the other hand, applying the induction hypothesis, $NK/K\simeq N/(N\cap K)$ is elementary of rank at most $\alpha$. We now have a short exact sequence of topological groups:
\[
  \xymatrix{ 1 \ar[r] & NK/K \ar[r] & G/K \ar[r] &  G/NK \ar[r]& 1}.
\]
The group $G/K$ is thus a group extension of a rank zero group by a rank at most $\alpha$ group, hence $G/K\in \Es$ with $\rk(G/K)\sleq \alpha+1$.\par
 
\indent Suppose the rank of $G$ is given by a union. Say $G=\bigcup_{i\in \omega}H_i$ with $(H_i)_{i\in \omega}$ an $\subseteq$-increasing sequence of open subgroups with $\rk(H_i)\sleq \alpha$. We see that $G/K=\bigcup_{i\in \omega}H_iK/K$, and the induction hypothesis implies $H_iK/K\simeq H_i/H_i\cap K\in \Es$ with 
\[
\rk(H_i/H_i\cap K)\sleq \rk(H_i)
\]
for each $i$. Therefore, $G/K\in \Es$ with $\rk(G/K)\sleq \alpha+1$ finishing the induction.
\end{proof}

\subsection{Permanence properties}
For the main results of this subsection, we keep track of how the construction rank increases when using a permanence property to produce a new group in $\Es$. It is important to remember the arithmetic on the construction ranks is \emph{ordinal arithmetic}. Arithmetic for transfinite ordinals has a number of peculiarities. For example, it is non-commutative; see \cite{Ku80} for further details.\par

\begin{prop}\label{prop:grp_ext}
Let $G$ be a t.d.l.c.s.c. group. If $H\trianglelefteq G$ with $H,G/H\in \Es$, then $G\in \Es$ and 
\[
\rk(G)\sleq \rk(H)+\rk(G/H)+1.
\]
\end{prop}

\begin{proof}
We induct on $\rk(G/H)$ for the proposition. Since the base case holds by definition, consider $G$ with $\rk(G/H)=\alpha+1$. Suppose the rank of $G/H$ is given by an increasing union; that is $G/H=\bigcup_{i\in \omega }O_i$ such that $(O_i)_{i\in \omega}$ is an $\subseteq$-increasing sequence of open subgroups with $\rk(O_i)\sleq \alpha$ for each $i\in\omega$. So $G=\bigcup_{i\in \omega}\pi^{-1}(O_i)$ where $\pi:G\rightarrow G/H$ is the usual projection. Since $\rk\left(\pi^{-1}(O_i)/H\right)\sleq \alpha$ for each $i\in \omega$, the induction hypothesis implies $\pi^{-1}(O_i)\in \Es$ with 
\[
\rk(\pi^{-1}(O_i))\sleq \rk(H)+\alpha+1.
\]
We conclude that $G\in \Es$ with $\rk(G)\sleq \rk(H)+\rk(G/H)+1$. \par

\indent Suppose the rank of $G/H$ is given by a group extension; say it is given by $L\trianglelefteq G/H$. Certainly, $\pi^{-1}(L)\trianglelefteq G$, and $\rk(\pi^{-1}(L)/H)\leq \alpha$. Applying the induction hypothesis, $\pi^{-1}(L)\in \Es$ with
\[
\rk(\pi^{-1}(L))\sleq \rk(H)+\alpha+1.
\] 
On the other hand, $\rk(G/\pi^{-1}(L))=0$, so $G\in \Es$ with $\rk(G)\sleq \rk(H)+\rk(G/H)+1$. This completes the induction, and we conclude the proposition.
\end{proof}

\begin{rmk} 
Proposition~\rm\ref{prop:grp_ext} shows we may define $\Es$ by allowing all group extensions. This makes for a somewhat easier definition to parse; however, the given definition is easier to use in practice.
\end{rmk}

For our next permanence properties, a subsidiary lemma is required.

\begin{lem} \label{lem:basic_E_lemmas} Let $G$ be a t.d.l.c.s.c. group.
\begin{enumerate}[(1)]
\item If $G$ is residually discrete, then $G\in \Es$ with $\rk(G)\sleq 2$.
\item If $G=\ol{\bigcup_{i\in \omega}L_{i}}$ with $(L_{i})_{i\in \omega}$  an $\subseteq$-increasing sequence of closed elementary subgroups of $G$ such that there is $U\in \Uc(G)$ with $U\sleq N_G(L_{i})$ for all $i$, then $G\in \Es$.
\end{enumerate}
\end{lem}

\begin{proof} 
For $(1)$, $\Res{}(G)=\{1\}$, so via Proposition~\rm\ref{prop:res}, $G$ is a countable increasing union of open SIN groups. A SIN group is compact-by-discrete and hence is elementary with rank at most $1$. As $G$ is a countable increasing union of such groups, $G\in \Es$ with $\rk(G)\sleq 2$.\par

\indent For $(2)$, $UL_{i}\sleq_oG$ is elementary for each $i$, and since $G=\bigcup_{i\in \omega}UL_{i}$, we conclude that $G\in \Es$.
\end{proof}

\begin{thm}\label{thm:closure_sgrp}
If $G\in \Es$, $H$ is a t.d.l.c.s.c. group, and $\psi:H\rightarrow G$ is a continuous, injective homomorphism, then $H\in \Es$ with $\rk(H)\sleq 3(\rk(G)+1)$. In particular, if $H\sleq G$ is a closed subgroup, then $H\in \Es$ with $\rk(H)\sleq 3(\rk(G)+1)$.
\end{thm}

\begin{proof} 
We induct on $\rk(G)$. For the base case, $G$ is either profinite or discrete. If $G$ is discrete, then $H$ is discrete and we are done. Suppose $G$ is profinite with $(U_i)_{i\in \omega}$ a normal basis at $1$ for $G$. The sequence $(\psi^{-1}(U_i))_{i\in\omega}$ thus consists of open normal subgroups of $H$ with trivial intersection. We conclude that $H$ is residually discrete and via Lemma~\rm\ref{lem:basic_E_lemmas}, that $H\in \Es$ with $\rk(H)\sleq 2$.\par

\indent Suppose the proposition holds up to construction rank $\alpha$ and $\rk(G)=\alpha+1$. Suppose first the rank of $G$ is given by an increasing union; say $G=\bigcup_{i\in \omega}O_i$ with $\rk(O_i)\sleq \alpha$ for each $i$. Certainly, $H=\bigcup_{i\in \omega}\psi^{-1}(O_i)$, and the induction hypothesis implies $\psi^{-1}(O_i)\in \Es$ with 
\[
\rk\left(\psi^{-1}(O_i)\right)\sleq 3(\alpha+1)
\]
for each $i \in \omega$. We conclude that 
\[
\rk(H)\sleq 3(\alpha+1)+1< 3(\alpha+2)
\]
verifying the induction claim in this case.\par

\indent Suppose $\rk(G)$ is given by a group extension; say $N\trianglelefteq G$ is such that $\rk(N)=\alpha$ and $\rk(G/N)=0$. Since $\psi:\psi^{-1}(N)\rightarrow N$, the induction hypothesis implies $\psi^{-1}(N)\in \Es$ with $\rk\left(\psi^{-1}(N)\right)\sleq 3(\alpha+1)$. On the other hand, $\psi$ induces a continuous, injective map $\tilde{\psi}:H/\psi^{-1}(N)\rightarrow G/N$. If $G/N$ is discrete, then $H/\psi^{-1}(N)$ is discrete and 
\[
\rk(H)\sleq 3(\alpha+1)+1<3(\alpha+2).
\]
If $G/N$ is profinite, then $H/\psi^{-1}(N)$ is residually discrete as in the base case. The group $H/\psi^{-1}(N)$ is thus an increasing union of open SIN groups, hence $H=\bigcup_{i\in \omega} O_i$ with $O_i/\psi^{-1}(N)$ a SIN group for each $i$. Applying Proposition~\rm\ref{prop:grp_ext}, $O_i\in \Es$ with 
\[
\rk(O_i)\sleq 3(\alpha +1)+1+1.
\]
Since $H=\bigcup_{i\in \omega}O_i$, it follows that $H\in \Es$ with $\rk(H)\sleq 3(\alpha+1)+3$. This completes the induction, and we conclude the theorem.
\end{proof}

We note an easy but useful corollary.

\begin{cor} \label{cor:sgrp_rank}
Suppose $G\in \Es$. If $H\sleq G$ and $\rk(G)<\beta$ for $\beta$ a limit ordinal, then $\rk(H)<\beta$.
\end{cor}
\begin{proof} In view of Theorem~\rm\ref{thm:closure_sgrp}, it suffices to show for any non-zero limit ordinal $\beta$ and ordinal $\gamma<\beta$, $3\gamma+3<\beta$. Letting $(\beta_{\alpha})_{\alpha}$ list the non-zero limit ordinals, we induct on $\alpha$. For the base case, $\beta_0=\omega$, the result is immediate. \par

\indent In the case $\beta=\beta_{\lambda}$ with $\lambda$ a limit ordinal, there is $\alpha<\lambda$ so that $\gamma<\beta_{\alpha}$. By the induction hypothesis, we conclude that $3\gamma+3<\beta_{\alpha}<\beta_{\lambda}$. In the case $\beta=\beta_{\alpha+1}$, there is $m<\omega$ so that $\gamma<\beta_{\alpha}+m$. Hence, 
\[
3\gamma+3<3\beta_{\alpha}+3m+3\sleq \sup_{\xi<\beta_{\alpha}}3\xi+3m+3=\beta_{\alpha}+3m+3<\beta_{\alpha+1}
\]
where $\sup_{\xi<\beta_{\alpha}}3\xi=\beta_{\alpha}$ follows by the induction hypothesis. This finishes the induction, and we conclude the corollary.
\end{proof}

\begin{rmk}\label{rmk:construction_rank}
It is easy to build examples of elementary groups $G$ and $H$ such that $H\hookrightarrow G$ and $\rk(H)>\rk(G)$. For example, let $F_i$ for $i\in \Zb$ list copies of some non-trivial finite group. For $i\sleq 0$, let $U_i=F_i$ and for $i>0$, put $U_i=1$. We now see $\Tsumz{F_i}{U_i}\hookrightarrow \prod_{i\in \Zb}F_i$ with $\rk(\Tsumz{F_i,U_i})=1$ while $\rk( \prod_{i\in \Zb}F_i)=0$. One can build more complicated examples with $\rk(H)-\rk(G)>1$. \par

\indent We also remark that the bound in Theorem~\rm\ref{thm:closure_sgrp} is not sharp. For example, the rank zero case of the induction shows if $G$ is profinite, then $\rk(H)\sleq 2$.
\end{rmk}

\indent The class of elementary groups is closed under taking quotients. Towards demonstrating this permanence property, we require an additional lemma.

\begin{lem}\label{lem:closure_quot}
Let $G$ be a t.d.l.c.s.c. group. If $P\trianglelefteq G$ is elementary, $L\trianglelefteq G$, and $[P,L]=\{1\}$, then $\ol{PL}/L\in \Es$.
\end{lem}

\begin{proof}
We induct on $\rk(P)$ for the lemma. For the base case, $\rk(P)=0$, $P$ is either compact or discrete. When $P$ is compact, $PL$ is closed, so $PL/L\simeq P/P\cap L$ and is elementary. When $P$ is discrete, every element of $P$ has an open centralizer since $P$ discrete and normal. Every element of $P$ in $\ol{PL}$ thus has an open centralizer. It follows 
\[
\SIN(\ol{PL}/L)=\ol{PL}/L,
\]
and in view of Proposition~\rm\ref{prop:SIN}, $\ol{PL}/L$ is elementary.\par

\indent Suppose $\rk(P)=\alpha+1$. We first reduce to a compactly generated group. Fix $U\in \Uc(\ol{PL})$, fix $V\in \Uc(P)$, and let $(p_i)_{i\in \omega}$ list a countable dense subset of $P$. For each $i\in \omega$, set
\[
P_i:=\grp{V^U,p_0^U,\dots,p_i^U}.
\]
Each $P_i$ is normalized by $U$, so $P_iUL$ is an open subgroup of $\ol{PL}$. The sequence $(P_iUL)_{i\in \omega}$ is therefore an $(\subseteq)$-increasing exhaustion of $\ol{PL}$ by open subgroups, hence the induction may be completed by showing each $P_iUL\in \Es$. We may indeed reduce further: $\ol{P_iL}/L\cc P_iUL/L$, hence the induction will be completed by showing $\ol{P_iL}/L\in \Es$ for each $i$. \par

\indent To this end, fix $i\in \omega$, put $Q:=P_i$, and set $R:=\ol{QL}$. We now show $R/L$ is elementary which will verify the induction claim. Since $Q\sleq_oP$, $\rk(Q)\sleq \alpha+1$, and since $Q$ is compactly generated, there is $M\trianglelefteq Q$ such that $\rk(M)\sleq \alpha$ and $\rk(Q/M)=0$. Furthermore, $Q\trianglelefteq R$ and $M\trianglelefteq R$, because $L$ centralizes $P$. Passing to $R/M$, we have that $Q/M\trianglelefteq R/M$ and $\ol{ML}/M\trianglelefteq R/M$. The group $Q/M$ is either profinite or discrete, and $[Q/M,\ol{ML}/M]=\{1\}$. By the base case and construction of $R$, 
\[
\ol{\left( \ol{ML}/M\right)\left( Q/M \right)}/ \left( \ol{ML}/M\right)=(R/M)/(\ol{ML}/M)\simeq R/\ol{ML}
\]
is elementary. \par

\indent On the other hand, $\rk(M)\sleq \alpha$, $[M,L]=\{1\}$, and $M,L\trianglelefteq R$. The induction hypothesis therefore implies $\ol{ML}/L$ is also elementary. Since 
\[
(R/L)/(\ol{ML}/L)\simeq R/\ol{ML}\in \Es,
\]
we apply Proposition~\rm\ref{prop:grp_ext} to conclude that $R/L$ is elementary. The induction is thus completed.
\end{proof}

\begin{thm}\label{thm:closure_quot} 
If $G\in \Es$ and $L\trianglelefteq G$, then $G/L\in \Es$.
\end{thm}

\begin{proof} We induct on $\rk(G)$ for the theorem. As the base case is immediate, we suppose $\rk(G)=\alpha+1$. Consider first the case in which $\rk(G)$ is given by a union. That is $G=\bigcup_{i\in \omega}H_i$ with $(H_i)_{i\in\omega}$ an $\subseteq$-increasing sequence of open subgroups each with rank at most $\alpha$. Now $G/L=\bigcup_{i\in \omega}H_iL/L$, and since $H_iL/L\simeq H_i/H_i\cap L$ for each $i$, the induction hypothesis implies $H_iL/L\in \Es$ for each $i$. We conclude that $G/L\in \Es$.

\indent Suppose $\rk(G)$ is given by a group extension; say $H\trianglelefteq G$ is such that $\rk(H)=\alpha$ and $\rk(G/H)= 0$. Pass to $\tilde{G}:=G/L\cap H$ and let $\tilde{H}$ and $\tilde{L}$ be the images of $H$ and $L$ in $\tilde{G}$, respectively. Since $\rk(H)= \alpha$, the group $\tilde{H}$ is elementary by the induction hypothesis, and the construction of $\tilde{G}$ ensures that $[\tilde{H},\tilde{L}]=\{1\}$. Lemma~\rm\ref{lem:closure_quot} thus implies $\ol{\tilde{H}\tilde{L}}/\tilde{L}$ is elementary. Seeing as $\ol{\tilde{H}\tilde{L}}/\tilde{L}\simeq \ol{HL}/L$, we conclude that $\ol{HL}/L$ is indeed elementary. \par

\indent On the other hand, $\ol{HL}/H\trianglelefteq G/H$ and $G/H$ is either profinite or discrete. Hence, 
\[
(G/H)/(\ol{HL}/H)\simeq G/\ol{HL}\simeq (G/L)/(\ol{HL}/L)
\]
is either profinite or discrete. The group $G/L$ is thus a group extension of either a profinite group or a discrete group by an elementary group and, thus, is elementary. This completes the induction, and we conclude the theorem.
\end{proof}

\begin{rmk}
A bound on the construction rank of $G/L$ for $L\trianglelefteq G\in \Es$ is given in Corollary~\rm\ref{cor:rk_bound_quot}.
\end{rmk}

The reader familiar with the theory of elementary amenable discrete groups recalls that the class of elementary amenable groups can be defined analogously to the class of elementary groups; see \cite{Ch80}. Our closure properties hitherto mirror those of the class of elementary amenable groups defined as such. The next theorem, however, is surprisingly disanalogous. Indeed, the analogous statement is patently false in the class of elementary amenable groups.

\begin{thm}\label{thm:closure_residual}
If $G$ is a t.d.l.c.s.c. group that is residually elementary, then $G\in \Es$ with
\[
\rk(G)\sleq \sup\{\rk(G/N)\mid N\trianglelefteq G\text{ with }G/N\in \Es\}+3.
\]
In particular, $\Es$ is closed under inverse limits that result in t.d.l.c.s.c. groups.
\end{thm}
\begin{proof} 
Let $\mc{F}$ be the collection of closed normal subgroups of $G$ with elementary quotient. We first claim $\mc{F}$ is a filtering family; it suffices to show $\mc{F}$ is closed under intersection. Take $N,M\in \mc{F}$ and consider $M/M\cap N$. The restriction of the usual projection gives a continuous, injective homomorphism $M/M\cap N\rightarrow G/N$. Since $G/N$ is elementary, Theorem~\rm\ref{thm:closure_sgrp} implies $M/M\cap N$ is elementary, so $M/M\cap N$ is a closed normal elementary subgroup of $G/N\cap M$ with elementary quotient. As $\Es$ is closed under group extension, we conclude that $G/M\cap N\in \Es$ and, therefore, that $\mc{F}$ is closed under intersection. \par

\indent Fix $O\sleq_oG$ a compactly generated open subgroup of $G$ and put $\mc{F}_O:=\{N\cap O\mid  N\in \mc{F}\}$. For each $N\in \mc{F}$, we have that $ON/N \simeq O/N\cap O$ and, therefore, that $O/N\cap O$ is elementary. The family $\mc{F}_O$ is thus a filtering family of closed normal subgroups of $O$ with elementary quotient. Since $G$ is residually elementary, $\bigcap \mc{F}_O=\{1\}$, so Corollary~\rm\ref{cor:CM_comp-by-dis} implies there is $N\in \mc{F}_O$ that is profinite-by-discrete and, hence, elementary. The group $O/N$ is also elementary, whereby Proposition~\rm\ref{prop:grp_ext} implies $O$ is elementary with 
\[
\rk(O)\sleq 1+\rk(G/N)+1\sleq \rk(G/N)+2.
\]

\indent Since $G$ has a countable $\subseteq$-increasing exhaustion by such $O$, $G\in \Es$ with 
\[
\begin{array}{ccl}
\rk(G) & \sleq & \sup\{\rk(G/N)+2\mid N\trianglelefteq G\text{ with }G/N\in \Es\}+1\\
	   & \sleq & \sup\{\rk(G/N)\mid N\trianglelefteq G\text{ with }G/N\in \Es\}+3.
\end{array}
\]

\indent For the inverse limit claim, suppose $(C_{\alpha},\phi_{\alpha})_{\alpha \in I}$ is an inverse system of elementary groups with a t.d.l.c.s.c. group $H$ as the inverse limit. The projection onto each coordinate $\alpha\in I$, $\pi_{\alpha}:H\rightarrow C_{\alpha}$, is a continuous homomorphism into an elementary group $C_{\alpha}$. Theorem~\rm\ref{thm:closure_sgrp} thus implies $H/\ker(\pi_{\alpha})$ is elementary for all $\alpha \in I$. Since 
\[
\bigcap_{\alpha\in I}\ker(\pi_{\alpha})=\{1\},
\]
$H$ is residually elementary and, therefore, elementary.
\end{proof}

We conclude with three easy but useful permanence properties of $\Es$; we do not compute bounds for the construction rank as such bounds are not used herein and are easily computed.

\begin{prop}\label{cons:closure_quasi}
If $G$ is a quasi-product with quasi-factors $N_1,\dots, N_k$ such that $N_i\in \Es$ for each $i$, then $G\in \Es$.
\end{prop}

\begin{proof}
We induct on $k$ for the proposition. If $k=1$, then the result obviously holds. Suppose $G$ is a quasi-product with elementary quasi-factors $N_1,\dots, N_{k+1}$. Fix $U\in \Uc(G)$ and put $M:=\cgrp{N_1,\dots ,N_k}$. We now have that
\[
G/M\simeq UN_{k+1}M/M\simeq UN_{k+1}/(UN_{k+1}\cap M).
\]
The rightmost group is a quotient of an elementary group, hence Theorem~\rm\ref{thm:closure_quot} implies $G/M\simeq UN_{k+1}/(UN_{k+1}\cap M)\in \Es$.\par

\indent The group $M$, on the other hand, is a quasi-product with elementary quasi-factors $N_1,\dots, N_k$. The induction hypothesis thereby implies $M\in \Es$, and as $\Es$ is closed under group extension, $G\in \Es$ completing the induction.
\end{proof}

\begin{prop}\label{prop:sum_rank} 
If $(G_n)_{n\in \omega}$ is a sequence of elementary groups with a distinguished $U_n\in \Uc(G_n)$ for each $n$, then $\Tsumw{G_n}{U_n}\in \Es$.
\end{prop}
 
\begin{proof}
 By construction, $\Tsumw{G_n}{U_n}$ is an increasing union of groups of the form 
\[
 G_{0}\times\dots\times G_{n}\times \prod_{i\sgeq n+1}U_{i}.
\]
These groups are built via repeated group extension from profinite groups and $G_{0},\dots ,G_{n}$, so Proposition~\rm\ref{prop:grp_ext} implies $G_{0}\times\dots\times G_{n}\times \prod_{i\sgeq n+1}U_{i}\in \Es$. It now follows that $\Tsumw{G_n}{U_n}\in \Es$.
\end{proof}
 
\begin{prop}
If $G$ is a t.d.l.c.s.c. group and $(C_i)_{i\in \omega}$ is an $\subseteq$-increasing sequence of elementary subgroups of $G$ such that $N_G(C_i)$ is open for each $i$ and $\ol{\bigcup_{i\in \omega}C_i}=G$, then $G\in \Es$.
\end{prop}
 
\begin{proof}
Fix $O\sleq G$ open and generated by a compact $X\subseteq O$ and let $U\in \Uc(O)$. Via Proposition~{\rm\ref{prop:factor}}, there is $A\subset O$ finite, symmetric, and containing $1$ so that $X\subseteq AU$ and $UAU=AU$. Since $\bigcup_{i\in \omega}C_i$ is dense, we may find $B\subseteq C_i$ for a large enough $i$ so that $BU=AU$ and $B$ is finite and symmetric and contains $1$. We see that $UBU=BU$, and hence, $O=\grp{B}U$.\par

\indent Find $V\sleq_o U$ with $V\sleq N_G(C_i)$. We see that $\cgrp{B^V}\sleq C_i\in \Es$, and since $\cgrp{B^V}\trianglelefteq_{cc} \grp{B^V}V$, we infer that $\grp{B^V}V\in \Es$. On the other hand, the group $\grp{B^V}V$ is finite index in $O$, so we may take the normal core $L$ of $\grp{B^V}V$ in $O$. Theorem~\rm\ref{thm:closure_sgrp} implies $L\in \Es$, whence $O$ is elementary-by-finite and, therefore, elementary. Since $G$ is an increasing union of compactly generated open elementary subgroups, $G\in \Es$ proving the proposition.
\end{proof}

Collecting our results, we have the following theorem.

\begin{thm}\label{thm:closure_main}
$\Es$ enjoys the following permanence properties:
\begin{enumerate}[(1)]
\item $\Es$ is closed under group extension.

\item If $G\in \Es$, $H$ is a t.d.l.c.s.c. group, and $\psi:H\rightarrow G$ is a continuous, injective homomorphism, then $H\in \Es$. In particular, $\Es$ is closed under taking closed subgroups.

\item $\Es$ is closed under taking quotients by closed normal subgroups.

\item If $G$ is a residually elementary t.d.l.c.s.c. group, then $G\in \Es$. In particular, $\Es$ is closed under inverse limits that result in t.d.l.c.s.c. groups.

\item $\Es$ is closed under taking quasi-products.

\item $\Es$ is closed under local direct products.

\item If $G$ is a t.d.l.c.s.c. group and $(C_i)_{i\in \omega}$ is an $\subseteq$-increasing sequence of elementary subgroups of $G$ such that $N_G(C_i)$ is open for each $i$ and $\ol{\bigcup_{i\in \omega}C_i}=G$, then $G\in \Es$.
\end{enumerate}
\end{thm}

\subsection{Elementarily robust classes}
We conclude this section by observing a generalization of our results for elementary groups.

\begin{defn}\label{def:e_close}
Given a class of t.d.l.c.s.c. groups $\ms{G}$, the \textbf{elementary closure} of $\ms{G}$, denoted $\ms{EG}$, is the smallest class of t.d.l.c.s.c. groups such that
\begin{enumerate}[(i)]
\item $\ms{EG}$ contains $\ms{G}$, all second countable profinite groups, and countable discrete groups.

\item $\ms{EG}$ is closed under group extensions of second countable profinite groups, countable discrete groups, and groups in $\ms{G}$.

\item If $G$ is a t.d.l.c.s.c. group and $G=\bigcup_{i\in \omega}G_i$ where $(G_i)_{i\in \omega}$ is an $\subseteq$-increasing sequence of open subgroups of $G$ with $G_i\in\ms{EG}$ for each $i$, then $G\in \ms{EG}$. We say $\ms{EG}$ is closed under countable increasing unions.
\end{enumerate}
\end{defn}

The next observation follows from the proof of Theorem~\rm\ref{thm:closure_main}.

\begin{thm}Suppose $\ms{G}$ is a class of t.d.l.c.s.c. groups that is closed under isomorphism of topological groups, taking closed subgroups, and taking Hausdorff quotients. Suppose further $\ms{G}$ satisfies the following:
\begin{enumerate}[(a)] 
\item If $H$ is a t.d.l.c.s.c. group and $ \psi:H\rightarrow G$ is a continuous, injective homomorphism with $G\in \ms{G}$, then $H\in \ms{EG}$.
\item If $G$ is a t.d.l.c.s.c. group, $H,L\trianglelefteq G$, $[H,L]=\{1\}$, and $H\in \ms{G}$, then $\ol{HL}/L\in \ms{EG}$.
\end{enumerate}
Then the permanence properties stated in Theorem~\rm\ref{thm:closure_main} hold of $\ms{EG}$.
\end{thm}

Classes $\ms{G}$ that are closed under isomorphism of topological groups, taking closed subgroups, and taking Hausdorff quotients and, additionally, satisfy $(a)$ and $(b)$ are called \textbf{elementarily robust}. In \cite[Theorem 6.3]{W_2_14}, we show $\ms{P}$, the collection of all l.c.s.c. $p$-adic Lie groups for all primes $p$, is elementarily robust.

\section{A characterization of elementary groups}\label{sec:characterization}
Our work above, while showing the class of elementary groups is robust, does not give an easy way to identify elementary groups. We here characterize elementary groups in terms of well-founded descriptive-set-theoretic trees giving a criterion by which to identify elementary groups. As a consequence of this characterization, we obtain a new rank on elementary groups.

\subsection{Preliminaries}
 
\begin{defn}
A topological space $X$ is \textbf{Polish} if the topology is separable and completely metrizable.
\end{defn}

\indent By classical results, \cite[(5.3)]{K95}, t.d.l.c.s.c. groups are Polish; i.e. as topological spaces these groups are Polish spaces. \par

\indent For a Polish space $X$, $F(X)$ denotes the set of all closed subsets of $X$. This is the so-called \textbf{Effros Borel space}; we remark, but do not use, that $F(X)$ comes with a canonical Borel sigma algebra \cite[12.C]{K95}.

\begin{thm}[(Kuratowski, Ryll-Nardzewski {\cite[(12.13)]{K95}})]\label{selector}
Let $X$ be a Polish space. There is a sequence of Borel functions $d_n:F(X)\rightarrow X$ such that for nonempty $F\in F(X)$, $\{d_n(F)\}_{n\in \omega}$ is dense in $F$.
\end{thm}

\indent The functions $d_n$ are called \textbf{selector} functions for $F(X)$. Given $Y$ a closed subset of $X$, we have that $F(Y)\subseteq F(X)$. A set of selector functions $D$ for $F(X)$ then restricts to a set of selector functions for $F(Y)$. We will abuse notion and say $D$ is also a set of selector functions for $F(Y)$.\par

\indent We will also require the notion of a descriptive-set-theoretic tree; this notion of a tree differs from the usual graph-theoretic definition. The definitions given here are restricted to the collection of finite sequences of natural numbers. See \cite[2.A]{K95} for an excellent, general account. \par

\indent  Denote the collection of finite sequences of natural numbers by $\wbaire$. For sequences $s:=(s_0,\dots,s_n)\in \wbaire$ and $r:=(r_0,\dots,r_m)\in \wbaire$, we write $s\subseteq r$ if $s$ is an initial segment of $r$. That is to say, $n\sleq m$, and $s_i=r_i$ for $0\sleq i\sleq n$. The empty sequence, denoted $\emptyset$, is considered to be an element of $\wbaire$ and is an initial segment of any $t\in \wbaire$. We define
\[
s\conc r:=(s_0,\dots,s_n,r_0,\dots,r_m).
\]
For $t=(t_0,\dots, t_k)\in \wbaire$, the \textbf{length} of $t$, denoted $|t|$, is the number of coordinates; i.e. $|t|:=k+1$.  If $|t|=1$, we write $t$ as a natural number as opposed to a sequence of length one. For $0\sleq i\sleq |t|-1$, $t(i):=t_i$. For $\alpha\in \baire$, we set $\alpha\rest_n:=(\alpha(0),\dots,\alpha(n-1))$, so $\alpha \rest_n\in \wbaire$ for any $n\sgeq 0$. 

\begin{defn} 
$T\subseteq \wbaire$ is a \textbf{tree} if it is closed under taking initial segments. We call the elements of $T$ the \textbf{nodes} of $T$. If $s\in T$ and there is no $n\in \Nb$ such that $s\conc n\in T$, we say $s$ is a \textbf{terminal node} of $T$. An \textbf{infinite branch} of $T$ is a sequence $\alpha\in \baire$ such that $\alpha\rest_n\in T$ for all $n$. If $T$ has no infinite branches, we say $T$ is \textbf{well-founded}.
\end{defn}

For $T$ a tree and $s\in T$, we put $T_s:=\{r\in \wbaire\mid s\conc r\in T\}$. The set $T_s$ is the tree obtained by taking the elements in $T$ that extend $s$ and deleting the initial segment $s$ from each.\par

\indent For $T$ a well-founded tree, there is an ordinal valued rank $\rho_T$ on the nodes of $T$ defined inductively as follows: If $s\in T$ is terminal, $\rho_T(s)=0$. For a non-terminal node $s$,
\[
\rho_T(s):=\sup\left\{\rho_T(s\conc n)+1\mid n\in \Nb\text{ and }s\conc n \in T\right\}.
\]
The \textbf{rank} of a well-founded tree $T$ is then defined to be
\[
\rho(T):=\sup\{\rho_T(s)+1\mid s\in T\}.
\]
In the case $T\neq \emptyset$, it is easy to verify $\rho(T)=\rho_T(\emptyset)+1$. We thus see that $\rho(T)$ is always either a successor ordinal or zero. We make two further observations. \par

\begin{obs}\label{obs:wf_rank} Suppose $T\subseteq \wbaire$ is a well-founded tree and $s\in T$. Then
\begin{enumerate}[(1)]
\item $\rho_T(s)+1=\rho(T_s)$ and
\item $\rho(T)=\sup\{\rho(T_{i})\mid  i\in T\}+1$.
\end{enumerate}
\end{obs}

\subsection{Decomposition trees}

Fix $G$ a t.d.l.c.s.c. group, $U\in \Uc(G)$, and $D:=\{d_n\}_{n\in \Nb}$ selector functions for $F(G)$. For $H\sleq G$ and $n\in \Nb$, define
\[
R^{(U,D)}_n(H):=\grp{U\cap H,d_0(H),\dots,d_{n}(H)}
\]
where the $d_i$ are selector functions from $D$. We now define a tree $T_{(U,D)}(G)$ and associated subgroups of $G$. Put

\begin{enumerate}[$\bullet$]
\item $\emptyset\in T_{(U,D)}(G)$ and $G_{\emptyset}:=G$.
\item Suppose we have defined $s\in T_{(U,D)}(G)$ and $G_{s}\sleq G$. Put $s\conc n\in T_{(U,D)}(G)$ and 
\[
G_{s\conc n}:=\Res{}\left(R^{(U,D)}_n(G_s)\right)
\]
if and only if $G_s\neq \{1\}$. 
\end{enumerate}

We call $T_{(U,D)}(G)$ the \textbf{decomposition tree} of $G$ with respect to $U$ and $D$. This tree is always non-empty, and any terminal node corresponds to the trivial group. We make one further observation; the proof is straightforward and, therefore, omitted.
 
\begin{obs}\label{obs:key_drank}
For any $s\in T_{(U,D)}(G)$, $T_{(U,D)}(G)_s=T_{(G_s\cap U, D)}(G_s)$. Further, for $r\in T_{(G_s\cap U, D)}(G_s)$, the associated subgroup $(G_s)_r$ is the same as the subgroup $G_{s\conc r}$ associated to $s\conc r\in T_{(U,D)}(G)$.
\end{obs} 

We now characterize elementary groups in terms of decomposition trees. This requires a technical result.

\begin{lem}\label{lem:wf}
Suppose $G\in \Es$ and $T_{(U,D)}(G)$ is the decomposition tree for $G$ with respect to $U\in \Uc(G)$ and $D$ a set of selector functions for $F(G)$. If $\rk(G)\sleq \beta+m$ for $\beta$ a limit ordinal and $m\sgeq 1$ a finite ordinal, then $\rk(G_s)\sleq \beta$ for all $s\in T_{(U,D)}(G)$ with $|s|\sgeq m$.
\end{lem}

\begin{proof}
We induct on $m$ for a stronger hypothesis: Suppose $H\sleq G$ and form $T_{(H\cap U,D)}(H)$. If there is $O\sleq G$ with $H\sleq O\sleq G$ so that $\rk(O)\sleq \beta+m$ with $\beta$ some limit ordinal and $m\sgeq 1$ a finite ordinal, then $\rk(H_s)\leq \beta$ for any $s\in T_{(H\cap U,D)}(H)$ with $|s|\sgeq m$. In view of Corollary~\rm\ref{cor:sgrp_rank} and the fact that $H_{s\conc n}\sleq H_s$, it suffices to show the induction hypothesis for $s$ such that $|s|=m$. \par

\indent For the base case, suppose $\rk(O)\sleq \beta+1$ and fix $i\in \Nb$. We now have two cases: $\beta=0$ and $\beta$ a non-zero limit ordinal. If $\beta$ is a non-zero limit ordinal, then either $\rk(O)<\beta$ or $\rk(O)=\beta+1$. In the latter case, $O$ is an increasing union of open subgroups of rank strictly less than $\beta$ since $\beta$ is a limit ordinal. In either case, we may thus write $O=\bigcup_{j\in \omega}L_j$ with $(L_j)_{j\in \omega}$ an $\subseteq$-increasing sequence of open subgroups with $\rk(L_j)<\beta$ for each $j\in \omega$. Since $R^{(H\cap U,D)}_i(H_{\emptyset})$ is compactly generated, $R^{(H\cap U, D)}_i(H_{\emptyset})\sleq L_j$ for some $j$ and, therefore, has rank strictly less than $\beta$ via Corollary~\rm\ref{cor:sgrp_rank}. A second application of Corollary~\rm\ref{cor:sgrp_rank} now gives that $\rk(H_i)<\beta$. \par

\indent Suppose $\beta=0$. If $O$ has construction rank zero, then we are done, so we suppose $\rk(O)=1$. If $O$ is an increasing union of open subgroups of rank zero, then $R^{(H\cap U, D)}_i(H_{\emptyset})$ has rank zero since compactly generated, and therefore, $H_i$ has rank zero as required. Suppose $O$ is a group extension of rank zero groups; say $L\trianglelefteq O$ is such that $\rk(L),\rk(O/L)$ are zero. We see here that
\[
H_i=\Res{}\left(R^{(H\cap U, D)}_i(H_{\emptyset})\right)\sleq \Res{}(O)\sleq L,
\]
hence $\rk(H_i)=0$. We have thus verified the base case. \par

\indent Now suppose $\rk(O)\sleq\beta+m+1$ and fix $s\in T_{(H\cap U,D)}(H)$ with $|s|=m+1$. If $\rk(O)$ is given by an increasing union, there is $L\sleq_o O$ such that $R^{(H\cap U,D)}_{s(0)}(H_{\emptyset})\sleq L$ and $\rk(L)\sleq \beta+m$. Therefore, $H_{s(0)}\sleq L$ with $\rk(L)\sleq \beta+m$, so the induction hypothesis implies $\rk((H_{s(0)})_r)\sleq \beta$ for 
\[
r=(s(1),\dots,s(m))\in T_{(H_{s(0)}\cap U,D)}(H_{s(0)}).
\]
In view of Observation~\rm\ref{obs:key_drank}, we conclude that $(H_{s(0)})_r=H_s$, so the induction claim holds when $\rk(O)$ is given by an increasing union. \par

\indent Suppose $\rk(O)$ is given by a group extension; say $L\trianglelefteq O$ is such that $\rk(L)\sleq \beta +m$ and $O/L$ is either discrete or profinite. Then,
\[
H_{s(0)}=\Res{}\left(R^{(H\cap U,D)}_{s(0)}(H_{\emptyset})\right)\sleq \Res{}(O)\sleq L.
\]
We may now apply the induction hypothesis to $H_{s(0)}\sleq L\sleq G$ to conclude that $\rk((H_{s(0)})_r)\sleq \beta$ for $r=(s(1),\dots,s(m))$. Observation~\rm\ref{obs:key_drank} implies $(H_{s(0)})_r=H_s$, and the induction is finished. \par

\indent The desired lemma is now in hand: Suppose $\rk(G)\sleq \beta+m$ for $\beta$ a limit ordinal and $m\sgeq 1$ a finite ordinal. Applying our work above to $G=H=O$, we conclude that $\rk(G_s)\sleq \beta$ for all $s\in T_{(U,D)}(G)$ with $|s|\sgeq m$. 
\end{proof}

\begin{thm} 
Suppose $G$ is a  t.d.l.c.s.c. group, $U\in \Uc(G)$, and $D$ is a set of selector functions for $F(G)$. Then $G$ is elementary if and only if $T_{(U,D)}(G)$ is well-founded.
\end{thm}

\begin{proof} Suppose $G\in \Es$; we argue by induction on $\rk(G)$ that $T_{(U,D)}(G)$ is well-founded simultaneously for all $G$, $U$, and $D$. For the base case, $\rk(G)=0$, $G$ is residually discrete. Each element in $T_{(U,D)}(G)$ therefore has length at most $2$, and $T_{(U,D)}(G)$ is well-founded.\par

\indent Suppose $\rk(G)=\alpha+1$. We may write $\rk(G)=\beta+m$ for $\beta\sleq\alpha$ some limit ordinal and $m\sgeq 1$ a finite ordinal. Applying Lemma~\rm\ref{lem:wf}, $\rk(G_s)\sleq \beta$ for all $s\in T_{(U,D)}(G)$ with $|s|=m$. The induction hypothesis now implies $T_{(G_s\cap U,D)}(G_s)$ is well-founded, and from Observation~\rm\ref{obs:key_drank}, we infer $T_{(U,D)}(G)_s$ is well-founded. Since any infinite branch of $T_{(U,D)}(G)$ gives an infinite branch in $T_{(U,D)}(G)_s$ for $s$ the initial segment of length $m$, $T_{(U,D)}(G)$ must be well-founded. This completes the induction, and we conclude the forward implication.\\
\\
\indent Suppose $T_{(U,D)}(G)$ is well-founded. We induct on $\rho(T_{(U,D)}(G))$ simultaneously for all $G$, $U\in \Uc(G)$, and $D$ selector functions for $F(G)$ for the claim that $G\in \Es$. If $\rho(T_{(U,D)}(G))=1$, then $G=\{1\}$ and is obviously elementary. Suppose $\rho(T_{(U,D)}(G))=\beta+1$. Observation~\rm\ref{obs:wf_rank} and Observation~\rm\ref{obs:key_drank} imply $\rho(T_{(G_i\cap U, D)}(G_i))\sleq \beta$ for each $i\in \Nb$, so the induction hypothesis implies $G_i\in \Es$ for each $i$. Since $R^{(U,D)}_i(G)/G_i$ is a SIN group, we further have that $R_i^{(U,D)}(G)\in \Es$ for each $i$, and since the $(R_i^{(U,D)}(G))_{i\in \Nb}$ is an $\subseteq$-increasing exhaustion of $G$, we conclude that $G\in \Es$. The induction is now finished, and we conclude the reverse implication.
\end{proof}

\subsection{The decomposition rank}
The rank of a decomposition tree $T_{(U,D)}(G)$ as a well-founded tree gives another rank on $\Es$. To verify this, we need to show the rank is independent of our choice of $U$ and $D$. \par

\begin{prop}\label{prop:drank} 
Suppose $G\in \Es$, $U\in \Uc(G)$, and $D$ is a set of selector functions for $F(G)$. Suppose additionally, $H\in \Es$, $W\in \Uc(H)$, and $C$ is a set of selector functions for $F(H)$. If $\psi:H\rightarrow G$ is a continuous, injective homomorphism, then 
\[
\rho(T_{(W,C)}(H))\sleq \rho (T_{(U,D)}(G)).
\]
In particular, for $G\in \Es$, $U$ and $V$ in $\Uc(G)$, and $D$ and $C$ sets of selector functions for $F(G)$, $\rho (T_{(U,D)}(G))=\rho (T_{(V,C)}(G))$.
\end{prop}

\begin{proof} We induct on $\rho(T_{(U,D)}(G))$ simultaneously for all $G\in \Es$, $U\in \Uc(G)$, and $D$ a set of selector functions for $F(G)$. The base case is obvious since $\rho(T_{(U,D)}(G))=1$ implies $G=\{1\}$.\par

\indent Suppose $\rho(T_{(U,D)}(G))=\beta+1$. For each $i$, $R^{(W,C)}_i(H)$ is compactly generated, so there is $n(i)$ with $\psi\left(R^{(W,C)}_i(H)\right)\sleq R^{(U,D)}_{n(i)}(G)$. We thus have that
\[
\psi(H_i)=\psi\left(\Res{}\left(R^{(W,C)}_i(H)\right)\right)\sleq \Res{}\left(R^{(U,D)}_{n(i)}(G)\right)=:G_{n(i)}.
\]
The map $\psi$ thereby restricts to $\psi:H_i\rightarrow G_{n(i)}$, and Observation~\rm\ref{obs:wf_rank} and Observation~\rm\ref{obs:key_drank} imply
\[
\rho\left(T_{(G_{n(i)}\cap U,D)}(G_{n(i)})\right)=\rho\left(T_{(U,D)}(G)_i\right)\sleq \beta.
\]
Applying the induction hypothesis,
\[
\rho\left(T_{(H_i\cap W,C)}(H_i)\right)\sleq \rho\left(T_{(G_{n(i)}\cap U,D)}(G_{n(i)})\right).
\]
Therefore,
\[
\rho\left(T_{(W,C)}(H)\right)=\sup_{i\in \omega}\rho\left(T_{(H_i\cap W,C)}(H_i)\right)+1\sleq \sup_{i\in \omega}\rho\left(T_{(G_{n(i)}\cap U,D)}(G_{n(i)})\right)+1\sleq\rho\left(T_{(U,D)}(G)\right).
\]
This finishes the induction, and we conclude the proposition.
\end{proof}

\begin{rmk} The inequality $\sup_{i\in \omega}\rho(T_{(G_{n(i)}\cap U,D)}(G_{n(i)}))+1\sleq\rho(T_{(U,D)}(G))$ appearing above is an equality provided $n(i)$ is unbounded as $i\rightarrow \infty$. Indeed, if $n(i)$ is unbounded, then for all $k\in \omega$, there is $i$ so that $k\leq n(i)$. It follows $G_k\hookrightarrow G_{n(i)}$, so appealing to Proposition~\rm\ref{prop:drank}, $\rho(T_{(G_{k}\cap U,D)}(G_{k}))\leq \rho(T_{(G_{n(i)}\cap U,D)}(G_{n(i)}))$. Therefore,
\[
\rho(T_{(U,D)}(G))=\sup_{k\in \omega}\rho\left(T_{(G_{k}\cap U,D)}(G_{k})\right)+1\leq \sup_{i\in \omega}\rho\left(T_{(G_{n(i)}\cap U,D)}(G_{n(i)})\right)+1
\]
implying equality holds.
\end{rmk}

We define the \textbf{decomposition rank} of $G\in \Es$ to be
\[
\xi(G):=\rho(T_{(U,D)}(G))
\] 
for some $U\in \Uc(G)$ and $D$ a set of selector functions for $F(G)$. By Proposition~\rm\ref{prop:drank}, $\xi(G)$ is independent of the choice of $U$ and $D$. \par

\indent Applying Proposition~\rm\ref{prop:drank} again, we see the decomposition rank is indeed a group invariant.

\begin{cor}\label{cor:xi_monotone}
If $G\in \Es$, $H$ is a t.d.l.c.s.c. group, and $\psi:H\rightarrow G$ is a continuous, injective homomorphism, then $\xi(H)\sleq \xi(G)$. In particular, if $G\in \Es$ and $G\simeq H$, then $\xi(G)=\xi(H)$.
\end{cor}

\indent We note another important feature of the decomposition rank.

\begin{obs}\label{obs:res_dis_ele}
Suppose $G$ is a non-trivial t.d.l.c.s.c. group and $(O_i)_{i\in \omega}$ is an increasing exhaustion of $G$ by open subgroups. If each $O_i$ is residually discrete, then $G\in \Es$ with $\xi(G)=2$, and any decomposition tree for $G$ consists of the root along with its children.
\end{obs}

\indent We conclude this subsection by showing the decomposition rank may be computed by taking \textit{any} countable increasing exhaustion by compactly generated open subgroups. This shows we need not find selector functions to compute the decomposition rank.

\begin{lem}\label{lem:xi_union} 
Suppose $G$ is a non-trivial elementary group. Then,
\begin{enumerate}[(1)]
\item If $G=\bigcup_{i\in \omega}O_i$ with $(O_i)_{i\in \omega}$ an $\subseteq$-increasing sequence of compactly generated open subgroups of $G$, then $\xi(G)=\sup_{i\in \omega}\xi(\Res{}(O_i))+1$. In particular, $\xi(\Res{}(H))<\xi(G)$ for any compactly generated $H\sleq G$. 
\item If $G$ is compactly generated, then $\xi(G)=\xi(\Res{}(G))+1$.
\end{enumerate}
\end{lem}

\begin{proof} For $(1)$, fix $U\in \Uc(G)$ and $D$ selector functions for $F(G)$. For each $i$, there is $n(i)$ such that $O_i\sleq R_{n(i)}^{(U,D)}(G)$ since $O_i$ is compactly generated. Therefore, 
\[
\Res{}(O_i)\sleq \Res{}\left(R_{n(i)}^{(U,D)}(G)\right)=G_{n(i)},
\]
and Corollary~\rm\ref{cor:xi_monotone} implies $\xi(\Res{}(O_i))\sleq \xi(G_{n(i)})$. We conclude that
\[
\sup_{i\in \omega}\xi(\Res{}(O_i))+1\sleq \sup_{i\in \omega}\xi(G_i)+1=\xi(G).
\]

\indent On the other hand, $(O_i)_{i\in \omega}$ is an exhaustion of $G$ by open subgroups, so for each $k$, there is $n(k)$ with $R_k^{(U,D)}(G)\sleq O_{n(k)}$. Therefore, $G_k\sleq \Res{}(O_{n(k)})$, and applying Corollary~\rm\ref{cor:xi_monotone} again,
\[
\xi(G)=\sup_{k\in \omega}\xi(G_k)+1\sleq \sup_{i\in \omega}\xi(\Res{}(O_{i}))+1.
\]
Hence, $\xi(G)=\sup_{i\in \omega}\xi(\Res{}(O_i))+1$ as required.\par

\indent For $(2)$, suppose $G$ is a non-trivial compactly generated elementary group and let $(O_i)_{i\in \omega}$ be an $\subseteq$-increasing exhaustion of $G$ by compactly generated open subgroups. Observe that $(\Res{}(O_i))_{i\in\omega}$ also forms an $\subseteq$-increasing sequence, so in view of Corollary~\rm\ref{cor:xi_monotone}, $(\xi(\Res{}(O_i)))_{i\in \omega}$ is a monotone increasing sequence of ordinals. Since $G$ is compactly generated, there is $N$ such that $O_i=G$ for all $i\sgeq N$, hence $\Res{}(O_i)=\Res{}(G)$ for all $i\sgeq N$. We conclude that $(\xi(\Res{}(O_i)))_{i\in \omega}$ is a monotone increasing sequence eventually equal to $\xi(\Res{}(G))$, and therefore, $\sup_{i\in \omega}\xi(\Res{}(O_i))=\xi(\Res{}(G))$. Applying part $(1)$, $\xi(G)=\xi(\Res{}(G))+1$.
\end{proof}

\indent Section~\rm\ref{sec:examples} applies Lemma~\rm\ref{lem:xi_union} to compute the decomposition rank of examples; it is worthwhile to look ahead to these examples to gain some intuition for the decomposition rank.

\begin{rmk} 
The decomposition rank uses the second countability assumption. Second countability ensures the existence of selector functions, which allow us to build the decomposition trees in a uniform way. The decomposition trees are also countably branching via second countability.
\end{rmk}

\indent We observe a consequence of Lemma~\rm\ref{lem:xi_union}. This consequence, while obvious from the statement of the lemma, holds because well-founded decomposition trees characterize elementary groups. 

\begin{cor}\label{cor:residual_ele}
If $G$ is non-trivial, compactly generated, and elementary, then $\Res{}(G)\lneq G$.
\end{cor}

Non-trivial compactly generated groups $G$ so that $G=\Res{}(G)$ therefore cannot be elementary; to build intuition, we encourage the reader to consider how this makes the decomposition trees ill-founded. One naturally questions if this is the \textit{only} barrier to being elementary. More precisely,
\begin{quest}
Suppose $G$ is a t.d.l.c.s.c. group so that $\Res{}(H)\lneq H$ for any non-trivial compactly generated closed subgroup $H\sleq G$. Then is $G$ elementary?
\end{quest}

\subsection{Properties of the decomposition rank}
We conclude this section by exploring the decomposition rank further. We shall see the decomposition rank is better behaved than the construction rank but, surprisingly, the two ranks are closely related.\par

\indent We begin with a general lemma regarding SIN groups.

\begin{lem}\label{lem:SIN}
If $G$ is a compactly generated t.d.l.c.s.c. group and $N\trianglelefteq_{cc} G$ is a SIN group, then $G$ is a SIN group.
\end{lem}

\begin{proof}
Fix $U\in \Uc(G)$ and form the subgroup $UN$. Since $N$ is a SIN group, we may find $W\in \Uc(N)$ with $W\sleq U$ and $W\trianglelefteq N$. On the other hand, $W\sleq_oN\cap U$, so there is $V\trianglelefteq_o U$ with $V\cap N\sleq_o W$. We now have that
\[
V\cap N \sleq \bigcap_{un\in UN}unWn^{-1}u^{-1}=:J,
\]
$J\trianglelefteq UN$, and $J\in \Uc(N)$. \par

\indent Since $N$ is cocompact in $G$, $UN$ has finite index, so $N_G(J)$ has finite index. Letting $g_1,\dots,g_n$ list left coset representatives for $N_G(J)$ in $G$, we see that
\[
\bigcap_{g\in G}gJg^{-1}=\bigcap_{i=1}^ng_iJg_i^{-1}.
\]
Defining $K:=\bigcap_{g\in G}gJg^{-1}$, it follows that $K\in \Uc(N)$. \par

\indent Passing to $G/K$, $\pi(N)\trianglelefteq G/K$ is discrete where $\pi:G\rightarrow G/K$ is the usual projection. The subgroup $N$ is compactly generated since cocompact in a compactly generated group, hence the subgroup $\pi(N)$ is finitely generated. It follows that $\pi(N)$ has an open centralizer in $G/K$. Say $Q\sleq_o \pi(U)$ centralizes $\pi(N)$. Clearly, $Q\trianglelefteq Q\pi(N)$, and using that $\pi(N)$ is cocompact in $G/K$, we additionally see that $Q\pi(N)$ has finite index in $G/K$. Just as in the previous paragraph, there is $L\sleq_oQ$ with $L\trianglelefteq G/K$. It now follows that $\pi^{-1}(L)$ is an open normal subgroup of $G$ contained in $U$.\par

\indent We conclude that inside every compact open subgroup $U$ of $G$, we may find a compact open normal subgroup of $G$. That is to say, $G$ is a SIN group.
\end{proof}

\begin{lem}\label{lem:xi_quot} 
If $G$ is a t.d.l.c.s.c. group, $H,L\trianglelefteq G$, and $H\in \Es$, then $\ol{HL}/L\in \Es$ and $\xi\left(\ol{HL}/L\right)\sleq \xi(H)$.
\end{lem}

\begin{proof} We induct on $\xi(H)$ for the lemma. As the base case is obvious, suppose $\xi(H)=\beta+1$. Fix $U\in \Uc(\ol{HL})$, $V\in \Uc(H)$, and $D$ selector functions for $F(H)$. Put
\[
O_i:=\grp{V^U,d_0(H)^U,\dots, d_{i}(H)^U}
\]
with $d_0,\dots,d_{i}\in D$. \par

\indent For each $i$, the set $UO_i$ is a compactly generated open subgroup of $\ol{HL}$ since $O_i$ is normalized by $U$. Put $L_i:=L\cap UO_i\trianglelefteq UO_i$ and observe that since the discrete residual is a characteristic subgroup, it is also the case that $\Res{}(O_i)\trianglelefteq UO_i$.

\begin{claim*}
$\Res{}(UO_i/L_i)=\ol{\Res{}(O_i)L_i}/L_i$.
\end{claim*}
\begin{proof}[of Claim]
The group $O_i/\Res{}(O_i)$ is a SIN group, and $O_i/\Res{}(O_i)\trianglelefteq_{cc} UO_i/\Res{}(O_i)$. Lemma~\rm\ref{lem:SIN} therefore implies $UO_i/\Res{}(O_i)$ is a SIN group. Now
\[
\left(UO_i/L_i\right)/\left(\ol{\Res{}(O_i)L_i}/L_i\right)\simeq UO_i/\ol{\Res{}(O_i)L_i},
\]
and since the group on the right is a quotient of the SIN group $UO_i/\Res{}(O_i)$, we infer that $\left(UO_i/L_i\right)/\left(\ol{\Res{}(O_i)L_i}\right)$ is residually discrete. Therefore,
\[
\Res{}\left(UO_i/L_i\right)\sleq \ol{\Res{}(O_i)L_i}/L_i.
\]

\indent On the other hand, letting $\pi:UO_i\rightarrow UO_i/L_i$ be the usual projection, $\Res{}(O_i)\sleq \pi^{-1}(K)$ for any open normal subgroup $K$ of $UO_i/L_i$. Thus, 
\[
\ol{\Res{}(O_i)L_i}/L_i=\ol{\pi(\Res{}(O_i))}\sleq \Res{}\left(UO_i/L_i\right),
\]
and we conclude that $\Res{}(UO_i/L_i)=\ol{\Res{}(O_i)L_i}/L_i$.
\end{proof}

\indent The group $O_i$ is a compactly generated subgroup of $H$, so Lemma~\rm\ref{lem:xi_union} implies $\xi(\Res{}(O_i))\sleq \beta$. We may thus apply the claim along with the induction hypothesis to conclude that $\Res{}(UO_i/L_i)=\ol{\Res{}(O_i)L_i}/L_i\in \Es$ with
\[
\xi\left(\Res{}(UO_i/L_i)\right)=\xi\left(\ol{\Res{}(O_i)L_i}/L_i\right)\sleq \xi\left(\Res{}(O_i)\right)\sleq\beta.
\]
The quotient $UO_i/L_i$ is thus elementary-by-SIN and, therefore, elementary. Furthermore, the discrete residual of $UO_i/L_i$ has rank at most $\beta$.\par

\indent We now see that
\[
\ol{HL}/L=\bigcup_{i\in \omega}UO_iL/L
\]
with each $UO_iL/L\simeq UO_i/L_i$ an elementary compactly generated open subgroup, so $\ol{HL}/L\in \Es$. Furthermore, by the previous paragraph, the decomposition rank of the discrete residual of each $UO_iL/L$ is bounded by $\beta$. Lemma~\rm\ref{lem:xi_union} therefore implies $\xi(\ol{HL}/L)\sleq \beta+1=\xi(H)$, and the induction is complete.
\end{proof}

\begin{thm}\label{thm:xi_monotone_quot} 
If $G\in \Es$ and $L\trianglelefteq G$, then $\xi(G/L)\sleq \xi(G)$.
\end{thm}
\begin{proof} 
We induct on $\xi(G)$. As the base case is obvious, suppose $\xi(G)=\beta +1$. Let $(O_i)_{i\in \omega}$ be an $\subseteq$-increasing sequence of compactly generated open subgroups of $G$ with $G=\bigcup_{i\in\omega}O_i$ and for each $i$, put $L_i:=L\cap O_i$. As in the claim in the previous proof, it follows that
\[
\Res{}\left(O_i/L_i\right)=\ol{\Res{}(O_i)L_i}/L_i.
\]
In view of Lemma~\rm\ref{lem:xi_quot}, we conclude that $O_i/L_i\in \Es$ with
\[
\xi\left(\Res{}\left(O_i/L_i\right)\right)=\xi\left(\overline{\Res{}(O_i)L_i}/L_i\right)\sleq \xi(\Res{}(O_i))\sleq \beta
\]
where the right-most inequality follows from Lemma~\rm\ref{lem:xi_union}.\par

\indent We also have that
\[
G/L=\bigcup_{i\in \omega}O_iL/L
\]
with each $O_iL/L\simeq O_i/L_i$ compactly generated, open, and elementary. Thus, $G/L\in \Es$, and since $\xi\left(\Res{}\left(O_i/L_i\right)\right)\sleq \beta$ for each $i$, Lemma~\rm\ref{lem:xi_union} implies $\xi(G/L)\sleq \beta +1$ completing the induction.
\end{proof}

\begin{prop}\label{prop:rk_xi}
If $G\in \Es$, then $\rk(G)\sleq 3\xi(G)$ and $\xi(G)\sleq \rk(G)+2$.
\end{prop}
\begin{proof} For the first claim, we induct on $\xi(G)$. The base case, $\xi(G)=1$, implies $G=\{1\}$; hence, $\rk(G)=0$ and the proposition holds. Suppose $\xi(G)=\beta+1$. Let $(O_i)_{i\in \omega}$ be an $\subseteq$-increasing sequence of compactly generated open subgroups of $G$ with $G=\bigcup_{i\in\omega}O_i$. By Lemma~\rm\ref{lem:xi_union}, $\xi(\Res{}(O_i))\sleq \beta$, so the induction hypothesis implies $\rk(\Res{}(O_i))\sleq 3\beta$. Since $O_i/\Res{}(O_i)$ is a SIN group, it follows that
\[
\rk\left(O_i\right)\sleq 3\beta+2.
\]
We conclude that $\rk(G)\sleq 3(\beta+1)=3\xi(G)$ finishing the induction.\par

\indent For the second claim, we induct on $\rk(G)$. For the base case, $G$ is either profinite or discrete, so $\xi(G)\sleq 2$. Suppose $\rk(G)=\alpha+1$. Consider first the case $\rk(G)$ is given by an increasing union. Say $G=\bigcup_{i\in \omega}O_i$ with $(O_i)_{i\in \omega}$ an $\subseteq$-increasing sequence of open subgroups such that $\rk(O_i)\sleq \alpha$ for each $i$; without loss of generality, we may take the $O_i$ to be compactly generated. By the induction hypothesis, $\xi(O_i)\sleq \alpha+2$, so Corollary~\rm\ref{cor:xi_monotone} implies $\xi(\Res{}(O_i))\sleq \alpha+2$. Lemma~\rm\ref{lem:xi_union} now implies $\xi(G)\sleq \alpha +3$ as required.\par

\indent Suppose $\rk(G)$ is given by a group extension. Say $H\trianglelefteq G$ is such that $\rk(H)\sleq \alpha$ and $G/H$ is either compact or discrete. Write $G=\bigcup_{i\in \omega}O_i$ with $(O_i)_{i\in \omega}$ an $\subseteq$-increasing sequence of compactly generated open subgroups of $G$. Since $O_i/H\cap O_i$ is residually discrete, $\Res{}(O_i)\sleq H\cap O_i$. In view of Corollary~\rm\ref{cor:xi_monotone} and the induction hypothesis, 
\[
\xi(\Res{}(O_i))\sleq \xi(H)\sleq \alpha+2.
\]
Applying Lemma~\rm\ref{lem:xi_union}, $\xi(G)\sleq \alpha +3$, and the induction is complete.
\end{proof}

As a corollary, we obtain a bound on the construction rank of a quotient.

\begin{cor}\label{cor:rk_bound_quot}
If $G\in \Es$ and $L\trianglelefteq G$, then $\rk(G/L)\sleq 3(\rk(G)+2)$.
\end{cor}
\begin{proof}
Via Proposition~\rm\ref{prop:rk_xi}, $\rk(G/L)\sleq 3\xi(G/L)$. Theorem~\rm\ref{thm:xi_monotone_quot} implies $3\xi(G/L)\sleq 3\xi(G)$. Applying Proposition~\rm\ref{prop:rk_xi} again, we conclude that $3\xi(G)\sleq 3(\rk(G)+2)$ proving the corollary.
\end{proof}

\begin{rmk} 
While more technical to define, the decomposition rank is the more useful rank. It has better regularity properties and may be computed algorithmically. 
\end{rmk}

\section{A Further Permanence Property}\label{sec:furtherperm}

\subsection{Preliminaries} Our first preliminary result is well known; we include a proof for completeness.
\begin{prop}[(Folklore)]
If $G$ is a t.d.l.c.s.c. group and $g\in G$ has a countable conjugacy class, then $C_G(g)$ is open.
\end{prop}
\begin{proof} 
Since $G/C_G(g)\leftrightarrow g^G$, we see that $|G:C_G(g)|=\aleph_0$. Fixing $(g_i)_{i\in \omega}$ coset representatives for $C_G(g)$ in $G$, we thus have that
\[
G=\bigcup_{i\in\omega}g_i C_G(g).
\]
The Baire category theorem now implies there is some $i\in \omega$ for which $g_iC_G(g)$ is non-meagre, and it follows that $C_G(g)$ is non-meagre. As $C_G(g)$ is closed, $C_G(g)$ has non-empty interior and, therefore, is open.
\end{proof}

We require two additional facts from the literature.
\begin{thm}[({\cite[Theorem (9.10)]{K95}})]\label{thm:automatic} Let $G$ and $H$ be topological groups and $\phi:G\rightarrow H$ a homomorphism. If $G$ is Baire, $H$ is separable, and $\phi$ is Baire measurable, then $\phi$ is continuous.
\end{thm}

\begin{thm}[(Bergman, Lenstra {\cite[Theorem 3]{BL89}})]\label{thm:BL}
Let $G$ be a group and $H$ a subgroup. Then the following are equivalent:
\begin{enumerate}[(1)]
\item The set of indices  $\{|H:H\cap gHg^{-1}|:g\in G\}$ has a finite upper bound.
\item There is $K\trianglelefteq G$ such that $K\sim_c H$.
\end{enumerate}
\end{thm}

\subsection{The permanence property}
Suppose $H$ and $G$ are t.d.l.c.s.c. groups and $\psi:H\rightarrow G$ is an injective, continuous homomorphism such that $\psi(H)$ is normal and dense in $G$. 

\begin{lem}\label{lem:perm_2_comm}
For each $W\in \Uc(H)$, $Comm_G(\psi(W))=G$.
\end{lem}
\begin{proof}
Fix $g\in G$ and let $\phi_g:H\rightarrow H$ be defined by $\phi_g(h)=\psi^{-1}(g\psi(h)g^{-1})$. Since $\phi_g$ is a composition of bijective homomorphisms, $\phi_g$ is a bijective homomorphism. For $O\subseteq H$ open, 
\[
\phi_g^{-1}(O)=\psi^{-1}(g^{-1}\psi(O)g),
\]
is a Borel set since a continuous, injective image of a Borel set is a Borel set \cite[(15.1)]{K95}. The homomorphism $\phi_g$ is thus a Borel map, and applying Theorem~\rm\ref{thm:automatic}, $\phi_g$ is continuous. The same argument applied to $\phi_{g^{-1}}=\phi_g^{-1}$ gives that $\phi^{-1}_g$ is also continuous, hence $\phi_g$ is an isomorphism of topological groups.\par

\indent Fixing $W\in \Uc(H)$, we now have that $\phi_g(W)\in \Uc(H)$, hence $W\sim_c\phi_g(W)$. We conclude that
\[
\psi(W)\sim_c\psi(\phi_g(W))=g\psi(W)g^{-1},
\]
and it follows that $Comm_G(\psi(W))=G$.
\end{proof}

\begin{lem}\label{lem:perm_2_norm}
For each $U\in \Uc(G)$, there is $W\in \Uc(H)$ with $\psi(W)\trianglelefteq U$.
\end{lem}
\begin{proof}
Fix $U\in \Uc(G)$. Since $\psi$ is continuous, we may find $V\in \Uc(H)$ such that $\psi(V)\sleq U$. For each $n\sgeq 1$, put
\[
\Omega_n:=\left\{g\in U\mid |\psi(V):\psi(V)\cap g\psi(V)g^{-1}|\sleq n\right\}.
\]
\begin{claim*}
For each $n\sgeq 1$, $\Omega_n$ is closed.
\end{claim*}
\begin{proof}[of Claim] Suppose $(g_i)_{i\in \omega}\subseteq \Omega_n$ converges to $g$. For each $i$, let $k_1^i,\dots, k_n^i$ be coset representatives, with possible repetition, for $\psi(V)\cap g_i\psi(V)g_i^{-1}$ in $\psi(V)$. We may assume $k_j^i\rightarrow k_j$ as $i\rightarrow \infty$ for each $1\sleq j\sleq n$ by passing to a subsequence.\par

\indent Consider $x\in \psi(V)$. For each $i$, there is $k^i_j$ and $y_i\in \psi(V)\cap g_i\psi(V)g_i^{-1}$ with $x=k^i_jy_i$. By passing to a subsequence, we may assume $j=j_0$ for all $i$. As $k^i_{j_0}\rightarrow k_{j_0}$, we conclude that $y_i\rightarrow y$ with $y\in \psi(V)\cap g\psi(V)g^{-1}$, so $x=k_{j_0}y$. It now follows that
\[
\psi(V)=\bigcup_{j=1}^nk_j\left(\psi(V)\cap g\psi(V)g^{-1}\right),
\] 
and $|\psi(V):\psi(V)\cap g\psi(V)g^{-1}|\sleq n$. Therefore, $g\in \Omega_n$.
\end{proof}

\indent Since $U=\bigcup_{n\sgeq 1} \Omega_n$, the Baire category theorem implies there is some $n$ such that $\Omega_n$ is non-meagre. In view of the claim, $\Omega_n$ has non-empty interior, and it follows there are $u_1,\dots, u_k\in U$ with $U=\bigcup_{i=1}^ku_i\Omega_n$. Let $m\sgeq n$ be such that $u_1,\dots, u_k\in \Omega_m$. We claim $U=\Omega_{mn}$.\par

\indent For $x\in U$, we may take $x=u_iw$ for some $1\sleq i\sleq k$ and $w\in \Omega_n$. Since $|\psi(V):\psi(V)\cap w\psi(V)w^{-1}|\sleq n$,
\[
|\psi(V)\cap u_i^{-1}\psi(V)u_i:\psi(V)\cap u_i^{-1}\psi(V)u_i\cap w\psi(V)w^{-1}|\sleq n.
\]
Conjugating with $u_i$, we obtain that
\[
|\psi(V)\cap u_i\psi(V)u_i^{-1}:\psi(V)\cap u_i\psi(V)u_i^{-1}\cap u_iw\psi(V)w^{-1}u_i^{-1}|\sleq n.
\]
It now follows that $|\psi(V):\psi(V)\cap u_i\psi(V)u_i^{-1}\cap x\psi(V)x^{-1}|\sleq mn$, and we infer that
\[
|\psi(V):\psi(V)\cap x\psi(V)x^{-1}|\sleq mn.
\]
Therefore, $U=\Omega_{mn}$. \par

\indent Theorem~\rm\ref{thm:BL} now gives $K\trianglelefteq U$ with $K\sim_c \psi(V)$. Form $\ol{K}$. Certainly, it remains the case that $|\psi(V):\ol{K}\cap \psi(V)|<\infty$. On the other hand, there are $k_1,\dots, k_n \in K $ with $
K\subseteq \bigcup_{i=1}^nk_i\psi(V)$. Since the latter set is closed, $\ol{K}\subseteq \bigcup_{i=1}^nk_i\psi(V)$, so 
\[
\ol{K}= \bigcup_{i=1}^nk_i\left(\psi(V)\cap \ol{K}\right).
\]
Therefore, $|\ol{K}:\ol{K}\cap \psi(V)|<\infty$, and $\ol{K}\sim_c \psi(V)$. \par

\indent Since $ \ol{K}\cap \psi(V)$ is open in $\ol{K}$, we may find $L\trianglelefteq_o U$ so that $\ol{K}\cap L \sleq_o\ol{K}\cap \psi(V)$. Since $ \ol{K}\cap\psi(V)$ is finite index in $\psi(V)$, $\ol{K}\cap L$ is also finite index in $\psi(V)$. The group $W:=\psi^{-1}\left(\ol{K}\cap L\right)$ is thus closed and finite index in $V$ and, therefore, is compact and open in $H$. On the other hand, $\psi(W)\trianglelefteq U$, so $W$ satisfies the lemma.
\end{proof}

\begin{thm}\label{thm:dense_normal} 
Suppose $H\in \Es$, $G$ is a t.d.l.c.s.c. group, and $\psi:H\rightarrow G$ is an injective, continuous homomorphism with $\psi(H)$ normal and dense in $G$. Then $G\in \Es$.
\end{thm}

\begin{proof}
Fix $U\in \Uc(G)$ and apply Lemma~\rm\ref{lem:perm_2_norm} to find $W\in \Uc(H)$ with $\psi(W)\trianglelefteq U$. Lemma~\rm\ref{lem:perm_2_comm} gives that $\psi(W)$ is commensurated, whereby
\[
\qci{G}{U/\psi(W)}:=\left\{g\in G\mid \exists\;V\in \Uc(G)\text{ such that } [g,V \cap U]\subseteq \psi(W)\right\}
\]
is a normal subgroup of $G$ via Proposition~\rm\ref{prop:quasicentralizer}. We observe that $\psi(H)\cap U\trianglelefteq U$ and $\psi(H)\cap U/\psi(W)$ is countable, so $\psi(H)\cap U/\psi(W)$ is quasi-central in $U/\psi(W)$. It follows that $\psi(H)\cap U\sleq \qci{G}{U/\psi(W)}$ and, therefore, 
\[
\ol{\psi(H)\cap \qci{G}{U/\psi(W)}}\trianglelefteq_oG.
\]
To prove the desired theorem, it thus suffices to show $\ol{\psi(H)\cap \qci{G}{U/\psi(W)}}\in \Es$, so we may assume, without loss of generality, that $\psi(H)\cap \qci{G}{U/\psi(W)}$ is dense in $G$. We may also assume $G$ is compactly generated since $G$ is an increasing union of compactly generated open subgroups.\par

\indent We induct on the decomposition rank of $H$ for the theorem. For the base case, $\xi(H)=1$, $H=\{1\}$, and the result is trivial. Suppose $H\in \Es$ with $\xi(H)=\beta+1$. By our reductions, $G$ is compactly generated; fix $X$ a compact generating set. Proposition~\rm\ref{prop:factor} gives a finite symmetric $A\subseteq G$ containing $1$ so that $X\subseteq AU$ and $UAU=AU$. We may take 
\[
A\subseteq \psi(H)\cap \qci{G}{U/\psi(W)}
\]
since $\psi(H)\cap \qci{G}{U/\psi(W)}$ is dense in $G$. Applying Proposition~\rm\ref{prop:factor} again, we have that $G=\grp{A}U$.\par

\indent Put $B:=\{\psi(W)a\mid a\in A\}$.

\begin{claim*} 
$N_G(B)$ is open.
\end{claim*}
\begin{proof}[of Claim]
Since $A\subseteq \qci{G}{U/\psi(W)}$ and is finite, there is $V\sleq_o U$ such that $[a,V]\subseteq \psi(W)$ for each $a\in A$. Taking $v\in V$ and $a\in A$, we have that
\[
v\psi(W)av^{-1}=\psi(W)vav^{-1},
\]
and since $ava^{-1}v^{-1}=[a,v]\in \psi(W)$, $\psi(W)vav^{-1}=\psi(W)a$. Hence, $V\sleq N_G(B)$, and $N_G(B)$ is open.
\end{proof}

In view of the claim, $\langle B \rangle\trianglelefteq \langle B\rangle V\sleq_o G$ where $V$ is as in the proof of the claim, so 
\[
\grp{B}\trianglelefteq \ol{\langle B \rangle}\trianglelefteq_{cc} \langle B\rangle V.
\]
Since $\psi^{-1}(B)$ is a compact open subset of $H$, $L:=\psi^{-1}(\grp{B})$ is a compactly generated open subgroup of $H$. Lemma~\rm\ref{lem:xi_union} implies $\xi(\Res{}(L))\sleq \beta$. As in the proof of Lemma~\rm\ref{lem:perm_2_comm}, for $g\in \grp{B}V$, the map $\phi_g:L\rightarrow L$ defined by $\phi_g(l):=\psi^{-1}(g\psi(l)g^{-1})$ is a topological group automorphism of $L$, hence $\phi_g(\Res{}(L))=\Res{}(L)$. We conclude that
\[
g\psi(\Res{}(L))g^{-1}=\psi\left(\phi_g(\Res{}(L))\right)=\psi(\Res{}(L)),
\]
and therefore, $\psi(\Res{}(L))\trianglelefteq \grp{B}V$. The group $\psi(\Res{}(L))$ is thus normal and dense in $\ol{\psi(\Res{}(L))}$, and $\xi(\Res{}(L))\sleq \beta$. The induction hypothesis thus implies $R:=\ol{\psi(\Res{}(L))}\in \Es$.\par

\indent The map $\psi$ induces an injective, continuous homomorphism 
\[
\chi:L/\psi^{-1}(R)\rightarrow \cgrp{B}/R
\]
with dense, normal image. Since a quotient of the SIN group $L/\Res{}(L)$, the group $L/\psi^{-1}(R)$ is a SIN group. Let $K\trianglelefteq L/\psi^{-1}(R)$ be a compact open normal subgroup. We now have that $\chi(K)\trianglelefteq \grp{B}R/R$ and since $\chi(K)$ is closed, $\chi(K)\trianglelefteq \cgrp{B}/R$. The image of $L/\psi^{-1}(R)$ in $\cgrp{B}/R/\chi(K)$ is then dense, normal, and countable. Hence, $QZ(\cgrp{B}/R/\chi(K))$ is dense in $\cgrp{B}/R/\chi(K)$, and 
\[
\SIN\left(\cgrp{B}/R/\chi(K)\right)=\cgrp{B}/R/\chi(K).
\]
In view of Theorem~\rm\ref{thm:closure_main}, it follows that $\cgrp{B}/R\in \Es$ and further, that $\langle B\rangle V\in \Es$. \par

\indent By choice of $B$ and $V$, $\langle B\rangle V$ is a finite index subgroup of $G=\grp{A}U$. Theorem~\rm\ref{thm:closure_main} implies the normal core of $\grp{B}V$ is elementary, whereby $G$ is elementary-by-finite and, thus, elementary. This completes the induction, and we conclude the theorem.
\end{proof}

\section{Examples of elementary groups}\label{sec:examples}

\subsection{First examples and non-examples}
\begin{prop}\label{prop:first_examples}
The following are elementary groups:
\begin{enumerate}[(1)]
\item T.d.l.c.s.c. SIN groups; in particular, abelian t.d.l.c.s.c. groups.
\item T.d.l.c.s.c. solvable groups.
\item Locally elliptic t.d.l.c.s.c. groups.
\item T.d.l.c.s.c. groups containing a compact open subgroup that has a dense quasi-centre. In particular, any t.d.l.c.s.c. group that contains $F^{\Nb}$ as a compact open subgroup for $F$ some finite group. 
\end{enumerate}
\end{prop}
\begin{proof}
$(1)$ and $(2)$ follow immediately since SIN groups are profinite-by-discrete and solvable groups are built via group extension from abelian groups. $(3)$ follows since locally elliptic t.d.l.c.s.c. groups are countable increasing unions of profinite groups.\par

\indent For $(4)$, suppose $G$ is a t.d.l.c.s.c. group and $U\in \Uc(G)$ has a dense quasi-centre. Since $QZ(U)\sleq QZ(G)\sleq \SIN(G)$, the SIN-core $\SIN(G)$ is open in $G$. Appealing to Proposition~\rm\ref{prop:SIN}, $\SIN(G)$ is a countable increasing union of SIN groups, so $\SIN(G)$ is also elementary. It now follows that $G$ is elementary.
\end{proof}

It is illuminating to compute the construction and decomposition ranks in the first two examples. Suppose $G$ is a $SIN$ group. Since $G$ admits compact open normal subgroups, $G$ is compact-by-discrete, hence $\rk(G)\sleq 1$. For the decomposition rank, if $G=\{1\}$, then $\xi(G)=1$. For $G$ non-trivial, Observation~\rm\ref{obs:res_dis_ele} implies $\xi(G)=2$ since $G$ is residually discrete.

\indent Things are more interesting for solvable groups. Suppose $G$ is $n\sgeq 1$-step solvable. That is to say the $n$-th term of the closed derived series is trivial; see Section~\rm\ref{sec:locsolv} for a precise definition.
\begin{claim*} 
$\rk(G)\sleq 2n-1$. 
\end{claim*}
\begin{proof}[of Claim]
We argue by induction on $n$. For the base case, $G$ is abelian and the above paragraph gives that $\rk(G)\sleq 1$. Suppose $G$ is $(n+1)$-step solvable. By the induction hypothesis, $\rk\left(\ol{[G,G]}\right)\sleq 2n-1$. On the other hand, let $\pi:G\rightarrow G/\ol{[G,G]}$ be the usual projection. Since $G/\ol{[G,G]}$ is abelian, there is $K\trianglelefteq G/\ol{[G,G]}$ compact and open, so $\pi^{-1}(K)$ is an extension of a compact group by an at most construction rank $2n-1$ group. We infer that $\rk(\pi^{-1}(K))\sleq 2n$. Since $G/\pi^{-1}(K)$ is discrete, we conclude that $\rk(G)\sleq 2n+1$ completing the induction.
\end{proof}

\indent We now compute an upper bound for the decomposition rank.

\begin{claim*} 
$\xi(G)\sleq n+1$.
\end{claim*}
\begin{proof}[of Claim] We argue by induction on $n$. The base case we have since SIN groups are residually discrete. Suppose $G$ is $(n+1)$-step solvable. Write $G$ as an increasing union of compactly generated open subgroups $O_n$. Since $G/\ol{[G,G]}$ is residually discrete, $\Res{}\left( O_n\right)\sleq \ol{[G,G]}$ for each $n$, and via the induction hypothesis and Corollary~\rm\ref{cor:xi_monotone}, $\xi(\Res{}\left( O_n\right))\sleq n+1$. Lemma~\rm\ref{lem:xi_union} now implies
\[
\xi(G)=\sup_{n\in \omega}\xi(\Res{}(O_n))+1\sleq n+2
\]
completing the induction.
\end{proof}

\indent Of course, we only computed upper bounds on the ranks in the above examples. Lower bounds require knowing more about the specific structure. For example, any of the groups above can be discrete; in such a case, either rank is at most two. We compute lower bounds later in this section. We shall see the decomposition rank is preferred for computing lower bounds.

\begin{rmk}
We have an algorithm for computing the decomposition rank: Take an exhaustion by compactly generated open subgroups; compute the decomposition rank of the discrete residual of each term of the exhaustion; and take the supremum of the ranks plus one.  The construction rank seems to be, somewhat at odds with the name, non-constructive to compute. There does not seem to be a way to identify whether or not the rank is given by an increasing union or by a group extension. Complicating things further, in the case the construction rank is given by group extension, it is not clear how to compute the correct normal subgroup that witnesses the construction rank. 
\end{rmk}

\indent We now consider non-examples. Recall there are many non-discrete compactly generated t.d.l.c.s.c. groups that are topologically simple. For example, let $\mc{T}_n$ denote the $n$-regular tree for any $n\sgeq 3$. By work of J. Tits \cite {Ti70}, there is an index two non-discrete compactly generated t.d.l.c.s.c. subgroup of $Aut(\mc{T}_n)$, denoted $Aut^+(\mc{T}_n)$, that is topologically simple. Alternatively, $PSL_n(\Qp)$ for $n\sgeq 3$ is a non-discrete compactly generated topologically simple t.d.l.c.s.c. group; cf. \cite{BHV08} \cite{D71}.

\begin{prop}\label{prop:ex_topsimple}
If a t.d.l.c.s.c. group $G$ is compactly generated, topologically simple, and elementary, then $G$ is discrete. In particular, $PSL_n(\Qp)$ and $Aut^+(\mc{T}_n)$ for $n\sgeq 3$ are non-elementary.
\end{prop}
 
\begin{proof}
If $G$ is trivial, we are done. Else $\Res{}(G)\lneq G$ via Corollary~\ref{cor:residual_ele}. Topological simplicity implies $\Res{}(G)=\{1\}$, and it follows that $\{1\}$ is open. Therefore, $G$ is discrete.
\end{proof}

\begin{rmk}
Although Proposition~\rm\ref{prop:ex_topsimple} shows there are no non-discrete compactly generated topologically simple groups in $\Es$, there are non-discrete \textit{non-compactly generated} topologically simple groups in $\Es$. For example, the topologically simple groups built by Willis \cite[Proposition 3.4]{Will07} are such groups. The construction ranks of Willis' examples are necessarily given by increasing union. These examples are indeed increasing unions of compact groups and as a corollary, have construction rank one and decomposition rank two.
\end{rmk}

\subsection{A family of elementary groups with decomposition rank unbounded below $\omega$}

\indent We construct a family of elementary groups with members of arbitrarily large finite decomposition rank. It will then follow via Proposition~\rm\ref{prop:rk_xi} that there are members with arbitrarily large finite construction rank. For this construction some notation is required: For $K\sleq L$, we use $\ngrp{K}_L$ to denote the normal subgroup generated by $K$ in $L$. When clear from context, we drop the subscript.\par

\indent Let $A_5$ be the alternating group on five letters; recall $A_5$ is a non-abelian finite simple group. Let $S$ denote the infinite four generated simple group built by G. Higman \cite{H51}. Form $H:=S^5\rtimes A_5$ where $A_5\acts \{0,1,2,3,4\}$ in the usual fashion and fix a transitive, free action of $H$ on $\Nb$. Observe that the normal subgroup generated by $A:=A_5$ in $H$ is $H$ itself. \par

\indent We inductively define compactly generated elementary groups $L_n$ and distinguish $K_n\in \Uc(L_n)$ so that $\ngrp{K_n}=L_n$.  For the base case, $n=1$, define $L_1:=H$ and $K_1:=A$. So $L_1$ is compactly generated, $K_1$ is a compact open subgroup of $L_1$, and $\ngrp{K_1}=L_1$. Suppose we have defined a compactly generated $L_n$ with a compact open subgroup $K_n$ so that $\ngrp{K_n}=L_n$. Let $(L^i_n)_{i\in \Nb}$ and $(K^i_n)_{i\in \Nb}$ list countably many copies of $L_n$ and $K_n$, respectively, and form $\bigoplus_{i\in \Nb}(L_n^i,K_n^i)$. Taking the previously fixed action of $H$ on $\Nb$, we see that $H\acts \bigoplus_{i\in \Nb}(L_n^i,K_n^i)$ by shift; see the discussion after Definition~\rm\ref{def:ldp}. We may thus form
\[
L_{n+1}:=\bigoplus_{i\in \Nb}(L_n^i,K_n^i)\rtimes H
\]
and put $K_{n+1}:=K_n^\Nb\rtimes A$. Certainly, $K_{n+1}$ is a compact open subgroup of $L_{n+1}$. Letting $X$ be a compact generating set for $L_n^0$ and $F$ be a finite generating set for $H$ in $L_{n+1}$, one verifies that $X\times \prod_{i>0}K_n^i\cup F$ is a compact generating set for $L_{n+1}$. It is easy to further verify that $\ngrp{K_{n+1}}_{L_{n+1}}=L_{n+1}$. This completes our inductive construction.

\begin{prop}\label{prop:unbd_fam} 
For each $n\sgeq 1$, $\xi(L_n)\sgeq n+1$.
\end{prop}

\begin{proof} We argue by induction on $n$. For the base case, $L_1=H$ is non-trivial and discrete. A fortiori, $L_1$ is residually discrete, hence $\xi(L_1)=2$ via Observation~\rm\ref{obs:res_dis_ele}.\par

\indent Suppose the induction hypothesis holds up to $n$ and consider $L_{n+1}$. We first compute $\Res{}(L_{n+1})$. Consider $O\trianglelefteq_o L_{n+1}$. Since $K_n^{\Nb}$ is a compact open subgroup of $L_{n+1}$, $O$ must contain 
\[
K_n^{(k,\infty)}:=\{f:\Nb\rightarrow K_n\mid f(0)=\dots =f(k)=1\}
\]
for some $k\in \Nb$. Since $H$ acts transitively on $\Nb$ and $O$ is normal, $O$ indeed contains $K_n^{\Nb}$. Recalling $\ngrp{K_n}_{L_n}=L_n$, we conclude that
\[
\bigoplus_{i\in \Nb}(L_n^i,K_n^i)=\ngrp{K_n^{\Nb}}_{L_{n+1}}\sleq O.
\]
It now follows that $\Res{}(L_{n+1})=\bigoplus_{i\in \Nb}(L_n^i,K_n^i)$.\par

\indent Lemma~\rm\ref{lem:xi_union} gives that $\xi(L_{n+1})= \xi(\Res{}(L_{n+1}))+1$ because $L_{n+1}$ is compactly generated. The group $L_n$ admits a continuous injection into $\Res{}(L_{n+1})$, so 
\[
\xi(\Res{}(L_{n+1}))\sgeq \xi(L_n)\sgeq n+1
\]
via Corollary~\rm\ref{cor:xi_monotone} and the induction hypothesis. We conclude that $\xi(L_{n+1})\sgeq n+2$, and the induction is complete.
\end{proof}

\indent The set $\{L_n\mid n\sgeq 1\}$ is thus a family of elementary groups with members of arbitrarily large finite decomposition rank. In view of Proposition~\rm\ref{prop:rk_xi}, we infer that $\rk(L_n)\sgeq n-1$ for each $n$, hence $\{L_n\mid n\sgeq 1\}$ is also a family of elementary groups with members of arbitrarily large finite construction rank. 

\begin{cor}
For $G:= \bigoplus_{n\sgeq 1}(L_n,K_n)$, $\xi(G)=\omega+1$. It follows that $\rk(G)=\omega+1$.
\end{cor}
\begin{proof}
For each $n\sgeq 1$, there is a continuous injection $L_n\hookrightarrow G$. Via Proposition~\rm\ref{prop:unbd_fam} and Corollary~\rm\ref{cor:xi_monotone}, $n+1\sleq \xi(G)$ for all $n\sgeq 1$, so $\omega \sleq \xi(G)$. Since the decomposition rank is always a successor ordinal or zero, we conclude that $\omega+1\sleq \xi(G)$. The converse inequality is an easy exercise.
\end{proof}
\indent The decomposition rank arises from well-founded trees $T\subseteq \wbaire$, hence it is always less than $\omega_1$, the first uncountable ordinal. However, it is unknown if $\omega_1$ is the least upper bound.

\begin{quest} 
Are there elementary groups of arbitrarily large decomposition rank below $\omega_1$?
\end{quest}

\begin{rmk}
Results for elementary amenable groups suggest a positive answer to the question; cf. \cite{OO13}. Alternatively, using bi-infinite iterated wreath products similar to M.G. Brin's construction in \cite{B05}, one can build elementary groups of decomposition rank $\omega+2$. The examples of decomposition rank $\omega+2$ are compelling as they are \textit{compactly generated} elementary groups with transfinite rank. It seems plausible, albeit difficult, this construction can be iterated to build groups of arbitrarily large rank below $\omega_1$. The decomposition rank $\omega+2$ examples will appear in a forthcoming paper joint with Colin Reid.  \par

\indent An answer in either direction to the question would be quite interesting. A negative answer would place strong restrictions on the class of elementary groups. A positive answer would imply there is no $SQ$-universal elementary group for the class of elementary groups. This would in particular imply there is no surjectively universal t.d.l.c.s.c. group for the class of t.d.l.c.s.c. groups answering a question of S. Gao and M. Xuan \cite{GM14}.\par
\end{rmk}

\section{Application 1: Structure theorems}\label{sec:structure}
For our first application, we consider elementary groups appearing as normal subgroups or as quotients of an arbitrary t.d.l.c.s.c. group to arrive at general structure theorems. 

\subsection{Preliminaries}
We require basic graph-theoretic notions. For a graph $\Gamma$, $V\Gamma$ denotes the vertices of $\Gamma$, and $E\Gamma$ denotes the edges. We take as a convention that edges connect distinct vertices. If $w\in V\Gamma$ is connected to $v\in V\Gamma$ by an edge, we say $w$ is a \textbf{neighbour} of $v$. In the case $\Gamma$ is a vertex transitive locally finite graph, the \textbf{degree} of $\Gamma$, denoted $\deg(\Gamma)$, is the number of neighbors of some (any) $v\in V\Gamma$. When a graph is connected, there is a metric given by the least length of an edge path. \par

\indent Graphs play an important role in the study of t.d.l.c. groups via an old result of H. Abels. 

\begin{thm} [(Abels \cite{A73})] \label{thm:cay_abels}
Let $G$ be a compactly generated t.d.l.c. group. Then there is a locally finite connected graph $\Gamma$ on which $G$ acts continuously and vertex transitively by graph automorphisms such that for all $v\in V\Gamma$, the stabilizer of $v$ in $G$ is compact and open. 
\end{thm}

\indent For $G$ a t.d.l.c. group, a connected locally finite graph on which $G$ acts continuously and vertex transitively by graph automorphisms with compact open stabilizers is called a \textbf{Cayley-Abels graph} for $G$. Theorem~\rm\ref{thm:cay_abels} shows every compactly generated t.d.l.c. group admits a Cayley-Abels graph; the converse also holds. We make one further remark: If $G$ is a compactly generated t.d.l.c. group and $\Gamma$ is a Cayley-Abels graph for $G$, then the kernel of $G\acts \Gamma$ is a compact normal subgroup. Therefore, in the case $G$ has no non-trivial compact normal subgroups, $G\acts \Gamma$ faithfully. We direct the reader to \cite{KM08} for a pleasant, self contained discussion of the Cayley-Abels graph; in \cite{KM08}, the Cayley-Abels graph is called the \textbf{rough Cayley graph}.\par

\indent Cayley-Abels graphs give rise to a finite invariant: For a compactly generated t.d.l.c.s.c. group $G$, the \textbf{degree} of $G$ is
\[
\deg(G):=\min\{\deg(\Gamma)\mid \Gamma\text{ is a Cayley-Abels graph for }G\}.
\]
We shall see this invariant plays an important role in the structure theory.\par

\indent If $G$ is a group and $X$ a set with a $G$-action, then a \textbf{$G$-congruence} $\sigma$ on $X$ is an equivalence relation on $X$ such that for all $g\in G$, $x\sim_{\sigma}y$ if and only if $g.x\sim_{\sigma}g.y$. The equivalence classes of $\sigma$ are called the \textbf{blocks} of $\sigma$. For $x\in X$, the block containing $x$ is denoted $x^{\sigma}$.\par

\indent We make use of $G$-congruences on Cayley-Abels graphs. In particular, suppose $G$ is a compactly generated t.d.l.c. group with $H\trianglelefteq G$ and suppose $\Gamma$ is a Cayley-Abels graph for $G$. The orbits of $H$ on $V\Gamma$ induce a $G$-congruence on $V\Gamma$; call this $G$-congruence $\sigma$. We may now define a quotient graph $\Gamma/\sigma$ by $V\Gamma/\sigma:=\{v^{\sigma}\mid v\in V\Gamma\}$ and
\[
E\Gamma/\sigma:=\{\{v^{\sigma},w^{\sigma}\}\mid v^{\sigma}\neq w^{\sigma}\text{ and }\exists v'\in v^{\sigma}\;w'\in w^{\sigma}\text{ such that }\{v,w\}\in E\Gamma\}
\]
This quotient graph has two very useful properties; the proofs are left to the reader.

\begin{obs}\label{obs:si}
\begin{enumerate}[(1)]
\item $\Gamma/\sigma$ is connected, locally finite, vertex transitive, and $\deg(\Gamma/\sigma)\sleq \deg(\Gamma)$.
\item $G/H\acts \Gamma/\sigma$ continuously and transitively by graph automorphisms. Additionally, for each $v^{\sigma}\in V\Gamma/\sigma$, the stabilizer of $v^{\sigma}$ is compact and open in $G/H$.
\end{enumerate}
\end{obs}

\indent We lastly require a powerful theorem of Caprace and Monod.

\begin{thm}[(Caprace, Monod {\cite[Theorem B]{CM11}})]\label{thm:min_norm} Let $G$ be a compactly generated t.d.l.c. group. Then one of the following holds:
\begin{enumerate}[(1)]
\item $G$ has an infinite discrete normal subgroup.
\item $G$ has a non-trivial locally elliptic radical.
\item $G$ has exactly $0<n<\infty$ many minimal non-trivial closed normal subgroups.
\end{enumerate}
\end{thm}

\begin{rmk} The statement of alternative $(2)$ in Theorem~\rm\ref{thm:min_norm} is weaker than the statement made in \cite[Theorem B]{CM11} that $G$ has a non-trivial compact normal subgroup. Examples show the stronger claim is false, however; the error occurs in the proof of \cite[Proposition 2.6]{CM11} on which \cite[Theorem B]{CM11} relies. Fortunately, upon replacing ``compact normal subgroup" by ``locally elliptic normal subgroup" in the statement of \cite[Proposition 2.6]{CM11}, it is an easy exercise to fix the proof given by Caprace and Monod and thus to prove Theorem~\rm\ref{thm:min_norm}. The author has been informed that Caprace and Monod are preparing a correction which implies Theorem~\rm\ref{thm:min_norm}.
\end{rmk}

\subsection{The elementary radical and the elementary residual}
\indent Let $G$ be a t.d.l.c.s.c. group and put 
\[
\mc{S}_{\Es}(G):=\{N\trianglelefteq G\mid  N \text{ is elementary}\}.
\]
We claim $\subseteq$-chains in $\mc{S}_{\Es}(G)$ admit upper bounds. Indeed, suppose $(N_{\alpha})_{\alpha<\lambda}$ is an $\subseteq$-increasing chain in $\mc{S}_{\Es}(G)$. Fix $U\in \Uc(G)$ and consider $(UN_{\alpha})_{\alpha<\lambda}$. Since $\bigcup_{\alpha<\lambda}UN_{\alpha}$ is open in $G$, it is a Lindel\"{o}f space when considered as a subspace, so there is a countable subcover $(UN_{\alpha_i})_{i\in \omega}$. Each $UN_{\alpha_i}$ is elementary, and therefore, $\bigcup_{i\in \omega}UN_{\alpha_i}\in \Es$. Applying Theorem~\rm\ref{thm:closure_main}, $\ol{\bigcup_{\alpha<\lambda}N_{\alpha}}\in \Es$. Since $\ol{\bigcup_{\alpha<\lambda}N_{\alpha}}\trianglelefteq G $, we conclude that $\ol{\bigcup_{\alpha<\lambda}N_{\alpha}}\in \mc{S}_{\Es}(G)$. Inclusion chains in $\mc{S}_{\Es}(G)$ thus have upper bounds, and Zorn's lemma implies $\mc{S}_{\Es}(G)$ has maximal elements. 

\begin{prop}
$\mc{S}_{\Es}(G)$ has a unique $\subseteq$-maximal element. 
\end{prop}
\begin{proof}
Suppose $M,N$ are two $\subseteq$-maximal elements. Form $\ol{MN}$ and consider $\ol{MN}/M\cap N$. It is easy to see $\ol{MN}/M\cap N$ is a quasi-product of $M/M\cap N$ and $N/M\cap N$, and these groups are elementary via Theorem~\rm\ref{thm:closure_main}. Two further applications of Theorem~\rm\ref{thm:closure_main} imply $\ol{MN}/M\cap N\in \Es$ and $\ol{MN}\in\Es$. We conclude that $M=\ol{MN}=N$, so $S_{\Es}(G)$ has a unique maximal element.
\end{proof}

We denote the unique maximal element of $\mc{S}_{\Es}(G)$ by $\Rad{\Es}(G)$. We call $\Rad{\Es}(G)$ the \textbf{elementary radical} of $G$. The elementary radical enjoys two useful properties:
\begin{obs}
\begin{enumerate}[(1)]
\item $\Rad{\Es}(G)$ is a closed characteristic subgroup of $G$ and contains every element of $\mc{S}_{\Es}(G)$. 
\item $\Rad{\Es}\left(G/\Rad{\Es}(G)\right)=\{1\}$.
\end{enumerate}
\end{obs}

\indent We now explore a notion dual to the elementary radical. Let $G$ be a t.d.l.c.s.c. group and put 
\[
\mc{Q}_{\Es}(G):=\{N\trianglelefteq G\mid G/N\text{ is elementary}\}.
\]
We claim $\subseteq$-decreasing chains in $\mc{Q}_{\Es}(G)$ have lower bounds. Indeed, suppose $(N_{\alpha})_{\alpha\in I}$ is an $\subseteq$-decreasing chain in $\mc{Q}_{\Es}(G)$ and put $N:=\bigcap_{\alpha\in I}N_{\alpha}$. The group $G/N$ is thus residually elementary. Via the permanence properties of $\Es$, $G/N$ is elementary, hence $N\in \mc{Q}_{\Es}(G)$. We conclude that $\subseteq$-decreasing chains in $\mc{Q}_{\Es}(G)$ have lower bounds, and applying Zorn's lemma, $\mc{Q}_{\Es}(G)$ has minimal elements.

\begin{prop}
$\mc{Q}_{\Es}(G)$ has a unique $\subseteq$-minimal element.
\end{prop}

\begin{proof} Suppose $M,N$ are $\subseteq$-minimal elements of $\mc{Q}_{\Es}(G)$. The usual projection map $ M/M\cap N \rightarrow G/N$ is a continuous, injective group homomorphism. Theorem~\rm\ref{thm:closure_main} thus implies $M/M\cap N$ is elementary, and since $(G/M \cap N)/(M/M\cap N)\simeq G/M$ and $G/M$ is elementary, a second application of Theorem~\rm\ref{thm:closure_main} implies $G/M\cap N$ is elementary. The minimality of $M$ and $N$ now gives that $M=M\cap N=N$ as required.
\end{proof}

\noindent We call the unique minimal element of $\mc{Q}_{\Es}(G)$ the \textbf{elementary residual} of $G$, denoted $\Res{\Es}(G)$. We note two important features of the elementary residual.
\begin{obs}
\begin{enumerate}[(1)]
\item $\Res{\Es}(G)$ is a closed characteristic subgroup of $G$ and is contained in every element of $\mc{Q}_{\Es}(G)$. 
\item $\Res{\Es}\left(\Res{\Es}(G)\right)=\Res{\Es}(G)$.
\end{enumerate}
\end{obs}

Collecting our results, we have proved the following theorem:
\begin{thm} Let $G$ be a t.d.l.c.s.c. group. Then
\begin{enumerate}[(1)]
\item There is a unique maximal closed normal subgroup $\Rad{\Es}(G)$ such that $\Rad{\Es}(G)$ is elementary. 
\item There is a unique minimal closed normal subgroup $\Res{\Es}(G)$ such that $G/\Res{\Es}(G)$ is elementary.
\end{enumerate}
\end{thm}

\subsection{The structure of compactly generated groups}
A key ingredient in the proof of the following theorem is the relationship between normal subgroups and $G$-congruences. The inspiration for this approach comes from a fascinating paper of Trofimov \cite{Tr03}.

\begin{thm}\label{thm:cg_dcomp}
Let $G$ be a compactly generated t.d.l.c.s.c. group. Then there is a finite series of closed characteristic subgroups 
\[
\{1\}=H_0\sleq H_1\sleq \dots \sleq H_{n}\sleq G
\]
such that 
\begin{enumerate}[(1)]
\item $G/H_{n}\in \Es$; 
\item for $0\sleq k \sleq n-1$, $(H_{k+1}/H_{k})/\Rad{\Es}(H_{k+1}/H_{k})$ is a quasi-product with $0<n_{k+1}<\infty$ many topologically characteristically simple non-elementary quasi-factors; and
\item $n\sleq \deg(G)$.
\end{enumerate}
\end{thm}

\begin{proof} 
Fix $\Gamma$ a Cayley-Abels graph for $G$ so that $\deg(G)=\deg(\Gamma)$. We inductively build a sequence of characteristic subgroups $\{1\}=H_0\sleq H_1\sleq \dots \sleq H_{i}$ so that condition $(2)$ of the theorem holds of successive terms and $\deg(\Gamma/\sigma_{j-1})>\deg(\Gamma/\sigma_{j})$ for $0<j\sleq i$ where $\sigma_j$ is the $G$-congruence induced by the orbits of $H_j$ on $V\Gamma$.\par

\indent  The base case is immediate with $H_0:=\{1\}$. Suppose we have defined $H_{k}$ and pass to $\tilde{G}:=G/H_{k}$. If $\tilde{G}\in \Es$, we stop. Else, let $\pi:G\rightarrow \tilde{G}/\Rad{\Es}(\tilde{G})$ be the usual projection. The group $\tilde{G}/\Rad{\Es}(\tilde{G})$ has trivial elementary radical, so $\tilde{G}/\Rad{\Es}(\tilde{G})$ has no non-trivial discrete normal subgroups and trivial locally elliptic radical. Theorem~\rm\ref{thm:min_norm} thus implies $\tilde{G}/\Rad{\Es}(\tilde{G})$ has exactly $0<n_{k+1}<\infty$ many minimal non-trivial normal subgroups $N_1,\dots, N_{n_{k+1}}$. These subgroups must be non-elementary and topologically characteristically simple.\par

\indent  We claim $\cgrp{N_1,\dots, N_{n_{k+1}}}$ is a quasi-product. Let $m:N_1\times\dots \times N_{n_{k+1}}\rightarrow \cgrp{N_1,\dots, N_{n_{k+1}}}$ be the multiplication map. The minimality of the $N_i$ gives that the $N_i$ pairwise centralize each other, and thereby, the multiplication map is also a homomorphism. It follows immediately that the image of $m$ is dense. To check that $m$ is injective, take $(x_1,\dots,x_{n_{k+1}})\in \ker(m)$ and suppose some $x_i$ is non-trivial. Without loss of generality, we may assume $i=1$. We now see that
\[
1\neq x_1\in N_1\cap \cgrp{N_2,\dots,N_{n_{k+1}}},
\] 
and by the minimality of $N_1$, $N_1\sleq \cgrp{N_2,\dots,N_{n_{k+1}}}$. The group $N_1$ commutes with the generators of $\cgrp{N_2,\dots,N_{n_{k+1}}}$, so $N_1$ is central in $\cgrp{N_2,\dots,N_{n_{k+1}}}$. Abelian groups, however, are elementary making $N_1$ a non-trivial elementary normal subgroup. This contradicts that $\tilde{G}/\Rad{\Es}(\tilde{G})$ has trivial elementary radical. The map $m$ therefore has trivial kernel and is injective. We conclude that $\cgrp{N_1,\dots, N_{n_{k+1}}}$ is a quasi-product.\par

\indent Putting
\[
H_{k+1}:=\pi^{-1}\left(\cgrp{N_1,\dots,N_{n_{k+1}}}\right),
\] 
we have that the successive terms of $H_0,\dots, H_k,H_{k+1}$ satisfy $(2)$ of the theorem. It remains to show $\deg(\Gamma/\sigma_{k})>\deg(\Gamma/\sigma_{k+1})$. Put $R:=\pi^{-1}(\Rad{\Es}(G/H_k))$ and let $\tau$ be the $G$-congruence on $V\Gamma$ corresponding to $R$. Since $H_k\sleq R\sleq H_{k+1}$, we have that 
\[
\deg(\Gamma/\sigma_k)\sgeq \deg(\Gamma/\tau)\sgeq \deg (\Gamma/\sigma_{k+1}),
\]
so it is enough to show $\deg(\Gamma/\tau)>\deg(\Gamma/\sigma_{k+1})$. \par

\indent Let $\pi:G\rightarrow G/R$ be the usual projection, put $\tilde{H}_{k+1}:=\pi(H_{k+1})$, and fix $v^{\tau}\in V\Gamma/\tau$. The group $G/R$ has no non-trivial elementary normal subgroups, so $\tilde{H}_{k+1}$ is non-discrete. We may therefore find $\tilde{h}\in (\tilde{H}_{k+1})_{(v^{\tau})}\setminus \{1\}$ as $(G/R)_{(v^{\tau})}$ is open where $(G/R)_{(v^{\tau})}$ is the stabilizer of $v^{\tau}$ in $G/R$. The group $G/R$ also has no non-trivial compact normal subgroup, hence Observation~\rm\ref{obs:si} implies $G/R\acts \Gamma/\tau$ faithfully. There is thus $w^{\tau}$ nearest to $v^{\tau}$ such that $\tilde{h}.w^{\tau}\neq w^{\tau}$. It follows there is a neighbour $u^{\tau}$ of $w^{\tau}$ such that $\tilde{h}.u^{\tau}=u^{\tau}$. Fixing $h\in H_{k+1}$ with $\pi(h)=\tilde{h}$, in the quotient graph $\Gamma/\tau$ the vertex $u^{\tau}$ has $\tilde{h}.w^{\tau}=(h.w)^{\tau}$ and $w^{\tau}$ as distinct neighbors. In the quotient graph $\Gamma/\sigma_{k+1}$, however, $(h.w)^{\sigma_{k+1}}=w^{\sigma_{k+1}}$, so $u^{\sigma_{k+1}}$ has at least one less neighbour. Hence, $\deg(\Gamma/\tau)>\deg(\Gamma/\sigma_{k+1})$, and we have verified the inductive claim.\par

\indent Our inductive construction procedure must halt at $n\sleq \deg(G)$. The resulting series $\{1\}=H_0\sleq H_1\sleq \dots \sleq H_{n}\sleq G$ satisfies the theorem.
\end{proof}

\begin{cor}
Let $G$ be a compactly generated t.d.l.c. group. Then there is a finite series of closed normal subgroups 
\[
\{1\}\sleq H_0\sleq H_1\sleq \dots \sleq H_{n}\sleq G
\]
such that 
\begin{enumerate}[(1)]
\item $H_0$ is compact and $G/H_0$ is second countable;
\item $G/H_{n}\in \Es$; 
\item for $0\sleq k \sleq n-1$, $(H_{k+1}/H_{k})/\Rad{\Es}(H_{k+1}/H_{k})$ is a quasi-product with $0<n_{k+1}<\infty$ many topologically characteristically simple non-elementary quasi-factors; and
\item $n\sleq \deg(G)$.
\end{enumerate}
\end{cor}
\begin{proof}
Via \cite[(8.7)]{HR79}, there is $H_0\trianglelefteq G$ such that $H_0$ is compact and $G/H_0$ is second countable. We now apply Theorem~\rm\ref{thm:cg_dcomp} to $G/H_0$ to find $\tilde{H}_1\sleq\dots \sleq \tilde{H}_n\sleq G/H_0$ with $n\sleq \deg(G/H_0)\sleq \deg(G)$. Letting $\pi:G\rightarrow G/H_0$ be the usual projection and putting $H_i:=\pi^{-1}(\tilde{H}_i)$ for $1\sleq i\sleq n$, the series $H_0\sleq H_1\sleq\dots \sleq H_n$ satisfies the theorem. 
\end{proof}

\begin{rmk} 
The length of the characteristic series given by Theorem~\rm\ref{thm:cg_dcomp} can be arbitrarily long. Additionally, the quasi-factors need not be topologically simple or compactly generated. Example~\rm\ref{ex:cg_dcomp} demonstrates these phenomena. In the case of l.c.s.c. $p$-adic Lie groups, however, it is always the case that $n\sleq 1$ and that the quasi-factors are compactly generated and topologically simple; see \cite{W_2_14}.
\end{rmk}

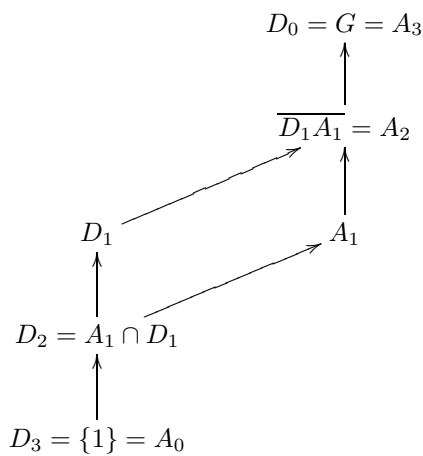
\begin{figure}[H]
\[\xymatrix{
& D_0=G=A_3 & \\
& \ol{D_1A_1}=A_2\ar@{->}[u] &\\
D_1\ar@{->}[ur] &  A_1\ar@{->}[u]&\\
 D_2=A_1\cap D_1 \ar@{->}[u]\ar@{->}[ur] && \\
 D_3=\{1\}=A_0 \ar@{->}[u] &&
 }\]
\caption{The inclusions between the descending and ascending elementary series}
\end{figure}

\subsection{The descending and ascending elementary series}
In the non-compactly generated case, we obtain weaker, nonetheless useful structure theorems. Via the elementary radical and elementary residual, we build two characteristic subgroup series in a t.d.l.c.s.c. group. Figure 1 indicates the relationship between these series; the arrows indicate inclusion. The diagram follows from the permanence properties of $\Es$.

\begin{defn} 
For a t.d.l.c.s.c. group $G$ the \textbf{descending elementary series} is defined by $D_0:=G$, $D_1:=\Res{\Es}(D_0)$, $D_2:=\Rad{\Es}(D_1)$, and $D_3:=\{1\}$.
\end{defn}

\begin{defn} For a t.d.l.c.s.c. group $G$ the \textbf{ascending elementary series} is defined by $A_0:=\{1\}$, $A_1:=\Rad{\Es}(G)$, $A_2:=\pi^{-1}(\Res{\Es}(G/A_1))$ where $\pi:G\rightarrow G/A_1$ is the usual projection, and $A_3:=G$.
\end{defn}

\begin{defn}
We say a t.d.l.c.s.c. group is \textbf{elementary-free} if it has no non-trivial elementary normal subgroups and no non-trivial elementary quotients.
\end{defn}

\begin{thm}\label{thm:des_E_series}
Let $G$ be a t.d.l.c.s.c. group. Then the descending elementary series 
\[
G\sgeq D_1\sgeq D_2\sgeq \{1\}
\]
is a series of closed characteristic subgroups with $G/D_1$ elementary, $D_1/D_2$ elementary-free, and $D_2$ elementary.
\end{thm}

\begin{proof}
It is immediate that $D_1$, $D_2$ are closed and characteristic and that $G/D_1$ and $D_2$ are elementary. For the remaining claim, it suffices to show $\Rad{\Es}(D_1/D_2)=\{1\}$ and $\Res{\Es}(D_1/D_2)=D_1/D_2$. The former is immediate since $D_2=\Rad{\Es}(D_1)$.\par

\indent For the case of $\Res{\Es}(D_1/D_2)$, let $\pi:D_1\rightarrow D_1/D_2$ be the usual projection and put $H:=\pi^{-1}(\Res{\Es}(D_1/D_2))$. We see $H$ is the intersection of all normal subgroups of $D_1$ extending $D_2$ whose quotient is elementary. Since $D_2$ is characteristic in $D_1$, $H$ is characteristic in $D_1$ and, thus, is normal in $G$. We now have a short exact sequence of topological groups 
\[
1\rightarrow D_1/H\rightarrow G/H\rightarrow G/D_1\rightarrow 1.
\]
This sequence shows that $G/H$ is a group extension of elementary groups and, therefore, elementary. Since $D_1$ is minimal in the collection of normal subgroups whose quotients are elementary, we conclude that $D_1=H$. Hence, $\Res{\Es}(D_1/D_2)=D_1/D_2$ as required.
\end{proof}

\begin{thm}\label{thm:acs_E_series}Let $G$ be a t.d.l.c.s.c. group. Then the ascending elementary series 
\[
\{1\}\sleq A_1\sleq A_2\sleq G
\]
is a series of closed characteristic subgroups with $A_1$ elementary, $A_2/A_1$ elementary-free, and $G/A_2$ elementary.
\end{thm}

\begin{proof}
It is immediate that $A_1$, $A_2$ are closed and characteristic and that $A_1$ and $G/A_2$ are elementary. For the remaining claim, it again suffices to show $\Rad{\Es}(A_2/A_1)=\{1\}$ and $\Res{\Es}(A_2/A_1)=A_2/A_1$. For the former, $\Rad{\Es}(A_2/A_1)$ is characteristic in $A_2/A_1$ and, therefore, normal in $G/A_1$. Letting $\pi:G\rightarrow G/A_1$ be the usual projection, we see $\pi^{-1}(\Rad{\Es}(A_2/A_1))$ is a elementary normal subgroup of $G$. By choice of $A_1$, it must be the case $\pi^{-1}(\Rad{\Es}(A_2/A_1))=A_1$, so $\Rad{\Es}(A_2/A_1)=\{1\}$. \par

\indent  For $\Res{\Es}(A_2/A_1)$, put $H:=\Res{\Es}(A_2/A_1)$. Since $H$ is characteristic in $A_2/A_1$, we have that $H$ is normal in $G/A_1$. This gives a short exact sequence of topological groups: 
\[
1\rightarrow (A_2/A_1)/H\rightarrow (G/A_1)/H\rightarrow (G/A_1)/(A_2/A_1)\rightarrow 1.
\]
This sequence shows that $(G/A_1)/H$ is a group extension of elementary groups and, therefore, elementary. Since $A_2/A_1$ is minimal with respect to normal subgroups of $G/A_1$ whose quotient is elementary, $A_2/A_1=H$ as required.
\end{proof}

\indent Theorem~\rm\ref{thm:dcomp} now follows from either Theorem~\rm\ref{thm:des_E_series} or Theorem~\rm\ref{thm:acs_E_series}. We note these theorems give different decompositions. It is also worth noting $D_1/D_2$ and $A_2/A_1$ need not be compactly generated when $G$ is compactly generated. Examples~\rm\ref{ex:cg_dcomp} and \rm\ref{ex:d_series} demonstrate these phenomena. 

\begin{quest}[(Caprace)] 
Are $D_1/D_2$ and $A_2/A_1$ isomorphic? Is the obvious inclusion $D_1/D_2\rightarrow A_2/A_1$ necessarily an isomorphism of topological groups?
\end{quest}

\subsection{Examples}
\indent Fix a prime $p$ and $U\in \Uc(PSL_3(\Qp))$ and let $(G_i)_{i\in \Zb}$ list countably many copies of $PSL_3(\Qp)$ with $U_i$ the copy of $U$ in $G_i$. Putting $G:=\Tsumz{G_i}{U_i}$, one verifies $G$ has non-trivial proper normal subgroups, is elementary free, and is not compactly generated. It will be convenient to see the group $G$ as the set of functions $f:\Zb\rightarrow PSL_3(\Qp)$ so that $f(i)\in U$ for all but finitely $i$; cf. the discussion after Definition~\rm\ref{def:ldp}.

\begin{example}\label{ex:cg_dcomp}
We use $G$ to show the length of the characteristic series in Theorem~\rm\ref{thm:cg_dcomp} can be arbitrarily long. These examples are due to Pierre-Emmanuel Caprace; the author thanks Caprace for allowing these to be included in the present work work.\par

\indent There is a natural action of $\Zb$ on $G$  by shifting the coordinates; that is to say, $\Zb\acts G$ by $(n.f)(i):=f(i-n)$. We may thus form the t.d.l.c.s.c. group $L_1:=G\rtimes \Zb$ where $\Zb$ acts on $G$ by shift. Let $X$ be a compact generating set for $PSL_3(\Qp)$ and put $K:=\prod_{i<0}U_i\times X\times \prod_{i>0}U_i$. Taking $(1,t)$ a generator for $\Zb$ in $L_1$, it follows that $K\cup \{(1,t)\}$ is a compact generating set for $L_1$. We conclude that $L_1$ is a compactly generated t.d.l.c.s.c. group, so we may compute the decomposition given by Theorem~\rm\ref{thm:cg_dcomp}. One verifies that $\Rad{\Es}(L_1)$ is trivial and that $G$ is the unique minimal non-trivial normal subgroup of $L_1$. It then follows that the decomposition is
\[
\{1\}\sleq G\sleq L_1.
\]

\indent We now show how to produce a group with a longer decomposition. Fix $U\in \Uc(L_1)$. Since $L_1$ has no non-trivial compact normal subgroups, $L_1$ acts on the set of left cosets of $U$ continuously, transitively, and faithfully. By enumerating the collection of cosets by $\Zb$, $L_1\acts \Zb$ continuously, transitively, and faithfully by permutations. As in the discussion after Definition~\rm\ref{def:ldp}, there is a group action $L_1\acts G$ given by permuting the coordinates. We now consider the t.d.l.c.s.c. group $G\rtimes L_1=:L_2$. Similar to the previous example, $L_2$ is compactly generated with trivial elementary radical and with $G$ as the unique minimal non-trivial normal subgroup. The decomposition given by Theorem~\rm\ref{thm:cg_dcomp} is thus
\[
\{1\}\sleq G\sleq G\rtimes G \sleq L_2.
\]

\indent Continuing in this fashion, we produce compactly generated t.d.l.c.s.c. groups such that the series given by Theorem~\rm\ref{thm:cg_dcomp} is arbitrarily long. \par

\indent We remark that $L_1$ shows we may not assume the topologically characteristically simple quasi-factors are topologically simple or compactly generated. One verifies that $L_1$ additionally shows $D_1/D_2$ and $A_2/A_1$ need not be compactly generated for a compactly generated group.
\end{example}

\begin{example}\label{ex:d_series}
Again using $G$, we produce a t.d.l.c.s.c. group such that all inclusions of the descending elementary series are proper. Let $V:=\prod_{i\in \Zb}U_i\in \Uc(G)$ and let $\mc{D}$ be the countable collection of left cosets of $V$ in $G$. Certainly, $G\acts \mc{D}$ continuously and transitively by left multiplication; this action is also faithful since $G$ has no non-trivial compact normal subgroups. As in the previous example, $G\acts A_5^{\mc{D}}$ by shift, so we may form $H:=\left(A_5^{\mc{D}}\rtimes G\right) \times \mathbb{Z}$. The descending elementary series for $H$ is
\[
H\sgeq G\ltimes A_5^{\mc{D}}\sgeq A_5^{\mc{D}}\sgeq \{1\},
\]
and all of the inclusions in this series are proper.\par

\indent  We remark that $H$ also demonstrates that the descending and ascending elementary series may differ. The ascending elementary series of $H$ has $A_1=A_5^{\mc{D}}\times \mathbb{Z}$ and $A_2=H$.
\end{example}

\section{Application 2: Locally solvable groups}\label{sec:locsolv}
For our second application, we prove local-to-global results for the class of t.d.l.c.s.c. group with an open solvable subgroup. Such groups are called \textbf{locally solvable}. In the case the open solvable subgroup is nilpotent, we call these groups \textbf{locally nilpotent}.

\subsection{Preliminaries}

The \textbf{derived series} of a topological group $G$ is the sequence of closed normal subgroups $(G^{(k)})_{k\geqslant 0}$ defined by $G^{(0)}:=G$ and $G^{(n+1)}:=\ol{[G^{(n)},G^{(n)}]}$. We say $G$ is \textbf{solvable} if the derived series stabilizes at $\{1\}$ after finitely many steps. For solvable $G$, the \textbf{derived length}, $l(G)$, is the least $k$ such that $G^{(k)}=\{1\}$.\par

\indent The \textbf{lower central series} of a topological group $G$ is the sequence of closed normal subgroups $(G_k)_{k\geqslant 1}$ defined by $G_1:=G$ and $G_{n+1}:=\ol{[G,G_n]}$. We say $G$ is \textbf{nilpotent} if the lower central series stabilizes at $\{1\}$ after finitely many steps. The \textbf{nilpotence class} of nilpotent $G$, $n(G)$, is the least $k$ so that $G_k=\{1\}$. Observe that the last non-trivial term in the lower central series of a nilpotent topological group is central.\par

\subsection{The structure of locally solvable groups}
For a locally solvable t.d.l.c.s.c. group $G$, define 
\[
r_s(G):=\min\left\{l(U)\mid U\in \mathcal{U}(G)\text{ and is solvable}\right\}.
\]
We call $r_s(G)$ the \textbf{solvable rank} of $G$. A t.d.l.c.s.c. group is locally solvable of solvable rank zero if and only if it is a countable discrete group.

\begin{thm}\label{thm:solv_str}
If $G$ is a locally solvable t.d.l.c.s.c. group, then $G\in \Es$ with $\rk(G)\sleq 4^{r_s(G)}$. 
\end{thm}
\begin{proof}
We induct on $r_s(G)$ for the theorem. The base case case, $r_s(G)=0$, is immediate since such a group is discrete.\par

\indent Suppose $G$ is a locally solvable group with $r_s(G)=k+1$. Let $U\in \Uc(G)$ be solvable with derived length $k+1$ and put 
\[
\mc{A}_{k+1}:=\left\{C\in S(G)\mid \forall\; V\in \Uc(G)\hspace{1pt}\;\exists \text{ solvable } W\sleq_oV : C\sleq N_G\left(W^{(k)}\right)\right\}.
\]
Certainly, $\mc{A}_{k+1}$ is hereditary. To see conjugation invariance, take $C\in \mc{A}_{k+1}$, $V\in \mc{U}(G)$, and $g\in G$. Find $W\sleq_o g^{-1}V g$ such that $C\sleq N_{G}(W^{(k)})$. We thus have that
\[
gCg^{-1}\sleq N_G\left(gW^{(k)}g^{-1}\right)=N_G\left((gWg^{-1})^{(k)}\right)
\]
and $gWg^{-1}\sleq_o V$. It now follows that $gCg^{-1}\in \mc{A}_{k+1}$, whereby $\mc{A}_{k+1}$ is invariant under conjugation.\par

\indent By way of Proposition~\rm\ref{prop:syn_subgroup}, we may form the $\mc{A}_{k+1}$-core; denote this subgroup by $N_{k+1}$.

\begin{claim*}
$U^{(k)}\sleq N_{k+1}$.
\end{claim*}

\begin{proof}[of Claim] Take $u\in U^{(k)}$, $C\in \mc{A}_{k+1}$, and $V\in \mc{U}(G)$. It suffices to show there is $W\sleq_o V$ such that $\cgrp{u,C}\sleq N_G(W^{(k)})$. By definition of $\mc{A}_{k+1}$, there is $W\sleq_o U\cap V$ with $C\sleq N_G(W^{(k)})$. Since $W^{(k)}\sleq U^{(k)}$ and $U^{(k)}$ is abelian, $u\in N_G(W^{(k)})$. Therefore, $\cgrp{u,C}\sleq N_{G}(W^{(k)})$ as required. 
\end{proof}

\indent Put $N=\ol{N_{k+1}}$. By the claim, $r_s(G/N)\sleq k$, and the induction hypothesis implies $G/N\in \Es$ with $\rk(G/N)\sleq 4^k$. On the other hand, let $(n_i)_{i\in \omega}$ list a countable dense subset of $N_{k+1}$. Our group $N$ is an increasing union of $M_i:=\grp{ U\cap N, n_0,\dots,n_i}$. Moreover, by the choice of $n_0,\dots,n_i$, the group $M_i$ is a member of $\mc{A}_{k+1}$ for each $i$ since $U\in \mc{A}_{k+1}$. We may thus find $W\sleq_o U$ such that $M_i$ normalizes $W^{(k)}$. So $W^{(k)}\cap M_i\sleq \Rad{\mc{LE}}(M_i)$, and since $(W\cap M_i)^{(k)}\sleq W^{(k)}\cap M_i$, we infer that $r_s(M_i/\Rad{\mc{LE}}(M_i))\sleq k$. The induction hypothesis now implies $M_i/\Rad{\mc{LE}}(M_i)\in \Es$ with $\rk(M_i/\Rad{\mc{LE}}(M_i))\sleq 4^k$. The group $\Rad{\mc{LE}}(M_i)$ is elementary with rank at most $1$, whereby Proposition~\rm\ref{prop:grp_ext} gives that $M_i\in \Es$ with $\rk(M_i)\sleq 4^k+2$. It now follows that $N\in \Es$ with $\rk(N)\sleq 4^k+3$. \par

\indent A second application of Proposition~\rm\ref{prop:grp_ext} implies $G\in \Es$ with
\[
\rk(G)\sleq 4^k+4^k+4\sleq 4^{k+1}.
\]
This completes the induction, and we conclude the theorem.
\end{proof}

As a corollary, we recover a theorem of Willis.

\begin{cor}[(Willis, {\cite[Theorem 2.2]{Will07}})] 
If $G$ is a non-discrete, compactly generated, t.d.l.c. group that is topologically simple, then $G$ is not locally solvable.
\end{cor}
\begin{proof} 
Suppose for contradiction $G$ is locally solvable. By \cite[(8.7)]{HR79}, $G$ is second countable, and we conclude that $G\in \Es$ via Theorem~\rm\ref{thm:solv_str}. Proposition~\rm\ref{prop:ex_topsimple} then implies $G$ is discrete which contradicts our assumptions on $G$.
\end{proof}

The proof of Theorem~\rm\ref{thm:solv_str} gives a procedure for decomposing a locally solvable t.d.l.c.s.c. group into groups that are either profinite, discrete, or have smaller solvable rank. The proof also implies groups with a compact open abelian subgroup have a nice structure. We record these observations here:

\begin{prop}\label{prop:abelian_str}
If $G$ is a locally abelian t.d.l.c.s.c. group, then $\SIN(G)\trianglelefteq_o G$.
\end{prop}

\begin{thm}\label{thm:solv_str_2} 
Suppose a t.d.l.c.s.c. group $G$ is locally solvable and $r_s(G)\sgeq 1$. Then there is $1\sleq k\sleq r_s(G)$ and a sequence of closed characteristic subgroups 
\[
\{1\}=:N_{0}\sleq N_1\sleq \dots\sleq N_k\sleq G
\]
such that
\begin{enumerate}[(1)]
\item $r_s(G)>r_s(G/N_1)>\dots>r_s(G/N_k)=0$, and
\item for all $1\sleq j\sleq k$, $N_j/N_{j-1}=\bigcup_{i\in \omega}H_i$ with $(H_i)_{i\in \omega}$ an $\subseteq$-increasing sequence of compactly generated open subgroups of $N_j/N_{j-1}$ for which  $r_s(H_i/\Rad{\mc{LE}}(H_i))< r_s(G/N_{j-1})$ for all $i\in \omega$.
\end{enumerate}
\end{thm}

\subsection{The structure of locally nilpotent groups}
The \textbf{nilpotence rank} of a locally nilpotent t.d.l.c.s.c. group $G$ is defined to be 
\[ 
r_n(G):=\min\left\{n(U)\mid U\in \mathcal{U}(G)\text{ and is nilpotent}\right\}
\]
where $n(U)$ is the nilpotence class of $U$. A t.d.l.c.s.c. group has nilpotence rank $1$ if and only if it is discrete.

\begin{thm} 
Suppose $G$ is a locally nilpotent t.d.l.c.s.c. group. Then there is $1\sleq k \sleq r_n(G)$ and a sequence of closed characteristic subgroups 
\[
\{1\}=:N_0\sleq N_1\sleq\dots\sleq N_k\sleq_o G
\]
such that $N_i/N_{i-1}\simeq \SIN(G/N_{i-1})$ for $1\sleq i \sleq k$.
\end{thm}

\begin{proof}
We proceed by induction on $i$ to build the $N_i$. For the base case, put $N_1:=\SIN(G)$; recall $\SIN(G)$ is a closed characteristic subgroup. If $N_1=G$, then we stop. Else, suppose $N_1\neq G$ and let $U\in \Uc(G)$ be such that $n(U)=r_n(G)$. The group $G$ must be non-discrete, so $n(U/Z(U))<n(U)$. Since $Z(U)\sleq \SIN(G)$, we conclude that $r_n(G/N_1)<r_n(G)$. \par

\indent Suppose we have defined a closed characteristic $N_l$ for some $l\sgeq 1$. If $r_n(G/N_l)=1$, we stop. Else, let $\pi_l:G\rightarrow G/N_l$ be the usual projection and put $N_{l+1}:=\pi_l^{-1}(\SIN(G/N_l))$. We now have that $N_{l+1}$ is closed and characteristic, $N_{l+1}/N_l\simeq \SIN(G/N_l)$, and $ G/N_{l+1}\simeq(G/N_l)/\SIN(G/N_l)$. Letting $U\in \Uc(G/N_l)$ be such that $n(U)=r_n(G/N_l)$, we additionally see that the group $\SIN(G/N_l)$ contains $Z(U)$, so $r_n(G/N_{l+1})<r_n(G/N_l)$.\par

\indent Since the nilpotence rank of $G/N_l$ drops at each stage of our construction, our procedure halts at some $k\sleq n(U)$. At this $k$, it must be the case that $r_n(G/N_k)=1$, and therefore, $G/N_k$ is discrete. We conclude that
\[
N_1\sleq N_2\sleq\dots\sleq N_k\sleq_o G
\]
satisfies the theorem. 
\end{proof} 

\begin{rmk} 
The results in this section demonstrate that local assumptions affect global structure \emph{even in the setting of elementary groups}.
\end{rmk}

\subsection{Examples}
\indent We here present examples illustrating Proposition~\rm\ref{prop:abelian_str} and Theorem~\rm\ref{thm:solv_str_2}. 

\begin{example} For Proposition~\rm\ref{prop:abelian_str}, let $F:=\Zb/3\Zb$ and $(F_i)_{i\in \Zb}$ list countably many copies of $F$. Set $U_i=F$ for $i\sleq 0$ and $U_i=\{1\}$ for $i>0$ and form $\Tsumz{F_i}{U_i}$. Since $\Zb\acts \Tsumz{F_i}{U_i}$ by shift, we may form 
\[
G:=\Tsumz{F_i}{U_i}\rtimes \Zb;
\]
see the discussion after Definition~\rm\ref{def:ldp} for more details. The group $\prod_{i\sleq 0}F_i$ is a compact open subgroup of $G$, hence $G$ is locally abelian. One verifies that
\[
\Tsumz{F_i}{U_i}= \SIN(G).
\]
\end{example}

\begin{example}
For Theorem~\rm\ref{thm:solv_str_2}, let $P$ be the non-discrete topologically simple locally elliptic group with a compact open abelian subgroup built by Willis \cite[Proposition 3.4]{Will07}. Let $V\in \Uc(P)$ be abelian and define $\mc{D}$ to be the collection of left cosets of $V$ in $P$. The group $P$ acts on $\mc{D}$ by left multiplication, and this action is continuous, faithful, and transitive. Take $S$ to be the finitely generated infinite simple group built by Higman \cite{H51} and set $F:=\Zb/3\Zb$. We now form $S^{<\mc{D}}$ and $F^{S^{<\mc{D}}}$, and since $S^{<\mc{D}}$ acts on $F^{S^{<\mc{D}}}$ by shift, we may construct the t.d.l.c.s.c. group $F^{S^{<\mc{D}}}\rtimes S^{<\mc{D}}$. We may further construct
\[
H:=\left(F^{S^{<\mc{D}}}\rtimes S^{<\mc{D}}\right)\rtimes P
\]
where $P\acts (F^{S^{<\mc{D}}}\rtimes S^{<\mc{D}})$ via $p.(\alpha,f):=(p.\alpha,p.f)$ with $p.f$ the usual shift and $p.\alpha$ is the shift induced by the shift action of $P$ on $S^{<\mc{D}}$.\par

\indent The subgroup 
\[
U:=\left(F^{S^{<\mc{D}}}\times \{1\}\right)\times V
\]
is a compact open two step solvable subgroup of $H$. The group $H$ acts on the set of left cosets of $U$, denoted $\mc{F}$, transitively with $F^{S^{<\mc{D}}}$ in the kernel of the action. We take a second semi-direct product to obtain
\[
G:=S^{<\mc{F}}\rtimes H
\]
where $H\acts S^{<\mc{F}}$ by shift. Since $U$ remains a compact open subgroup of $G$, $G$ is locally solvable with $r_s(G)=2$.\par

\indent We now compute the decomposition given by Theorem~\rm\ref{thm:solv_str_2}. One verifies that $N_1=S^{<\mc{F}}\rtimes F^{S^{<\mc{D}}}$. Since $F^{S^{<\mc{D}}}$ acts trivially on $S^{<\mc{F}}$, this semi-direct product is indeed a direct product. Setting $H_n:=S^{n}\times F^{S^{<\mc{D}}}$, we have that $N_1=\bigcup_{n\sgeq 0}H_n$, and for each $n\sgeq 0$, $H_n/\Rad{\mc{LE}}(H_n)$ is discrete and thereby has solvable rank zero.\par

\indent On the other hand, $G/N_1=S^{<\mc{D}}\rtimes P$ and thus is locally abelian. Since $P$ is topologically simple, it follows that $N_2=\psi^{-1}(\SIN(G/N_1))=G$, hence $N_2/N_1=\bigcup_{n\sgeq 0} O_n$ with $(O_n)_{n\sgeq 0}$ an $\subseteq$-increasing sequence of compactly generated $SIN$ groups. The group $O_n/\Rad{\mc{LE}}(O_n)$ is discrete and so has solvable rank zero. We have now computed the decomposition for $G$ as given by Theorem~\rm\ref{thm:solv_str_2}:
\[
\{1\} \sleq  S^{<\mc{F}}\rtimes F^{S^{<\mc{D}}}\sleq N_2=G.
\]
\end{example}

\section{Application 3: $[A]$-regular groups}\label{sec:a-reg}

For our last application, we consider a somewhat technical local assumption, which plays an important role in the new, deep theory developed in \cite{CRW_1_13} and \cite{CRW_2_13}. Specifically, we consider the structure of t.d.l.c.s.c. groups that are $[A]$-regular.

\subsection{Preliminaries} These preliminary definitions and facts either come from or are implicit in \cite{CRW_1_13}; we include a discussion for completeness.

\begin{defn} Let $[A]$ be the smallest class of profinite groups such that the following conditions hold:
\begin{enumerate}[(i)]
\item $[A]$ contains all abelian profinite groups and all finite simple groups.
\item If $U\in [A]$ and $K\trianglelefteq U$, then $K\in [A]$ and $U/K\in [A]$.
\item If $U=V_1\dots V_n$ with $V_i\trianglelefteq U$ and $V_i\in [A]$ for each $1\sleq i \sleq n$, then $U\in [A]$. 
\end{enumerate}
\end{defn}
\noindent It is an easy application of Fitting's theorem to see $[A]$ consists of virtually nilpotent profinite groups. For a profinite group $U$, we let $[A](U)$ denote the closed subgroup generated by all normal subgroups of $U$ that are members of the class $[A]$. \par

\indent For $G$ a t.d.l.c. group, a compact $K\sleq G$ is \textbf{locally normal} if $N_G(K)$ is open. Via locally normal subgroups, we arrive at the central definition of this section.

\begin{defn} Let $U$ be a profinite group and $G$ be a t.d.l.c. group.
\begin{enumerate}[(i)]
\item For $K\sleq U$ a closed subgroup, $K$ is $[A]$\textbf{-regular} if for every $L\trianglelefteq U$ closed that does not contain $K$, the image of $K$ in $U/L$ contains a non-trivial locally normal $[A]$-subgroup of $U/L$. 
\item For $H\sleq G$ a closed subgroup, $H$ is $[A]$\textbf{-regular} if $U\cap H$ is $[A]$-regular in $U$ for all $U\in \Uc(G)$.
\item $G$ is $[A]$\textbf{-semisimple} if $G$ has trivial quasi-centre and the only locally normal subgroup of $G$ belonging to $[A]$ is $\{1\}$
\end{enumerate}
\end{defn}

\indent In particular, a t.d.l.c.s.c. group $G$ is $[A]$-regular, if for every $U\in \Uc(G)$ and every non-trivial quotient group $U/L$, it is the case that $U/L$ contains a non-trivial locally normal $[A]$-subgroup. As a corollary, if $G$ is $[A]$-regular, then $G/N$ is $[A]$-regular for any closed normal subgroup $N$ of $G$.\par

\indent We now recall a few  basic lemmas regarding locally normal $[A]$-subgroups and a deep result of Caprace, Reid, and Willis, which elucidates the connection between $[A]$-regularity and $[A]$-semisimplicity.\par

\begin{lem}\label{lem:inf_norm_A} 
If a profinite group $U$ contains an infinite locally normal $[A]$-subgroup, then $U$ contains an infinite normal $[A]$-subgroup.
\end{lem}

\begin{proof}
Suppose $L\sleq U$ is an infinite locally normal $[A]$-subgroup and find $W\trianglelefteq_oU$ such that $W\sleq N_U(L)$. Since $L\cap W\trianglelefteq L$, $L\cap W$ is an $[A]$-subgroup, and since $L$ is infinite and $W$ is open, $L\cap W$ is also infinite.\par

\indent Take $u_1,\dots,u_n$ left coset representatives for $W$ in $U$. Since $L\cap W$ is an $[A]$-subgroup and $u_i(L\cap W)u_i^{-1}\trianglelefteq W$ for each $i$, we have that
\[
K:=u_1(L\cap W)u_1^{-1}\dots u_n(L\cap W)u_n^{-1}
\]
is an $[A]$-subgroup. Further, $K\trianglelefteq U$ giving an infinite $[A]$-subgroup that is normal in $U$.
\end{proof}

\begin{prop}\label{prop:A_comm}
If $G$ is a t.d.l.c. group, then $[A](U)\sim_c [A](V)$ for all $U,V\in \Uc(G)$.
\end{prop}

\begin{proof} Fix $U\in \Uc(G)$. We first show $[A](W)\sim_c[A](U)$ for $W\trianglelefteq_oU$. Let $A\trianglelefteq W$ be an $[A]$-subgroup of $W$ and fix $u_1,\dots,u_n$ left coset representatives for $W$ in $U$. As in the previous proof, $L:=u_1Au_1^{-1}\dots u_nAu_n^{-1}$ is a normal $[A]$-subgroup in $U$, so $A\sleq L\sleq[A](U)$. We conclude that $[A](W)\sleq [A](U)$.\par

\indent On the other hand, take $a\in [A](U)\cap W$. It suffices to assume $a\in A\trianglelefteq U$ with $A$ some $[A]$-subgroup since $W$ is open. Since $W\cap A\trianglelefteq A$, the group $W\cap A$ is a normal $[A]$-subgroup of $W$, hence $a\in W\cap A\sleq [A](W)$. We now have that $[A](U)\cap W\sleq[A](W)\sleq [A](U)$, and therefore, $[A](U)\sim_c[A](W)$. \par

\indent For an arbitrary $V\in \Uc(G)$, we may find $W\trianglelefteq_o V$ such that $W\trianglelefteq_o U\cap V$. By our argument above, 
\[
[A](V)\sim_c [A](W)\sim_c[A](U\cap V).
\]
On the other hand, the same argument implies $[A](U)\sim_c[A](U\cap V)$. Therefore, $[A](U)\sim_c[A](V)$ by the transitivity of $\sim_c$.
\end{proof}

\begin{thm}[(Caprace, Reid, Willis {\cite[Theorem 6.10]{CRW_1_13}})]\label{thm:A_reg_rad} If $G$ is a t.d.l.c. group, then $G$ has a closed normal subgroup $R_{[A]}(G)$, the $[A]$\textbf{-regular radical} of $G$, such that 
\begin{enumerate}[(1)]
\item $R_{[A]}(G)$ is the unique largest closed normal subgroup of $G$ that is $[A]$-regular.
\item $G/R_{[A]}(G)$ is $[A]$-semisimple, and given any closed normal subgroup $N$ of $G$ such that $G/N$ is $[A]$-semisimple, $R_{[A]}(G)\sleq N$. 
\end{enumerate}
\end{thm}

\subsection{Quasi-centralizers and commensurators}

We now prove a series of technical lemmas which are used in the next section to show $\qc{G}{U/[A](U)}$ is elementary in a t.d.l.c.s.c. group $G$ with a trivial quasi-centre.

\begin{lem}\label{lem:quasi_center_quot}
Suppose $U$ is a profinite group, $B\trianglelefteq U$, and $U/B$ has a dense quasi-centre. If $A\trianglelefteq U$ and $A\sleq_oB$, then $U/A$ has a dense quasi-centre.
\end{lem}

\begin{proof} Fix $\Sigma:=(g_i)_{i\in I}\subseteq U$ coset representatives for the elements of $QZ(U/B)$ in $U/B$. Consider the set of elements of $U/A$ with the form $g_ibA$ for $i\in I$ and $b\in B$. By the choice of the $g_i$, for each $i\in I$ there is $W_i\sleq_oU$ such that $[g_i,W_i]\subseteq B$. In fact, more is true: for any $b\in B$ and $w\in W_i$, we have that $[g_ib,w]=[b,w]^{g_i}[g_i,w]\in B$.\par

\indent Now the map $\xi:W_i\rightarrow U/A$ by $w\mapsto[g_ib,w]A$ is continuous and by the above, has image in $B/A$. Since $1\in im(\xi)$ and $B/A$ is finite, there is $V_i\sleq_o W_i$ such that $\xi(V_i)=1$. It follows that $C_{U/A}(g_ibA)$ is open for each $i\in I$ and $b\in B$. Therefore, $\{g_ibA\mid b\in B\text{ and }i\in I\}\subseteq QZ(U/A)$.\par

\indent On the other hand, the map $\psi: U/A\rightarrow U/B$ defined by $uA\mapsto uB$ is open and continuous. It follows the set $\{g_ibA\mid b\in B\text{ and }i\in I\}$ is additionally dense in $U/A$. Thus, $U/A$ has a dense quasi-centre.
\end{proof}

\begin{lem}\label{lem:vnil_E} 
Suppose $G$ is a t.d.l.c.s.c. group. Suppose further there is $U\in \Uc(G)$ with a virtually nilpotent $B\trianglelefteq U$ such that $Comm_G(B)=G$ and $QZ(U/B)$ is dense in $U/B$. Then $G$ is elementary. 
\end{lem}
\begin{proof}
Given a triple $(G,U,B)$ as hypothesized, we may find $W\trianglelefteq_o U$ such that $A:=W\cap B$ is nilpotent. Plainly, $A\trianglelefteq U$, $A$ is nilpotent, and $Comm_G(A)=G$. We induct on the minimal nilpotence class of such an $A$ for the lemma. For the base case, $n(A)=1$, the group $A$ is trivial. Lemma~\rm\ref{lem:quasi_center_quot} thus implies $U=U/A$ has a dense quasi-centre. Since $QZ(G)\sleq \SIN(G)$, we conclude that $\SIN(G)$ is open in $G$, and it follows that $G\in \Es$.\par 

\indent Suppose the lemma holds for all triples $(G,U,B)$ as hypothesized such that there is a nilpotent $A\sleq_oB$ with $A\trianglelefteq U$ and $n(A)\sleq k$. Suppose the triple $(G,U,B)$ is as hypothesized, but every $A\sleq_oB$ with $A\trianglelefteq U$ and $A$ nilpotent has $n(A)= k+1$. Fix such an $A$. \par

\indent Define $\mc{H}\subseteq S(G)$ by $C\in \mc{H}$ if and only if for all $V\in \Uc(G)$, there exists $N_W\trianglelefteq W\sleq_o V$ and $(W_i)_{i\in \omega}$ a normal basis at $1$ for $W$ such that
\begin{enumerate}[(i)]
\item $N_W\sim_c A$, and
\item $C\sleq \bigcap_{i\in \omega} N_G(N_W\cap W_i)$.
\end{enumerate}
It is immediate that $\mc{H}$ is hereditary and that $U\in \mc{H}$.

\begin{claim*} 
$\mc{H}$ is conjugation invariant.
\end{claim*}
\begin{proof}[of Claim] 
Take $C\in \mc{H}$ and $g\in G$. Fix $V\in \mc{U}(G)$ and find $W\sleq_o g^{-1}Vg$ as given by the definition of $C\in\mc{H}$. So there is $(W_i)_{i\in \omega}$ a normal basis at $1$ for $W$ and $N_W\trianglelefteq W$ with $N_W\sim_c A$ such that $C\sleq \bigcap_{i\in\omega}N_G(N_W\cap W_i)$. \par

\indent We claim $gWg^{-1}\sleq_o V$ satisfies the conditions for $gCg^{-1}\in \mc{H}$. Certainly, $gN_Wg^{-1}\trianglelefteq gWg^{-1}$,  $(gW_ig^{-1})_{i\in \omega}$ is a normal basis at $1$ for $gWg^{-1}$, and 
\[
gCg^{-1}\sleq \bigcap_{i\in \omega} N_G(gN_Wg^{-1}\cap gW_ig^{-1}).
\]
Since $gN_Wg^{-1}\sim_c gAg^{-1}$ and $gAg^{-1}\sim_c A$, it is additionally the case that $gN_Wg^{-1}\sim_c A$. Therefore, $\mc{H}$ is conjugation invariant.
\end{proof}

\indent  We may now form $N_{\mc{H}}$, the $\mc{H}$-core.

\begin{claim*}
$Z(A)\sleq N_{\mc{H}}$.
\end{claim*}

\begin{proof}[of Claim] Take $g\in Z(A)$ and $C\in \mc{H}$. Fix $V\in \mc{U}(G)$ and find $W\sleq_o V$ as given by the definition of $C\in \mc{H}$. So there is $(W_i)_{i\in \omega}$ a normal basis at $1$ for $W$ and $N_W\trianglelefteq W$ with $N_W\sim_c A$ and $C\sleq \bigcap_{i\in\omega}N_G(N_W\cap W_i)$. Since $N_W\sim_c A$, there is $i\in \omega$ such that $N_W\cap W_i\sleq A\cap N_W$. Fix such an $i$ and put $N_{W_i}:=N_W\cap W_i$.\par

\indent We now have that $N_{W_i}\trianglelefteq W_i$, $N_{W_i}\sim_cA$, and $(W_j)_{j\sgeq i}$ is a normal basis at $1$ for $W_i$. Since $N_{W_i}\sleq A$, it is further the case that
\[
\cgrp{C,g}\sleq \bigcap_{j\sgeq i}N_G(N_{W_i}\cap W_j).
\]
We conclude that $\cgrp{C,g}\in \mc{H}$ and therefore, that $g\in N_{\mc{H}}$. 
\end{proof} 

\indent Put $N:=\ol{N_{\mc{H}}}$ and let $\pi:G\rightarrow G/N$ be the usual projection. Since $Comm_{G/N}(\pi(B))=G/N$, the triple $(G/N,\pi(U),\pi(B))$ satisfies the hypotheses of the lemma. Moreover, $\pi(A)\sleq_o\pi(B)$, $\pi(A)\trianglelefteq \pi(U)$, and $n(\pi(A))\sleq k$. The induction hypothesis thus implies $G/N$ is elementary. 

\indent On the other hand, let $(n_i)_{i\in \omega}$ list a countable dense set of $ N_{\mc{H}}$ and define 
\[
P_i:=\grp{U,n_0,\dots,n_i}.
\]
Since $U\in \mc{H}$ and $n_0,\dots,n_i\in N_{\mc{H}}$, we have that $P_i\in \mc{H}$ for each $i$. Following from the definition of $\mc{H}$, we may find $W\sleq_oU$ and $L\trianglelefteq W$ such that $L\sleq_oA$ and $P_i\sleq N_G(L)$. \par

\indent Lemma~\rm\ref{lem:quasi_center_quot} implies $U/A$ has a dense quasi-centre, and since $W/(A\cap W)\sleq_o U/A$, the group $W/A\cap W$ has a dense quasi-centre. Applying Lemma~\rm\ref{lem:quasi_center_quot} to $W$ gives that $W/L$ also has a dense quasi-centre because $L\sleq_o A\cap W$ and $L\trianglelefteq W$. Letting $\pi:P_i\rightarrow P_i/L$ be the usual projection, we infer that
\[
\pi(W)\sleq \ol{QZ(P_i/L)}\sleq \SIN(P_i/L),
\] 
and therefore, $\SIN(P_i/L)\trianglelefteq_oP_i/L$. It now follows that $P_i$ is elementary, and since $N\sleq \bigcup_{i\in \omega}P_i$ is closed, Theorem~\rm\ref{thm:closure_main} implies $N$ is also elementary.\par

\indent The class $\Es$ is closed under group extension, so we conclude that $G\in \Es$ completing the induction.
\end{proof}

\begin{lem}\label{lem:qc_E} 
Suppose $G$ is a t.d.l.c.s.c. group and $U\in \Uc(G)$ contains a non-trivial virtually nilpotent $B\trianglelefteq U$ such that $Comm_G(B)=G$. Then $\qc{G}{U/B}$ is an elementary normal subgroup of $G$ containing $B$.
\end{lem}

\begin{proof}
Put 
\[
N:=\qc{G}{U/B}:=\ol{\left\{g\in G\mid \exists\;W\in \Uc(G)\text{ such that } [g,W\cap  U]\subseteq B\right\}}.
\]
Proposition~\rm\ref{prop:quasicentralizer} implies $N$ is a normal subgroup containing $B$ since $B$ and $U$ are commensurated and $B\trianglelefteq U$. Furthermore, $B\trianglelefteq V:=U\cap N\in \Uc(N)$, $QZ(V/B)$ is dense in $V/B$, and $Comm_N(B)=N$. Lemma~\rm\ref{lem:vnil_E} therefore implies $N$ is also elementary. 
\end{proof}

We now present a slight adaptation of a lemma due to Caprace, Reid, and Willis from their work \cite{CRW_2_13}. \par

\indent Let $G$ be a t.d.l.c. group, $U\in \Uc(G)$, and $K\trianglelefteq U$ be infinite. Take $g\in G$ and consider $K^g:=gKg^{-1}$. Certainly, $K^g\trianglelefteq U^g$, and since $U$ is compact and open, there are $u_1,\dots,u_n$ in $U$ such that $U\subseteq \bigcup_{i=1}^n u_iU^g$. For all $u\in U$, there is then some $1\sleq i \sleq n$ such that $uK^gu^{-1}=u_iK^gu_i^{-1}$. Putting
\[
(K^g)^U:=\left\{u_1K^gu_1^{-1},\dots,u_nK^gu_n^{-1}\right\},
\]
the set of subgroups $(K^g)^U$ is permuted by $U$ under the action by conjugation.\par

\indent Take $g_1,\dots,g_n\in G$ and let $K^{h_1},\dots,K^{h_m}$ list $\{K\}\cup\bigcup_{i=1}^n(K^{g_i})^U$. We now consider the group 
\[
H:=\cgrp{K^{h_1},\dots,K^{h_m}}.
\]
Since $U$ permutes the generating set, we immediately see that $U\sleq N_G(H)$. Let us make a few further observations: Since $N_{G}(K^{h_j})$ is open in $G$ for each $1\sleq j\sleq m$, the subgroup $V:=\bigcap_{j=1}^mN_{H\cap U}(K^{h_j})$ is a compact open subgroup of $H$. We additionally have that $V\cap K^{h_j}\trianglelefteq V$ for each $1\sleq j \sleq m$, hence
\[
L:=(V\cap K^{h_1})\dots(V\cap K ^{h_m})
\]
is an infinite normal subgroup of $V$.\par

\indent We claim that $L$ is also commensurated in $H$. Since $L\sleq N_H(K^{h_j})$, we have that $LK^{h_j}$ is a compact subgroup, and furthermore, $LK^{h_j}/L$ is finite as $(K^{h_j}\cap V)\sleq L$ and $LK^{h_j}/L\leftrightarrow K^{h_j}/L\cap K^{h_j}$. The group $L$ is a thus a finite index subgroup of $LK^{h_j}$, and it follows that $K^{h_j}\sleq Comm_H(L)$. As $V\sleq Comm_H(L)$, we conclude that $Comm_H(L)$ is a dense open subgroup of $H$ and therefore equals $H$. That is to say, $L$ is commensurated in $H$.\par

\indent We have now demonstrated the following lemma:

\begin{lem}[(Caprace, Reid, Willis \cite{CRW_2_13})]\label{lem:CRW}
Suppose $G$ is a t.d.l.c. group, $U\in \Uc(G)$, $K\trianglelefteq U$ is infinite, and $g_1,\dots,g_n\in G$. Let $K^{h_1},\dots,K^{h_m}$ list $\{K\}\cup\bigcup_{i=1}^n(K^{g_i})^U$ and put 
\[
H:=\cgrp{K^{h_1},\dots,K^{h_m}}.
\]
Then,
\begin{enumerate}[(1)]
\item $U\sleq N_G(H)$, and
\item for $V:=\bigcap_{j=1}^mN_{H\cap U}(K^{h_j})$ and $L:=(V\cap K^{h_1})\dots(V\cap K ^{h_m})$, it is the case that $L\trianglelefteq V\in \Uc(H)$, $L$ is infinite, and $Comm_H(L)=H$. 
\end{enumerate}
\end{lem}

\subsection{Structure theorems}

\begin{lem}\label{lem:qc_A_E}
Suppose $G$ is a t.d.l.c.s.c. group with trivial quasi-centre. If $G$ has a non-trivial locally normal $[A]$-subgroup and $U\in \Uc(G)$, then $\qc{G}{U/[A](U)}$ is a non-trivial elementary normal subgroup.
\end{lem}

\begin{proof}
Since $QZ(G)$ is trivial, the non-trivial locally normal $[A]$-subgroup that $G$ is hypothesized to have must be infinite. It follows that $U$ has an infinite locally normal $[A]$-subgroup, and from Lemma~\rm\ref{lem:inf_norm_A}, we infer that $U$ has an infinite normal $[A]$-subgroup. Thus, $[A](U)$ is non-trivial. We now form 
\[
N:=\qc{G}{U/[A](U)}:=\ol{\left\{g\in G\mid \exists\;W\in \Uc(G)\text{ such that } [g,W\cap  U]\subseteq [A](U)\right\}}.
\]
By Proposition~\rm\ref{prop:A_comm}, $[A](U)$ is commensurated, hence Proposition~\rm\ref{prop:quasicentralizer} implies $N$ is a normal subgroup that contains $[A](U)$. It remains to show $N$ is elementary.\par

\indent Let $\{A_{j}\}_{j\in J}$ list all normal $[A]$-subgroups of $U$ and let $(W_i)_{i\in \omega}$ be a normal basis at $1$ for $[A](U)$. For each $i$, we may find $\Omega_i:=\{A_{j_1},\dots,A_{j_n}\}\subseteq \{A_{j}\}_{j\in J}$ so that $\bigcup \Omega_i$ contains coset representatives for $[A](U)/W_i$. Put $\Omega:=\bigcup_{i\in \omega}\Omega_i$, list $\Omega$ as $(B_i)_{i\in \omega}$, and for each $i\in \omega$, define $C_i:=B_0B_1\dots B_i$. The sequence $(C_i)_{i\in \omega}$ is thus an $\subseteq$-increasing sequence of infinite closed normal virtually nilpotent subgroups of $U$ with $\bigcup_{i\in \omega}C_i$ dense in $[A](U)$. \par

\indent Fix $\Sigma=\{g_i\mid i\in\omega\}$ a countable dense \emph{subgroup} of $G$. For each $i\in \omega$, let $C_i^{h_1},\dots,C_i^{h_{m(i)}}$ list $C_i,(C_i^{g_0})^U,\dots,(C_i^{g_i})^U$ and define
\[
P_i:=\cgrp{C_i^{h_1},\dots,C_i^{h_{m(i)}}}.
\]
The construction of $P_i$ allows us to apply Lemma~\rm\ref{lem:CRW} giving $V\in \Uc(P_i)$ and 
\[
(V\cap C_i^{h_1})\dots(V\cap C_i ^{h_{m(i)}})=L\trianglelefteq V
\]
with $Comm_{P_i}(L)=P_i$. Each $C_i$ is virtually nilpotent, so $L$ is also virtually nilpotent. Setting $N_i:=\qc{P_i}{V/L}$, Lemma~\rm\ref{lem:qc_E} implies $N_i$ is an elementary normal subgroup of $P_i$.\par

\indent On the other hand, let $\pi:P_i\rightarrow P_i/N_i=:\tilde{P_i}$ be the usual projection. We have that $N_{P_i}(C_i^{h_j})$ is open in $P_i$ and thus, that $N_{\tilde{P_i}}(\pi(C_i^{h_j}))$ is open in $\tilde{P_i}$ for each $1\sleq j\sleq m(i)$. However, $\pi(C_i^{h_j})$ is finite as $L\cap C^{h_j}_i$ is finite index in $C_i^{h_j}$, and this implies
\[
\pi(C_i^{h_j})\sleq QZ(\tilde{P}_i)\sleq \SIN(\tilde{P_i}).
\]
for each $j$. It follows that $\SIN(\tilde{P}_i)=\tilde{P}_i$, so $\tilde{P}_i$ is elementary. Since $\Es$ is closed under group extension, we conclude that $P_i$ is elementary.\par

\indent The $P_i$ form an $\subseteq$-increasing sequence of elementary groups with $U\sleq N_G(P_i)$ for each $i$. Lemma~\rm\ref{lem:basic_E_lemmas} thus implies $P:=\ol{\bigcup_{i\in \omega}P_i}$ is also elementary. 

\begin{claim*}
$P\trianglelefteq G$. 
\end{claim*}

\begin{proof}[of Claim]It suffices to show the countable dense subgroup $\Sigma$ of $G$ used in the definition of $P$ normalizes $P$. Fixing $P_j$ from the construction of $P$ and  $g\in \Sigma$, it is indeed enough to show $gP_jg^{-1}\subseteq P$. \par

\indent Take $C_j^{ug_k}$ one of the generating subgroups from the construction of $P_j$. We may assume $u\in\Sigma$ since the normalizer of $C_j^{g_k}$ is open and $\Sigma$ is dense. So $gug_k\in \Sigma$, and there is some sufficiently large $l$ such that $C_j\sleq C_l$ and $gug_k$ is in the first $l$ elements of the enumeration of $\Sigma$. Thus, $C_l^{gug_k}\sleq P_l$. Since $P_j$ is topologically generated by finitely many groups of the form $C_j^{ug_k}$, it follows that $gP_jg^{-1}\sleq P_{l'}\sleq P$ for some large $l'$ proving the claim.
\end{proof}

By the claim and the construction of $P$, we now have that $[A](U)\sleq P\trianglelefteq N$. Let $\pi:N\rightarrow N/P$ be the usual projection and put $W:=N\cap U$. The group $\pi(W)$ is a quotient of $W/[A](U)$ and, therefore, has a dense quasi-centre. It follows that $\SIN(N/P)$ is open in $N/P$, so $N/P\in \Es$. Since $\Es$ is closed under group extension, we conclude that $N=\qc{G}{U/[A](U)}$ is elementary.
\end{proof}

We now prove the main theorem of this section.

\begin{thm}
If $G$ is an $[A]$-regular t.d.l.c.s.c. group, then $G$ is elementary.
\end{thm}

\begin{proof}
Suppose $G$ is an $[A]$-regular group, so $G/\Rad{\Es}(G)$ is also $[A]$-regular. Suppose for contradiction $G/\Rad{\Es}(G)$ is non-trivial. We thus have that $G/\Rad{\Es}(G)$ is a non-discrete t.d.l.c.s.c. group with a trivial quasi-centre and a non-trivial locally normal $[A]$-subgroup. Fixing $U\in \Uc(G/\Rad{\Es}(G))$, we apply Lemma~\rm\ref{lem:qc_A_E} to conclude that $\qc{G/\Rad{\Es}(G)}{U/[A](U)}$ is a non-trivial elementary normal subgroup of $G/\Rad{\Es}(G)$. This is absurd as $G/\Rad{\Es}(G)$ has a trivial elementary radical. 
\end{proof} 

\begin{cor}\label{cor:str_A}
If $G$ is a t.d.l.c.s.c. group, then $G/\Rad{\Es}(G)$ is $[A]$-semisimple. In particular, if $G$ is an elementary-free t.d.l.c.s.c. group, then $G$ is $[A]$-semisimple.
\end{cor}
\begin{proof} Certainly, $G/\Rad{\Es}(G)$ must have a trivial quasi-centre. Since $G/\Rad{\Es}(G)$ has no non-trivial elementary normal subgroups, Lemma~\rm\ref{lem:qc_A_E} implies $G/\Rad{\Es}(G)$ has no non-trivial abelian locally normal subgroups. Therefore, $G/\Rad{\Es}(G)$ is $[A]$-semisimple.
\end{proof}

\begin{cor}\label{Arcons}
If $G$ is a t.d.l.c.s.c. group, then $R_{[A]}(G)$ is elementary. 
\end{cor}

\begin{proof} 
By Corollary~\rm\ref{cor:str_A}, $G/\Rad{\Es}(G)$ is $[A]$-semisimple, and Theorem~\rm\ref{thm:A_reg_rad} therefore implies $R_{[A]}(G)\sleq \Rad{\Es}(G)$. Applying Theorem~\rm\ref{thm:closure_main}, we conclude that $R_{[A]}(G)$ is elementary.
\end{proof}

We lastly recover a theorem from the literature.
\begin{cor}[(Caprace, Reid, Willis \cite{CRW_2_13})] A non-discrete compactly generated topologically simple t.d.l.c. group is $[A]$-semisimple.
\end{cor}
\begin{proof}
Via \cite[(8.7)]{HR79}, such a group $G$ is second countable, so Proposition~\rm\ref{prop:ex_topsimple} implies $G$ is non-elementary. We conclude that $G$ is elementary-free and via Corollary~\rm\ref{cor:str_A}, is $[A]$-semisimple.
\end{proof}

\begin{rmk} 
In general, $R_{[A]}(G)\lneq \Rad{\Es}(G)$.
\end{rmk}

%=========================== The bibliography===================================

\bibliographystyle{bibgen}
\bibliography{C:/Users/Phillip/Dropbox/Research/TeX/Bibtex/biblio}

\affiliationone{
   Phillip Wesolek\\
   MSCS University of Illinois at Chicago\\
   322 Science and Engineering Offices MC249\\
   851 S. Morgan St. Chicago, Il 60607-7045\\
   USA}
\affiliationtwo{~} %inserts a space to make this field empty
\affiliationthree{
   Current address:\\
   Universit\'{e} catholique de Louvain IRMP\\ 
   Chemin du Cyclotron 2, box L7.01.02\\
   1348 Louvain-la-Neuve \\
   Belgique \\
   \email{phillip.wesolek@uclouvain.be}}
\end{document}